\documentclass[10pt]{article}\synctex=1 
\usepackage{amsfonts}
\usepackage{amssymb}
\usepackage{amsmath}
\usepackage{amsthm}
\usepackage{geometry}
\usepackage{commath}
\usepackage{booktabs}
\usepackage{lscape}
\usepackage{color}
\usepackage{colortbl}
\usepackage{xr}
\usepackage[authoryear]{natbib}

\newcommand{\myref}[2]{\hyperref[#1]{#2}}
\numberwithin{equation}{section}

\usepackage{latexsym}
\usepackage{graphicx} 
\usepackage{enumerate}
\usepackage{setspace}
\usepackage[toc,title,titletoc,header]{appendix}
\usepackage[bottom]{footmisc}   
\usepackage[pdfborder={0 0 0}]{hyperref}
\definecolor{fede_color}{cmyk}{0,0.9,0.33,.22}

\usepackage{lineno}
\setpagewiselinenumbers

\newtheorem{theorem}{Theorem}[section]
\newtheorem{lemma}{Lemma}[section]

\theoremstyle{definition}

\theoremstyle{remark}

\newtheorem{algorithm}{Algorithm}[section]
\newcounter{assumptionM}
\newcounter{assumptionA}
\def\theassumptionM{M.\arabic{assumptionM}}
\def\theassumptionA{A.\arabic{assumptionA}}
\newenvironment{assumptionA}[1][]{\refstepcounter{assumptionA}\medskip\noindent{\textbf{Assumption \theassumptionA. #1}}}{\medskip}

\newcommand{\sgn}{\mathrm{sign}}

\hypersetup{
  colorlinks=true,linkcolor=blue, citecolor=blue, filecolor=blue, 
  linkbordercolor={0 1 1}, citebordercolor={1 0 0},
  pdftitle={Inference in partially identified models with many moment inequalities using Lasso},
  pdfauthor={Bugni, Caner, Kock, and Lahiri},
  pdfcreator={\LaTeX\ with package \flqq hyperref\frqq},
  pdfsubject={},
}

\defcitealias{chernozhukov/chetverikov/kato:2014c}{CCK18}

\begin{document}
\sloppy

\title{
Inference in partially identified models with many moment inequalities using Lasso\thanks{We thank the editor and two anonymous referees for comments and suggestions that have greatly improved this manuscript. We have also benefited from comments and suggestions from the participants of the 2015 World Congress in Montreal, the Second International Workshop in Financial Econometrics, and seminars at University of Maryland, Yale University, and McGill University. Of course, any and all errors are our own. Bugni acknowledges support by NIH Grant 40-4153-00-0-85-399 and NSF Grant SES-1729280. Bredahl Kock acknowledges support from CREATES - Center for Research in Econometric Analysis of Time Series (DNRF78), funded by the Danish National Research Foundation. Lahiri acknowledges support from National Science Foundation under grant no.\ DMS 130068.}
}

\author{
Federico A. Bugni\\
Department of Economics\\
Duke University
\and
Mehmet Caner\\
Department of Economics\\
North Carolina State University
\and
Anders Bredahl Kock\\
Department of Economics \& Business\\
University of Oxford
\and
{Soumendra Lahiri}\\
Department of Statistics\\
North Carolina State University
}

\maketitle

\begin{abstract}
	This paper considers inference in a partially identified moment (in)equality model with many moment inequalities. We propose a novel two-step inference procedure that combines the methods proposed by \cite{chernozhukov/chetverikov/kato:2014c} (\citetalias{chernozhukov/chetverikov/kato:2014c}, hereafter) with a first step moment inequality selection based on the Lasso. Our method controls asymptotic size uniformly, both in underlying parameter and data distribution. Also, the power of our method compares favorably with that of the corresponding two-step method in \citetalias{chernozhukov/chetverikov/kato:2014c} for large parts of the parameter space, both in theory and in simulations. Finally, we show that our Lasso-based first step can be implemented by thresholding standardized sample averages, and so it is straightforward to implement.

 \vspace{0.3in}


\end{abstract}



\newpage
\section{Introduction}

This paper contributes to the growing literature on inference in partially identified econometric models defined by many unconditional moment (in)equalities, i.e.,\ inequalities and equalities. Consider an economic model with a parameter $\theta $ belonging to a parameter space $\Theta $, whose main prediction is that the true value of $\theta $, denoted by $ \theta _{0}$, satisfies a collection of moment (in)equalities. This model is partially identified, i.e., the restrictions of the model do not necessarily restrict $\theta _{0}$ to a single value, but rather they constrain it to belong to a certain set, called the identified set. The literature on partially identified models discusses several examples of economic models that satisfy this structure, such as selection problems, missing data, or multiplicity of equilibria (see, e.g., \cite{manski:1995} and \cite{tamer:2003}).

The first contributions in the literature of partially identified moment (in)equalities focus on the case in which there is a fixed and finite number of moment (in)equalities, both unconditionally\footnote{These include \cite{CHT07}, \cite{andrews/berry/jiabarwick:2004}, \cite{imbens/manski:2004}, \citet{galichon/henry:2006,galichon/henry:2013}, \cite{beresteanu/molinari:2008}, \cite{romano/shaikh:2008}, \cite{rosen:2008}, \cite{andrews/guggenberger:2009}, \cite{stoye:2009}, \cite{andrews/soares:2010}, \citet{bugni:2010,bugni:2015}, \cite{canay:2010}, \cite{romano/shaikh:2010}, \cite{andrews/jiabarwick:2012}, \cite{bontemps/magnac/maurin:2012}, \cite{bugni/canay/guggenberger:2012}, \cite{romano/shaikh/wolf:2014}, and \cite{pakes/porter/ho/ishii:2015}, among others.} and conditionally\footnote{These include \cite{kim:2008}, \cite{ponomareva:2010}, \citet{armstrong:2014,armstrong:2015}, \cite{chetverikov:2013}, \cite{andrews/shi:2013}, and \cite{chernozhukov/lee/rosen:2013}, among others.}. In practice, however, there are many relevant econometric models that produce a large set of moment conditions (even infinitely many). As the literature shows (e.g.\ \citet{menzel:2009,menzel:2014}), the associated inference problems cannot be properly addressed by an asymptotic framework with a fixed number of moment (in)equalities.\footnote{As pointed out by \cite{chernozhukov/chetverikov/kato:2014c}, this is true even for conditional moment (in)equality models (which typically produce an infinite number of unconditional moment (in)equalities). As they explain, the unconditional moment (in)equalities generated by conditional moment (in)equality models inherit the structure from the conditional moment conditions, which limits the underlying econometric model.} To address this issue, \cite{chernozhukov/chetverikov/kato:2014c} (hereafter referred to as \citetalias{chernozhukov/chetverikov/kato:2014c}) obtain inference results in a partially identified model with \textit{many} moment inequalities.\footnote{See also the related technical contributions in \citet{chernozhukov/chetverikov/kato:2013a,chernozhukov/chetverikov/kato:2013b,chernozhukov/chetverikov/kato:2014a,chernozhukov/chetverikov/kato:2014b}.} According to this asymptotic framework, the number of moment inequalities, denoted by $p$, is allowed to be larger than the sample size $n$. In fact, the asymptotic framework allows $p$ to be an increasing function of $n$ and even to grow at exponential rates. Furthermore, \citetalias{chernozhukov/chetverikov/kato:2014c} allow their moment inequalities to be ``unstructured'', in the sense that they do not impose restrictions on the correlation structure of the sample moment conditions.\footnote{This feature distinguishes their framework from a standard conditional moment (in)equality model. While conditional moment conditions can generate an uncountable set of unconditional moment (in)equalities, their covariance structure is greatly restricted by the conditioning structure.} For these reasons, \citetalias{chernozhukov/chetverikov/kato:2014c} represents a significant advancement relative to the previous literature on inference in moment inequalities. In this paper, we generalize their econometric framework by also allowing for the presence of many unstructured moment equalities, i.e., we consider a partially identified model with many moment (in)equalities. This generalization is relevant in practice as applications of partially identified econometric models often include both moment inequalities and equalties.

This paper builds on the inference methods proposed in \citetalias{chernozhukov/chetverikov/kato:2014c}. Their goal is to test whether a collection of $p$ moment inequalities simultaneously holds or not. Their hypothesis test compares a test statistic given by the maximum of $p$ Studentized statistics, with suitable critical values that are based on three methods: self-normalization, multiplier bootstrap, and empirical bootstrap. In addition, each one of these can be implemented with an optional first-stage moment selection procedure with the objective of detecting slack moment inequalities, thus increasing the statistical power. If the first stage is used, it can be in turn based on the same three methods: self-normalization, multiplier bootstrap, and empirical bootstrap. So, for example, one possible critical value is one based on empirical bootstrap with a first-step moment selection procedure based on self normalization. According to their simulation results in \citetalias{chernozhukov/chetverikov/kato:2014c}, using a first-step moment selection procedure can produce significant power gains.

We contribute to this literature by proposing new critical values for the hypothesis testing problem in \citetalias{chernozhukov/chetverikov/kato:2014c}. Our critical values are the result of combining the approximation methods in \citetalias{chernozhukov/chetverikov/kato:2014c} in the second step (i.e.\ self-normalization, multiplier bootstrap, or empirical bootstrap), with a novel first-step moment inequality selection procedure based on the Lasso. Besides the proposing a different first-step moment selection procedure, our inference method uses a second step that can ignore the first-step moment inequality selection, thus increasing statistical power. We refer to the resulting hypothesis test as a two-step Lasso-based inference methods.

On the theoretical front, we contribute by investigating the asymptotic properties of our two-step Lasso inference methods. We establish the following results. First, we provide conditions under which our methods are uniformly valid, both in the underlying parameter $\theta $ and the distribution of the data. According to the literature in moment (in)equalities, obtaining uniformly valid asymptotic results is important to guarantee that the asymptotic analysis provides an accurate approximation to finite sample results.\footnote{In these models, the limiting distribution of the test statistic is discontinuous in the slackness of the moment inequalities, while its finite sample distribution does not exhibit such discontinuities. In consequence, asymptotic results obtained for any fixed distribution (i.e.\ pointwise asymptotics) can be grossly misleading, and possibly producing confidence sets that undercover (even asymptotically). See \cite{imbens/manski:2004}, \cite{andrews/guggenberger:2009}, \cite{andrews/soares:2010}, and \cite{andrews/shi:2013} (Section 5.1).} Second, by virtue of results in \citetalias{chernozhukov/chetverikov/kato:2014c}, all of our proposed tests are asymptotically optimal in a minimax sense. Third, we compare the power of our methods to the corresponding one in \citetalias{chernozhukov/chetverikov/kato:2014c}, both in theory and in simulations. Since our two-step procedure and the corresponding one in \citetalias{chernozhukov/chetverikov/kato:2014c} are based on the same approximations, our power comparison is a comparison of the Lasso-based first step vis-\`a-vis the ones in \citetalias{chernozhukov/chetverikov/kato:2014c}. On the theory front, we obtain a region of underlying parameters under which the power of our method dominates that of \citetalias{chernozhukov/chetverikov/kato:2014c}. We also conduct extensive simulations to explore the practical consequences of our theoretical findings. Our simulations indicate that a Lasso-based first step is usually as powerful as the one in \citetalias{chernozhukov/chetverikov/kato:2014c}, and can sometimes be more powerful. Fourth, we show that our Lasso-based first step can be implemented by thresholding standardized sample averages (see Lemma \ref{lem:LassoClosedForm}) and so it is straightforward to implement. At the same time, this implies that our Lasso-based first step coincides with a self-normalization first step with a very specific choice of the tuning parameter. This choice of tuning parameter is an important aspect of our contribution, as it is responsible for the aforementioned power advantages.

As in any other Lasso problem, the validity of our method depends on the appropriate choice of the Lasso tuning parameter. One limitation of our approach is that the validity of our Lasso-based first step is shown for a tuning parameter choice that depends on unknown population moments (see Eq.\ \eqref{eq:LambdaConcrete}). This is not specific to our problem as it is a characteristic feature of much of the Lasso literature. To implement our Lasso-based first step in practice, we propose replacing population moments with their sample counterparts. Our simulations show that this choice delivers excellent results in simulations, in terms of both size control and power. However, we acknowledge that our current theoretical results do not take into account the sample variability of our empirical version of the tuning parameter. We consider that a rigorous handling of these issues is beyond the scope of this paper.

The Lasso was first proposed in the seminal contribution by \cite{tibshirani:1996} as a regularization technique in the linear regression model in which the number of regressors is allowed to exceed the sample size. Since then, this method has found wide use as a dimension reduction technique in large dimensional models with strong theoretical underpinnings.\footnote{For excellent reviews of this method see, e.g., \cite{belloni/chernozhukov:2011}, \cite{buhlmann/vandegeer:2011}, \cite{fan/lv/qi:2011}, and \cite{hastie/tibshirani/wainwright:2015}.} It is precisely these powerful shrinkage properties that serve as motivation to consider the Lasso as a procedure to separate out and select binding moment inequalities from the non-binding ones in a partially identified model with many moment (in)equalities.

This paper proposes using the Lasso to select moments in a partially identified moment (in)equality model. In the context of point identified problems, there is an existing literature that proposes the Lasso to address estimation and moment selection in GMM settings. In particular, \cite{caner:2009} introduce Lasso type GMM-Bridge estimators to estimate structural parameters in a general model. The problem of selection of moment in GMM is studied in \cite{liao:2013} and \cite{cheng/liao:2015}. In addition, \cite{caner/zhang:2014} and \cite{caner/han/lee:2016} provide a method to estimate parameters in a GMM model with diverging number of moments/parameters and to select valid moments among many valid or invalid moments, respectively. In addition, \cite{fan/liao/yao:2015} consider the problem of inference in high dimensional models with sparse alternatives. Finally, \cite{caner/fan:2015} propose a hybrid two-step estimation procedure based on Generalized Empirical Likelihood, where instruments are chosen in a first-stage using an adaptive Lasso procedure.

The remainder of the paper is organized as follows. Section \ref{sec:Setup} describes the inference problem and introduces our assumptions. Section \ref{sec:Lasso} introduces the Lasso as a method to distinguish binding moment inequalities from non-binding ones and Section \ref{sec:Inference} considers inference methods that use the Lasso as a first-step moment selection procedure. Section \ref{sec:Power} compares the power properties of inference methods based on the Lasso with the ones in the literature. Section \ref{sec:MonteCarlos} provides evidence of the finite sample performance using Monte Carlo simulations. Section \ref{sec:Conclusions} concludes. Proofs of the main results and several intermediate results are reported in the appendix. 

Throughout the paper, we use the following notation. For any set $S$, $|S|$ denotes its cardinality, and for any vector $x \in \mathbb{R}^{d}$, $||x||_{1}\equiv \sum_{i=1}^{d}|x_i|$.

\section{Setup}\label{sec:Setup}

For each $\theta \in \Theta \subseteq \mathbb{R}^{q}$ with $q\in \mathbb{N}$, let $X(\theta):\Omega \to \mathbb{R}^{k}$ be a $k$-dimensional random variable with distribution $P$ and mean $\mu (\theta )\equiv E[X(\theta)]\in \mathbb{R} ^{k}$. Let $\mu _{j}(\theta )$ denote the $j$th component of $\mu(\theta)$ so that $\mu (\theta )=\{\mu _{j}(\theta )\}_{j\leq k}$. 
The econometric model predicts that the true parameter value $\theta _{0}$ satisfies the following collection of $p$ moment inequalities and $v\equiv k-p$ moment equalities: 
\begin{eqnarray} 
	\mu _{j}(\theta _{0}) &\leq &0\text{ for all }j=1,\ldots ,p, \text{ and } \notag \\
	\mu _{j}(\theta _{0}) &=&0\text{ for all }j=p+1,\ldots ,k. \label{eq:MI}
\end{eqnarray} 
As in \citetalias{chernozhukov/chetverikov/kato:2014c}, we are implicitly allowing the distribution $P$ and the number of moment (in)equalities, $k = p+v$ to depend on $n$, and we assume that $2k-p >1$. In particular, we are primarily interested in the case in which $p=p_{n}\to \infty$ and $v=v_{n}\to \infty$ as $n \to \infty$, but the subscripts will be omitted to keep the notation simple. In particular, $p$ and $v$ can be much larger than the sample size and increase at rates made precise in Section \ref{sec:Assumptions}. In addition, we allow the econometric model to be partially identified, i.e., the moment (in)equalities in Eq.\ \eqref{eq:MI} do not necessarily restrict $\theta _{0}$ to a single value, but rather they constrain it to belong to the identified set, denoted by $\Theta _{I}(P)$. By definition, the identified set is as follows:
\begin{equation}
\Theta _{I}(P)~\equiv ~\left\{ \theta \in \Theta :\left\{
\begin{array}{l}
	\mu _{j}(\theta )\leq 0\text{ for all }j=1,\ldots ,p, \text{ and }  \\ 
	\mu _{j}(\theta )=0\text{ for all }j=p+1,\ldots ,k
\end{array}\right\} \right\}.  \label{eq:IdSet}
\end{equation}
Since our econometric model is characterized by the moment (in)equalities in Eq.\ \eqref{eq:MI}, all of the parameter values in the identified set are considered to be observationally equivalent to $\theta _{0}$ in our econometric framework.

Our goal is to test whether a particular parameter value $\theta \in \Theta $ is a possible candidate for the true parameter value $\theta _{0} \in \Theta_{I}(P)$, or not. In other words, we are interested in testing:
\begin{equation}
	H_{0}:\theta _{0}=\theta~~ \text{ vs. }~~H_{1}:\theta _{0}\not=\theta .
	\label{eq:HypTest}
\end{equation}
In our partially identified framework, every parameter in the identified set $\Theta _{I}(P)$ is deemed observationally equivalent to the true parameter value $\theta _{0}$. In other words, we consider $H_{0}:\theta=\theta _{0}$ to be observationally equivalent to $H_{0}:\theta \in \Theta _{I}(P)$. As a consequence, Eq.\ \eqref{eq:HypTest} can be equivalently reexpressed as follows:
\begin{align}
	H_{0}:\theta \in \Theta _{I}(P)\quad \text{vs.}\quad H_{1}:\theta \not\in \Theta _{I}(P)
	 \label{eq:HypTest2a}
\end{align}
or, alternatively,
\begin{align}
H_{0} :\left\{\begin{array}{l}
	\mu _{j}(\theta )\leq 0\text{ for all }j=1,\ldots ,p, \text{ and }  \\ 
	\mu _{j}(\theta )=0\text{ for all }j=p+1,\ldots ,k
\end{array}\right\}~ \text{vs.}~ H_{1} :\left\{\begin{array}{l}
	\mu _{j}(\theta )> 0\text{ for some }j=1,\ldots ,p, \text{ or }\\ 
	\mu _{j}(\theta )\neq0\text{ for some }j=p+1,\ldots ,k
\end{array}\right\}.
 \label{eq:HypTest2}
\end{align}
 
In this paper, we propose a procedure to implement the hypothesis test in Eq.\ \eqref{eq:HypTest} (or, equivalently, Eq.\ \eqref{eq:HypTest2a} or Eq.\ \eqref{eq:HypTest2}) with a given significance level $\alpha \in (0,1)$ based on a random sample of $ X(\theta)\sim P(\theta)$, denoted by $X^{n}(\theta) \equiv \{X_{i}(\theta)\}_{i\leq n}$. The inference procedure will reject $H_0$ whenever a test statistic $ T_{n}(\theta )$ exceeds a critical value $c_{n}(\alpha ,\theta )$, i.e.,
\begin{equation}
\phi _{n}(\alpha ,\theta )~\equiv ~ 1[T_{n}(\theta )>c_{n}(\alpha ,\theta )].
\label{eq:HT}
\end{equation}
By the duality between hypothesis tests and confidence sets, a confidence set for $\theta _{0}$ can be constructed by collecting all parameter values for which the inference procedure is not rejected, i.e.,
\begin{equation}
C_{n}(1-\alpha )~\equiv ~\{\theta \in \Theta :T_{n}(\theta ) \leq c_{n}(\alpha ,\theta )\}. \label{eq:CS}
\end{equation}
Our formal results will have the following structure. Let $\mathcal{P}$ denote a set of probability distributions. We will show that for all $P \in \mathcal{P}$ and $\theta \in \Theta$ that satisfies $H_0$,
\begin{equation}
	P\del[1]{T_{n}(\theta ) > c_{n}(\alpha ,\theta )}~~\leq ~~ \alpha + o(1). \label{eq:CSvalid2}
\end{equation}
Moreover, the convergence in Eq.\ \eqref{eq:CSvalid2} will be shown to occur uniformly over both $P \in \mathcal{P}$ and $\theta\in \Theta$ that satisfies $H_0$. This uniform size control result in Eq.\ \eqref{eq:CSvalid2} has important consequences regarding our inference problem. First, this result immediately implies that the hypothesis test procedure in Eq.\ \eqref{eq:HT} uniformly controls asymptotic size i.e., for all $\theta \in \Theta$ and under $H_0:\theta_0 = \theta$,
\begin{equation}
	\underset{n\to \infty }{\lim \sup }~\sup_{P\in \mathcal{P} }~E[\phi _{n}(\alpha ,\theta )]~~\leq ~~ \alpha. \label{eq:HTvalid}
\end{equation}
Second, the result also implies that the confidence set in Eq.\ \eqref{eq:CS} is asymptotically uniformly valid, i.e.,
\begin{equation}
	\underset{n\to \infty }{\lim \inf }~\inf_{P\in \mathcal{P} }~\inf_{\theta \in \Theta _{I}(P)}~P\del[1]{\theta\in C_{n}(1-\alpha )}~~\geq ~~1-\alpha. \label{eq:CSvalid}
\end{equation}

The rest of the section is organized as follows. Section \ref{sec:Assumptions} specifies the assumptions on the probability space $\mathcal{P}$ that are required for our analysis. All the inference methods described in this paper share the test statistic $ T_{n}(\theta )$ and differ only in the critical value $c_{n}(\alpha ,\theta )$. The common test statistic is introduced and described in Section \ref{sec:TestStat}.

\subsection{Assumptions}\label{sec:Assumptions}

This paper considers the following assumptions.

\begin{assumptionA} \label{ass:Basic}
	For every $\theta \in \Theta \in \mathbb{R}^{q}$, let ${X}^{n}(\theta) \equiv \{X_{i}(\theta)\}_{i\leq n}$ be i.i.d.\ $k$-dimensional random vectors distributed according to $P\in \mathcal{P}$. Further, let $E[X_{1j}(\theta)]\equiv \mu_{j}(\theta) $, $Var[X_{1j}(\theta)] \equiv \sigma^2_{j}(\theta) >0$, and $Z_{ij}(\theta) \equiv (X_{ij}(\theta) - \mu_{j}(\theta))/ \sigma_{j}(\theta)$, where $X_{ij}(\theta)$ denotes the $j$ component of $X_i(\theta)$.
\end{assumptionA}

\begin{assumptionA} \label{ass:Rates}
For some constants $\delta \in (0,1]$, $c\in ((1-\delta)/2,1/2)$ and $C>0$, $M_{n,2+\delta }^{2+\delta }( \ln (2k-p) ) ^{( 2+\delta) /2}\leq C n^{1/2 - c}$ with $M_{n,2+\delta } \equiv \max_{j=1,\dots ,k} \sup_{\theta \in \Theta } ( E[|Z_{1j}(\theta)|^{2+\delta }]) ^{1/( 2+\delta ) } $ and  $Z_{ij}(\theta)$ is as in Assumption \ref{ass:Basic}.
\end{assumptionA}

\begin{assumptionA} \label{ass:Rates2}
For some constants $c\in (0,1/2)$ and $C>0$, $ B_{n}^{2}(\ln (2k-p)) \leq C n^{1/2 - c}$ with $B_{n} \equiv \sup_{\theta \in \Theta}( E[\max_{j=1,\dots,k}|Z_{1j}(\theta)|^{4 }]) ^{1/4 } $ and  $Z_{ij}(\theta)$ is as in Assumption \ref{ass:Basic}.
\end{assumptionA}

\begin{assumptionA} \label{ass:Rates3}
For some constants $c\in (0,1/2)$ and $C>0$, $\max \{M_{n,3}^{3},M_{n,4}^{2},B_{n}\}^{2}(\ln ( (2k-p)n) )^{7/2}\leq Cn^{1/2-c}$, where $M_{n,3 }$, $M_{n,4}$, and $B_{n}$ are as in Assumptions \ref{ass:Rates} and \ref{ass:Rates2}.
\end{assumptionA}

We now briefly describe these assumptions and relate them to those used by \citetalias{chernozhukov/chetverikov/kato:2014c}. Assumption \ref{ass:Basic} is standard in microeconometric applications. Assumption \ref{ass:Rates} defines $M_{n,2+\delta }^{2+\delta}$ as the maximum over $\theta\in \Theta$ and $j =1,\dots,k$ of the $(2+\delta)$-absolute moments of $Z_{ij}(\theta)\equiv (X_{ij}(\theta) - \mu_{j}(\theta))/ \sigma_{j}(\theta)$. Note that $2k-p = 2v+p$, i.e., the total number of moment inequalities $p$ plus twice the number of moment equalities $v$, all of which could depend on $n$. Also, note that $M_{n,2+\delta }$ is a function of $P$ and $k=v+p$, both of which could depend on $n$. In this sense, Assumption \ref{ass:Rates} is limiting the rate of growth of $M_{n,2+\delta }$ and the number of moment (in)equalities as $n$ diverges. If we restrict to $\delta =1$, Assumption \ref{ass:Rates} is analogous to \citetalias{chernozhukov/chetverikov/kato:2014c} (Eqs.\ (22) or (27)). Assumption \ref{ass:Rates2} has a similar interpretation as Assumption \ref{ass:Rates}, except that it applies to $\max_{j=1,\dots,k}|Z_{1j}(\theta)|^{4}$ instead of $|Z_{1j}(\theta)|^{2+\delta }$. Note that Assumption \ref{ass:Rates2} is analogous to \citetalias{chernozhukov/chetverikov/kato:2014c} (Eq.\ (27)). Assumptions \ref{ass:Basic}-\ref{ass:Rates2} will be used to show that our first-step moment selection based on the Lasso will not exclude any of the true binding moment inequalities (see Lemma \ref{lem:LassoNoOverFit}). Based on this result, we then show that the two-step Lasso SN approximation provides uniform size control under the same conditions (see Theorem \ref{thm:Lasso2SSize}). Assumption \ref{ass:Rates3} is an additional condition that further restricts the rate of growth of the moments of the random variables and number of moment (in)equalities as $n\to \infty$. We show in Lemma \ref{lem:A4impliesA2andA3} that Assumption \ref{ass:Rates3} implies that Assumption \ref{ass:Rates} holds with $\delta = 1$ and Assumption \ref{ass:Rates2} holds. We also note that Assumption \ref{ass:Rates3} is analogous to \citetalias{chernozhukov/chetverikov/kato:2014c} (Eq.\ (35)). This stronger condition will be used to provide asymptotic size control for the bootstrap-based inference. In particular, under Assumptions \ref{ass:Basic} and \ref{ass:Rates3}, we show that the two-step Lasso bootstrap approximation provides uniform size control (see Theorem \ref{thm:Lasso2SBootstrap}).\footnote{Assumptions \ref{ass:Basic}-\ref{ass:Rates3} are tailored for the construction of confidence sets in Eq.\ \eqref{eq:CS} in the sense that all the relevant constants are defined uniformly for all $\theta \in \Theta$. If we were only interested in the hypothesis testing problem in Eq.\ \eqref{eq:HT} for a particular value of ${\theta}$, then the previous uniform assumptions could be replaced by their pointwise versions at the parameter being tested. The pointwise version of these assumptions would involve $\theta$-specific sequences $M_{n,2+\delta}(\theta)$ and $B_{n}(\theta)$, and would thus be weaker than their uniform counterparts. In addition, the pointwise version of our resulting two-step inference methods would use $M_{n,2+\delta}(\theta)$ in Eq.\ \eqref{eq:LambdaConcrete}, and can thus be shown to have more statistical power than their uniform counterparts.}

\subsection{Test statistic}\label{sec:TestStat}
Throughout the paper, we consider the following test statistic:
\begin{equation}
T_{n}(\theta) ~\equiv~ \max \cbr[3]{\max_{j=1,\dots ,p}\frac{\sqrt{n}\hat{\mu}_{j}(\theta)}{\hat{ \sigma}_{j}(\theta)},\max_{s=p+1,\dots
,k}\frac{\sqrt{n}\left\vert \hat{\mu} _{s}(\theta)\right\vert }{\hat{\sigma}_{s}(\theta)}},
\label{eq:TestStat}
\end{equation}
where, for $j=1,\dots ,k$,
$
\hat{\mu}_{j}(\theta) \equiv \frac{1}{n}\sum_{i=1}^{n}X_{ij}(\theta)$  and 
$\hat{\sigma}_{j}^{2}(\theta) \equiv \frac{1}{n}\sum_{i=1}^{n}\left( X_{ij}(\theta)-\hat{\mu} _{j}(\theta)\right) ^{2}
$.
Note that Eq.\ \eqref{eq:TestStat} is not properly defined if $\hat{\sigma}_{j}^{2}(\theta)=0$ for some $j=1,\dots,k$ and, in such cases, we use the convention that $x/0 \equiv \infty \times 1[x>0] -\infty \times 1[x<0] $.

The test statistic is identical to that in \citetalias{chernozhukov/chetverikov/kato:2014c} with the exception that we allow for the presence of moment equalities. By definition, large values of $T_n(\theta)$ are an indication that $H_0: \theta = \theta_0$ is likely to be violated, leading to the rejection of the hypothesis test in Eq.\ \eqref{eq:HT}. The remainder of the paper considers several procedures to construct critical values that can be associated to this test statistic.

\section{Lasso as a first-step moment selection procedure}\label{sec:Lasso}

In order to propose a critical value for our test statistic $T_{n}(\theta)$, we need to approximate its distribution under $H_0$. According to the econometric model in Eq.\ \eqref{eq:MI}, the true parameter satisfies $p$ moment inequalities and $v$ moment equalities. By definition, the moment equalities are always binding under $H_0$. On the other hand, the moment inequalities may or may not be binding under $H_0$, and a successful approximation of the asymptotic distribution depends on being able to distinguish between these two cases. Incorporating this information into the hypothesis testing problem is one of the key issues in the literature on inference in partially identified moment (in)equality models.

In their seminal contribution, \citetalias{chernozhukov/chetverikov/kato:2014c} is the first paper in the literature to conduct inference in a partially identified model with many unstructured moment inequalities. Their paper proposes several procedures to select binding moment inequalities from non-binding based on three approximation methods: self-normalization (SN), multiplier bootstrap (MB), and empirical bootstrap (EB). Our contribution is to propose a novel approximation method based on the Lasso. By definition, the Lasso penalizes parameters values by their $\ell_{1}$-norm, with the ability of producing parameter estimates that are exactly equal to zero. This powerful shrinkage property is precisely what motivates us to consider the Lasso as a first-step moment selection procedure in a model with many moment (in)equalities. As we will soon show, the Lasso is an excellent method to detect binding moment inequalities from non-binding ones, and this information can be successfully incorporated into an inference procedure for many moment (in)equalities.

For every $\theta \in \Theta$, let $J(\theta)$ denote the true set of binding moment inequalities, i.e., $J(\theta)~\equiv~ \{j=1,\dots ,p~:~\mu _{j}(\theta)\geq 0\}$. Let $\mu_{I}(\theta) \equiv \{\mu_j(\theta)\}_{j=1}^{p}$ denote the moment vector for the moment inequalities and let $\hat{\mu}_{I}(\theta) \equiv \{\hat{\mu}_j(\theta)\}_{j=1}^{p}$ denote its sample analogue. In order to detect binding moment inequalities, we consider the weighted Lasso estimator of $\mu_{I}(\theta)$, given by
\begin{equation}
\hat{\mu}_{L}(\theta) ~\equiv~  \underset{t\in \mathbb{R} ^{p}}{\arg \min } \cbr[3]{ \left( \hat{\mu}_{I}(\theta) -t\right)^{\prime }\hat{W}(\theta)\left( \hat{\mu}_{I}(\theta)-t\right) + \lambda _{n}\left\Vert \hat{W}(\theta)^{1/2}t\right\Vert
_{1}}, \label{eq:Lasso0}
\end{equation}
where $\lambda _{n}$ is a positive penalization sequence that controls the amount of regularization and $\hat{W}(\theta)$ is a positive definite weighting matrix. To simplify the computation of the Lasso estimator, we impose $\hat{W}(\theta) \equiv diag\{ 1/\hat{\sigma}_{j}(\theta)^{2}\} _{j=1}^{p}$. As a consequence, Eq.\ \eqref{eq:Lasso0} becomes:
\begin{equation}
\hat{\mu}_{L}(\theta)  ~=~ \cbr[3]{ \underset{m\in \mathbb{R} }{\arg \min }\left\{ \left( {\hat{ \mu}_{j}(\theta)}-m\right) ^{2}+\lambda _{n}\hat{\sigma}_{j}(\theta)|m|\right\}} _{j=1}^{p}.\label{eq:Lasso1}
\end{equation}

Notice that instead of using the Lasso in one $p$-dimensional model we instead use it in $p$ one-dimensional models. As we shall see later, $\hat{\mu}_{L}(\theta)$ in Eq.\ \eqref{eq:Lasso1} is closely linked to the soft-thresholded least squares estimator, which implies that its computation is straightforward.  The Lasso estimator $\hat{\mu}_{L}(\theta)$ implies a Lasso-based estimator of $J(\theta)$, given by
\begin{equation}
\hat{J}_{L}(\theta) ~\equiv~ \{j=1,\dots ,p:\hat{\mu}_{j,L}(\theta)/\hat{\sigma}_{j}(\theta) ~\geq~ -\lambda _{n}\}.
\label{eq:BindingLasso}
\end{equation}
In order to implement this procedure, we need to choose the sequence $\lambda_n$, which determines the degree of regularization imposed by the Lasso. A higher value of $\lambda_{n}$ will produce a larger number of moment inequalities considered to be binding, resulting in a lower rejection rate. In consequence, this is a critical choice for our inference methodology. According to our theoretical results,
 a suitable choice of $\lambda_n$ is
\begin{equation}
\lambda _{n}~=~(4/3+\varepsilon) n^{-1/2}\left( {M}_{n,2+\delta }^{2}n^{-\delta /(2+\delta )}-n^{-1}\right) ^{-1/2} \label{eq:LambdaConcrete}
\end{equation}
for any arbitrary $\varepsilon>0$. Assumption \ref{ass:Rates} implies that $\lambda _{n}$ in Eq.\ \eqref{eq:LambdaConcrete} satisfies $\lambda_n\to 0$. Notice that Eq.\ \eqref{eq:LambdaConcrete} is infeasible as it depends on the unknown expression ${M}_{n,2+\delta }$. In practice, one can replace this unknown expression with its sample analogue:
\begin{equation}
	\hat{M}^{2}_{n,2+\delta } ~=~  \max_{j=1,\dots ,k} ~\sup_{\theta \in \Theta } ~ \left(n^{-1} \sum\nolimits_{i=1}^{n}|(X_{ij}(\theta) - \hat{\mu}_{j}(\theta))/\hat{\sigma}_{j}(\theta)|^{2+\delta }\right) ^{2/( 2+\delta ) }.
	\label{eq:Mconcrete}
\end{equation}
We show in Section \ref{sec:MonteCarlos} that this practical choice delivers excellent results in simulations, in terms of both size control and power. Our theoretical results do not take into account the sample variability of our empirical version of the tuning parameter. We consider that a rigorous handling of these issues is beyond the scope of this paper.

As explained earlier, our Lasso procedure is used in the first step to detect binding moment inequalities from non-binding ones. The following result formally establishes that our Lasso procedure includes all binding ones with a probability that approaches one, uniformly.

\begin{lemma} \label{lem:LassoNoOverFit}
	Assume Assumptions \ref{ass:Basic}-\ref{ass:Rates2}, and let $\lambda_n$ be as in Eq.\ \eqref{eq:LambdaConcrete}. Then,
\begin{align*}
&P(J(\theta) \subseteq \hat{J}_{L}(\theta))\geq \\
&1-2\exp [ \ln (2k-p)( 1-n^{(2c+\delta -1)/(2+\delta)}/( 2C^{2/(2+\delta)}) ) ] [1+K( ( Cn^{-c+( 1-\delta ) /2}) ^{1/( 2+\delta ) }+1) ^{2+\delta }]-\tilde{K}n^{-c}\to 1,
\end{align*}
where $K,C,\tilde{K}>0$ and $c>(1-\delta)/2$ are universal constants, and so the convergence is uniform in all parameters $\theta$ and distributions $P$ that satisfy the conditions in the statement.
\end{lemma}

Thus far, our Lasso estimator of the binding constrains in Eq.\ \eqref{eq:BindingLasso} has been defined in terms of the solution of the $p$-dimensional minimization problem in Eq.\ \eqref{eq:Lasso1}. We conclude the subsection by providing an equivalent closed form solution for this set.

\begin{lemma}\label{lem:LassoClosedForm}
	Eq.\ \eqref{eq:BindingLasso} can be equivalently reexpressed as follows:	
	\begin{equation}
	\hat{J}_{L}(\theta) ~=~ \{j=1,\dots ,p: {\hat{\mu}_{j}(\theta)}/{\hat{\sigma}_{j}(\theta)}\geq -{3}\lambda_n/{2}\}.
	\label{eq:BindingLasso2}
	\end{equation}
\end{lemma}

Lemma \ref{lem:LassoClosedForm} is a very important computational aspect of our methodology. This result reveals that $\hat{J}_{L}(\theta) $ can be computed by comparing standardized sample averages with a modified threshold of $-{3}\lambda_n/{2}$. In other words, our Lasso-based first stage can be implemented without the need of solving the $p$-dimensional minimization problem in Eq.\ \eqref{eq:Lasso1}.

\section{Inference methods with Lasso-based first step}\label{sec:Inference}

In the remainder of the paper we show how to conduct inference in our partially identified many moment (in)equality model by combining the Lasso-based first step in Section \ref{sec:Lasso} with a second step based on the inference methods proposed by \citetalias{chernozhukov/chetverikov/kato:2014c}. In particular, Section \ref{sec:SN} combines our Lasso-based first step with their self-normalization approximation, while Section \ref{sec:Boot} combines it with their bootstrap approximations.

\subsection{Self-normalization approximation}\label{sec:SN}

Before describing our self-normalization (SN) approximation with Lasso first stage, we first describe the ``plain vanilla'' SN approximation without first stage moment selection. Our treatment extends the SN method proposed by \citetalias{chernozhukov/chetverikov/kato:2014c} to the presence of moment equalities, which are relevant in applications. 

As a preliminary step, we now define the SN approximation to the $(1-\alpha)$-quantile of $T_{n}(\theta)$ in a hypothetical moment (in)equality model composed of $|J|$ moment inequalities and $k-p$ moment equalities, given by
\begin{equation}
c_{n}^{SN}(|J|,\alpha)\equiv \left\{
\begin{array}{ll}
0 & \text{if }2(k-p)+|J|=0, \\
\tfrac{\Phi ^{-1}\left( 1-\alpha /(2(k-p)+|J|)\right) }{\sqrt{1-\left( \Phi ^{-1}\left( 1-\alpha /(2(k-p)+|J|)\right) \right) ^{2}/n}} & \text{if } 2(k-p)+|J|>0.
\end{array}
\right.
 \label{eq:CVSN_oracle}
\end{equation}
Lemma \ref{lem:SNSize} in the appendix shows that $c_{n}^{SN}(|J|,\alpha)$ provides asymptotic uniform size control in a hypothetical moment (in)equality model with $|J|$ moment inequalities and $k-p$ moment equalities under Assumptions \ref{ass:Basic}-\ref{ass:Rates}. The only difference between this result and \citetalias{chernozhukov/chetverikov/kato:2014c} (Theorem 4.1) is that we allow for the presence of moment equalities. Since our moment (in)equality model has $|J|=p$ moment inequalities and $k-p$ moment equalities, we can define the regular (i.e.\ one-step) SN approximation method by using $|J|=p$ in Eq.\ \eqref{eq:CVSN_oracle}, i.e.,
\begin{equation*}
c_{n}^{SN,1S}(\alpha )\equiv c_{n}^{SN}(p,\alpha )=\tfrac{\Phi ^{-1}\left( 1-\alpha /(2k-p)\right) }{\sqrt{1-\left( \Phi ^{-1}\left( 1-\alpha /(2k-p)\right) \right) ^{2}/n}}.
\end{equation*}
The following result is a corollary of Lemma \ref{lem:SNSize}.

\begin{theorem}[One-step SN approximation] \label{thm:SN1Scorollary}
Assume Assumptions \ref{ass:Basic}-\ref{ass:Rates}, $\alpha \in (0,0.5)$, and that $H_{0}$ holds. Then,
\begin{align*}
&P(T_n(\theta) > c_{n}^{SN,1S}(\alpha )) \\
&\leq
\alpha +\alpha 2^{1+\delta }KCn^{-c+( 1-\delta ) /2}[ ( \ln ( 2k-p) ) ^{-( 2+\delta ) /2}+{2} ^{1/2}( 1-( \ln \alpha ) /( \ln ( 2k-p) ) ) ^{(2+\delta )/2}] \to \alpha,
\end{align*}
where $K,C>0$ and $c>(1-\delta)/2$ are universal constants, and so the convergence is uniform in all parameters $\theta$ and distributions $P$ that satisfy the conditions in the statement.
\end{theorem}

By definition, this SN approximation considers all moment inequalities in the model as binding. A more powerful test can be constructed by using the data to reveal which moment inequalities are slack.  In particular, \citetalias{chernozhukov/chetverikov/kato:2014c} propose a two-step SN procedure which combines a first-step moment inequality selection based on SN methods and the second step SN critical value in Theorem \ref{thm:SN1Scorollary}. If we adapt their procedure to the presence of moment equalities, this yields
\begin{align}
c_n^{SN,2S}(\theta,\alpha) ~&\equiv~ c_{n}^{SN}(|\hat{J}_{SN}(\theta)|,\alpha - 2 \beta_{n} )
\label{eq:SN2S-CV}
\end{align}
with
\begin{align*}
\hat{J}_{SN}(\theta) ~&\equiv ~ \Big\{j\in \{1,\dots,p\}~:
~\sqrt{n}\hat{\mu}_j(\theta)/\hat{\sigma}_j(\theta)>-2c_n^{SN,1S}(\beta_n)\Big\},  
\end{align*}
where $\{\beta_n\}_{n\ge 1}$ is an arbitrary sequence of constants in $(0,\alpha/2)$. By extending arguments in \citetalias{chernozhukov/chetverikov/kato:2014c} to include moment equalities, one can show that inference based on the critical value $c^{SN,2S}(\theta,\alpha)$ in Eq.\ \eqref{eq:SN2S-CV} is asymptotically valid in a uniform sense.


In this paper, we propose an alternative SN procedure by using our Lasso-based first step. In particular, we define the following two-step Lasso SN critical value:
\begin{eqnarray}
c_{n}^{SN,L}(\theta, \alpha )~\equiv~ c_{n}^{SN}(|\hat{J}_{L}(\theta)|,\alpha) ,
\label{eq:cn-SNLasso}
\end{eqnarray}
where $\hat{J}_{L}(\theta)$ is as in Eq.\ \eqref{eq:BindingLasso2}. The following result shows that an inference method based on our two-step Lasso SN critical value is asymptotically valid in a uniform sense.

\begin{theorem}[Two-step Lasso SN approximation] \label{thm:Lasso2SSize}
	Assume Assumptions \ref{ass:Basic}-\ref{ass:Rates2}, $\alpha \in (0,0.5)$, and that $H_{0}$ holds, and let $\lambda_n$ be as in Eq.\ \eqref{eq:LambdaConcrete}. Then,
\begin{align*}
&P(T_{n}(\theta)>c_{n}^{SN,L}(\theta,\alpha))\leq\\
& \alpha+\left\{ 
\begin{array}{c}
\alpha 2^{1+\delta }KCn^{-c+( 1-\delta ) /2}[ ( \ln ( 2k-p) ) ^{-( 2+\delta ) /2}+{2} ^{1/2}( 1-( \ln \alpha ) /( \ln ( 2k-p) ) ) ^{(2+\delta )/2}]+2\tilde{K}n^{-c} \\ 
4\exp [ \ln (2k-p)( 1-n^{(2c+\delta -1)/(2+\delta)}/( 2C^{2/(2+\delta)}) ) ][1+K((Cn^{-c+(1-\delta)/2})^{1/(2+\delta )}+1)^{2+\delta }]
\end{array}
\right\}\to \alpha,
\end{align*}
where $K,C,\tilde{K}>0$ and $c>(1-\delta)/2$ are universal constants, and so the convergence is uniform in all parameters $\theta$ and distributions $P$ that satisfy the conditions in the statement.

\end{theorem}

We now compare our two-step SN Lasso method with the SN methods in \citetalias{chernozhukov/chetverikov/kato:2014c}. Since all inference methods share the test statistic, the only difference lies in the critical values. While the one-step SN critical values considers all $p$ moment inequalities as binding, our two-step SN Lasso critical value considers only $|\hat{J}_{L}(\theta)|$ moment inequalities as binding. Since $|\hat{J}_{L}(\theta)|\leq p$ and $c_{n}^{SN}(\alpha ,|J|)$ is weakly increasing in $|J|$ (see Lemma \ref{lem:CSNincreasing} in the appendix), then our two-step SN method results in a weakly larger rejection probability for all sample sizes. In contrast, the comparison between $c_{n}^{SN,L}(\theta,\alpha)$ and $c_{n}^{SN,2S}(\theta,\alpha)$ is not straightforward as these differ in two aspects. First, the set of binding constrains $\hat{J}_{SN}(\theta)$ according to SN differs from the set of binding constrains $\hat{J}_{L}(\theta)$ according to the Lasso. Second, the quantile of the critical values are different: the two-step SN method in Eq.\ \eqref{eq:SN2S-CV} considers the $\alpha - 2\beta_{n}$ quantile while the Lasso-based method considers the usual $\alpha$ quantile. This second point reveals that our second step effectively ignores the first-step moment selection procedure based on the Lasso, resulting in higher statistical power. As a consequence of these differences, the comparison of these critical values is ambiguous, and so is the resulting comparison of power. This topic will be discussed in further detail in Section \ref{sec:Power}.

\subsection{Bootstrap-based methods}\label{sec:Boot}

\citetalias{chernozhukov/chetverikov/kato:2014c} also propose two bootstrap-based approximation methods: multiplier bootstrap (MB) and empirical bootstrap (EB). Relative to the SN approximation, bootstrap methods have the advantage of taking into account the dependence between the coordinates of $\{\sqrt{n}\hat{\mu}_{j}({\theta})/\hat{\sigma}_{j}(\theta)\}_{j=1}^{p}$ involved in the definition of the test statistic $T_{n}(\theta)$.

As in the previous subsection, we first define the bootstrap approximation to the $(1-\alpha)$-quantile of $T_n(\theta)$ in a hypothetical moment (in)equality model composed of moment inequalities indexed by the set $J$ and the $k-p$ moment equalities. The corresponding MB and EB approximations are denoted by $c_{n}^{MB}(\theta, J, \alpha)$ and $c_{n}^{EB}(\theta, J, \alpha)$, respectively, and are computed as follows.

\begin{algorithm}\label{alg:MB}\textbf{Multiplier bootstrap (MB)}
\begin{enumerate}
\item Generate i.i.d.\ standard normal random variables $\{\epsilon _{i}\}_{i = 1 }^{n}$, and independent of the data $X^{n}(\theta)$.
\item Construct the multiplier bootstrap test statistic:
\begin{equation*}
W_{n}^{MB}(\theta, J)=\max \cbr[4]{\max_{j\in J}\frac{\frac{1}{\sqrt{n}}\sum_{i=1}^{n} \epsilon _{i}( X_{ij}(\theta)-\hat{\mu}_{j}(\theta)) }{\hat{\sigma}_{j}(\theta)} ,\max_{s=p+1,\dots ,k}\frac{\frac{1}{\sqrt{n}}|\sum_{i=1}^{n} \epsilon _{i}( X_{is}(\theta)-\hat{\mu}_{s}(\theta)) |}{\hat{\sigma}_{s}(\theta)}}.
\end{equation*}
\item Calculate $c_{n}^{MB}(\theta, J, \alpha)$ as the conditional $(1-\alpha )$-quantile of $W_{n}^{MB}(\theta, J)$ (given $X^{n}(\theta)$).
\end{enumerate}
\end{algorithm}

\begin{algorithm}\label{alg:EB}\textbf{Empirical bootstrap (EB)}
\begin{enumerate}
\item Generate a bootstrap sample $\{ X ^{*}_{i}(\theta)\}_{i = 1 }^{n}$ from the data, i.e., an i.i.d.\ draw from the empirical distribution of $X^{n}(\theta)$.
\item Construct the empirical bootstrap test statistic:
\begin{equation*}
W_{n}^{EB}(\theta, J)=\max \cbr[4]{ \max_{j\in J}\frac{\frac{1}{\sqrt{n}}\sum_{i=1}^{n} ( X^{*}_{ij}(\theta)-\hat{\mu}_{j}(\theta)) }{\hat{\sigma}_{j}(\theta)} ,\max_{s=p+1,\dots ,k}\frac{\frac{1}{\sqrt{n}}|\sum_{i=1}^{n} ( X^{*}_{is}(\theta)-\hat{\mu}_{s}(\theta)) |}{\hat{\sigma}_{s}(\theta)}}.
\end{equation*}
\item Calculate $c_{n}^{EB}(\theta, J, \alpha)$ as the conditional $(1-\alpha )$-quantile of $W_{n}^{EB}(\theta, J)$ (given $X^{n}(\theta)$).
\end{enumerate}
\end{algorithm}

The results in the remainder of the section will apply to both versions of the bootstrap, and under the same assumptions. For this reason, we can use $c_{n}^{B}(\theta, J, \alpha)$ to denote the bootstrap critical value where $B \in \{MB,EB\}$ represents either MB or EB. Lemma \ref{lem:BootSize} in the appendix shows that $c_{n}^{B}(\theta, J, \alpha)$ for $B \in \{MB,EB\}$ provides asymptotic uniform size control in a hypothetical moment (in)equality model composed of moment inequalities indexed by the set $J$ and the $k-p$ moment equalities under Assumptions \ref{ass:Basic} and \ref{ass:Rates3}. As in Section \ref{sec:SN}, the only difference between this result and \citetalias{chernozhukov/chetverikov/kato:2014c} (Theorem 4.3) is that we allow for the presence of the moment equalities. Since our moment (in)equality model has $|J|=p$ moment inequalities and $k-p$ moment equalities, we can define the regular (i.e.\ one-step) MB or EB approximation method by using $|J|=p$ in Algorithm \ref{alg:MB} or \ref{alg:EB}, respectively, i.e.,
\begin{equation*}
	c_{n}^{B,1S}(\theta, \alpha )~\equiv~ c_{n}^{B}(\theta ,\{1,\dots,p\},\alpha),
\end{equation*} 
where $c_{n}^{B}(\theta,J,\alpha )$ is as in Algorithm \ref{alg:MB} if $B=MB$ or Algorithm \ref{alg:EB} if $B=EB$. The following result is a corollary of Lemma \ref{lem:BootSize}.

\begin{theorem}[One-step bootstrap approximation]\label{thm:B1Scorollary}
	Assume Assumptions \ref{ass:Basic}, \ref{ass:Rates3}, $\alpha \in (0,0.5)$, and that $H_{0}$ holds.  Then,
\begin{equation*}
P(T_{n}(\theta)>c_{n}^{B,1S}(\theta, \alpha ))~\le~ \alpha +\tilde{C}n^{-\tilde{c}},
\end{equation*}
where $\tilde{c},\tilde{C}>0$ are constants that only depend on the constants $c,C$ in Assumption \ref{ass:Rates3}. 
Furthermore, if $\mu(\theta)={\bf 0}_{p}$, then
\begin{equation*}
	\vert P( T_{n}(\theta)>c_{n}^{B,1S}(\theta, \alpha ))-\alpha  \vert ~\leq~ \tilde{C}n^{-\tilde{c}}.
\end{equation*}
\end{theorem}

As in the SN approximation method, the regular (one-step) bootstrap approximation considers all moment inequalities in the model as binding. A more powerful bootstrap-based test can be constructed using the data to reveal which moment inequalities are slack. However, unlike in the SN approximation method, Theorem \ref{thm:B1Scorollary} shows that the size of the test using the bootstrap critical values converges to $\alpha$ when all the moment inequalities are binding. This difference comes from the fact that the bootstrap can better approximate the correlation structure in the moment inequalities, which is not taken into account by the SN approximation. As we will see in simulations, this translates into power gains in favor of the bootstrap.

\citetalias{chernozhukov/chetverikov/kato:2014c} propose a two-step bootstrap procedure, combining a first-step moment inequality selection based on the bootstrap with the second step bootstrap critical value in Theorem \ref{thm:B1Scorollary}.\footnote{They also consider the so-called ``hybrid'' procedures in which the first step can be based on one approximation method (e.g.\ SN approximation) and the second step could be based on another approximation method (e.g.\ bootstrap). While these are not explicitly addressed in this section, they are included in the Monte Carlo section.} If we adapt their procedure to the presence of moment equalities, this would be given by:
\begin{align}
c^{B,2S}(\theta,\alpha) ~\equiv ~c_{n}^{B}(\theta, \hat{J}_{B}(\theta), \alpha - 2 \beta_{n})
\label{eq:B2S-CV}
\end{align}
with:
\begin{align*}
\hat{J}_{B}(\theta) ~\equiv ~ \{j\in \{1,\dots,p\}~:~\sqrt{n}\hat{\mu}_j(\theta)/\hat{\sigma}_j(\theta)
>-2c^{B,1S}(\alpha,\beta_n)\},
\end{align*}
where $\{\beta_n\}_{n\ge 1}$ is an arbitrary sequence of constants in $(0,\alpha/2)$. Again, by extending arguments in \citetalias{chernozhukov/chetverikov/kato:2014c} to the presence of moment equalities, one can show that an inference method based on the critical value $c^{B,2S}(\theta,\alpha)$ in Eq.\ \eqref{eq:B2S-CV} is asymptotically valid in a uniform sense.

This paper proposes an alternative two-step bootstrap procedure by using our Lasso-based first step. For $B \in \{MB,EB\}$, define the following two-step Lasso bootstrap critical value:
\begin{equation}
c_{n}^{B,L}(\theta, \alpha )~\equiv~ c_{n}^{B}(\theta, \hat{J}_{L}(\theta), \alpha) ,
\label{eq:cn-BLasso}
\end{equation}
where $\hat{J}_{L}(\theta)$ is as in Eq.\ \eqref{eq:BindingLasso2}, and $c_{n}^{B}(\theta,J,\alpha )$ is as in Algorithm \ref{alg:MB} if $B=MB$ or Algorithm \ref{alg:EB} if $B=EB$. The following result shows that an inference method based on our two-step Lasso bootstrap critical value is asymptotically valid in a uniform sense.

\begin{theorem}[Two-step Lasso bootstrap approximation]\label{thm:Lasso2SBootstrap}
Assume Assumptions \ref{ass:Basic}, \ref{ass:Rates3}, $\alpha \in (0,0.5)$, and that $H_{0}$ holds, and let $\lambda_n$ be as in Eq.\ \eqref{eq:LambdaConcrete}. Then, for $B\in \{MB,EB\}$, 
\begin{align*}
&P( T_n(\theta) > c_n^{B,L} (\theta , \alpha)) \leq\\
& \alpha + \check{C}n^{-\check{c}}+ \tilde{C} n^{-\tilde{c}} +2\tilde{K}n^{-c}+ 4\exp[\ln(2k-p)( 1-\tfrac{n^{c-(1-\delta )/2}}{2C}) ] [1+K( ( \tfrac{C}{n^{c-( 1-\delta ) /2}}) ^{1/( 2+\delta ) }+1) ^{2+\delta }] \to \alpha,
\end{align*}
where $\tilde{c}, \check{c},\tilde{K},\tilde{C}, \check{C}>0$ and $c>(1-\delta)/2$ are universal constants, and so the convergence is uniform in all parameters $\theta$ and distributions $P$ that satisfy the conditions of the statement. 
Furthermore, if $\mu(\theta) ={\bf 0}_{p}$ and
\begin{equation}
\tilde{K}n^{-c}+ 2\exp[\ln(2k-p)( 1-\tfrac{n^{c-(1-\delta )/2}}{2C}) ] [1+K( ( \tfrac{C}{n^{c-( 1-\delta ) /2}}) ^{1/( 2+\delta ) }+1) ^{2+\delta }]
	~\leq~ \tilde{C} n^{-\tilde{c}}\label{eq:RestrictionOnParams}
\end{equation}
then,
\begin{equation*}
|P(T_n(\theta) > c_n^{B,L} (\theta , \alpha)) - \alpha |~\leq~ 3 \tilde{C} n^{- \tilde{c}} + \check{C}n^{-\check{c}} \to 0,
\end{equation*}
where the convergence is uniform in all parameters $\theta$ and distributions $P$ that satisfy the conditions of the statement.
\end{theorem}

By repeating arguments at the end of Section \ref{sec:SN}, it follows that our two-step bootstrap method results in a larger rejection probability than the one-step bootstrap method for all sample sizes.\footnote{To establish this result, we now use Lemma \ref{lem:CvBootIncreasing} instead of Lemma \ref{lem:CSNincreasing}.} Also, the comparison between $c_{n}^{B,L}(\theta,\alpha)$ and $c_{n}^{B,2S}(\theta,\alpha)$ is not straightforward as these differ in the same two aspects described Section \ref{sec:SN}. This comparison will be the topic of the next section.

\section{Power comparison}\label{sec:Power}

\citetalias{chernozhukov/chetverikov/kato:2014c} show that all of their inference methods satisfy uniform asymptotic size control under appropriate assumptions. Theorems \ref{thm:Lasso2SSize} and \ref{thm:Lasso2SBootstrap} show that our Lasso-based two-step inference methods also satisfy uniform asymptotic size control under similar assumptions. Given these results, the natural next step is to compare these inference methods in terms of criteria related to power.

One possible such criterion is minimax optimality, i.e., the ability that a test has of rejecting departures from $H_0$ at the fastest possible rate (without losing uniform size control). \citetalias{chernozhukov/chetverikov/kato:2014c} show that all their proposed inference methods are asymptotically optimal in a minimax sense, even in the absence of any inequality selection (i.e.\ defined as in Theorems \ref{thm:SN1Scorollary} and \ref{thm:B1Scorollary} in the presence of moment equalities). Since our Lasso-based inequality selection can only reduce the number of binding moment inequalities (thus increasing rejection), we can also conclude that all of our two-step Lasso-based inference methods (SN, MB, and EB) are also asymptotically optimal in a minimax sense. In other words, minimax optimality is a desirable property that is satisfied by all tests under consideration and, thus, cannot be used as a criterion to distinguish between them.

Thus, we proceed to compare our Lasso-based inference procedures with those proposed by \citetalias{chernozhukov/chetverikov/kato:2014c} in terms of rejection rates. Since all inference methods share the test statistic $T_{n}(\theta)$, the power comparison depends exclusively on the critical values.

\subsection{Comparison with one-step methods}

As pointed out in previous sections, our Lasso-based two-step inference methods will have more or equal power than the corresponding one-step analogue, i.e., 
\begin{eqnarray*}
	P \del[1]{T_n(\theta) > c_{n}^{SN,L}(\theta,\alpha)} &\geq& P \del[1]{ T_n(\theta) > c_{n}^{SN,1S}(\alpha)} \\
	P \del[1]{T_n(\theta) > c_{n}^{B,L}(\theta,\alpha)} &\geq& P \del[1]{T_n(\theta) > c_{n}^{B,1S}(\theta,\alpha)}~~\forall B \in \{MB,EB\},
\end{eqnarray*}
for all $\theta \in \Theta$ and $n \in \mathbb{N}$. 
This is a direct consequence of the fact that one-step critical values are based on considering all moment inequalities as binding, while the Lasso-based first step will restrict attention to the subset of them that are sufficiently close to binding, i.e., $\hat{J}_{L}(\theta) \subseteq \{1,\dots,p\}$.

\subsection{Comparison with two-step methods}

The comparison between our two-step Lasso procedures and the two-step methods in \citetalias{chernozhukov/chetverikov/kato:2014c} is not straightforward for two reasons. First, the set of binding inequalities according to the Lasso might be different from the other methods. Second, our Lasso-based methods considers the usual $\alpha$ quantile while the other two-step methods consider the $\alpha - 2\beta_{n}$ quantile for a sequence of positive constants $\{ \beta_{n} \}_{ n\geq 1}$.

To simplify the discussion, we focus exclusively on the case where the moment (in)equality model is only composed of inequalities, i.e., $k=p$, which is precisely the setup in \citetalias{chernozhukov/chetverikov/kato:2014c}. This is done for simplicity of exposition as the introduction of moment equalities would not qualitatively change the conclusions that follow.

We begin by comparing the power of the two-step SN method with the two-step Lasso SN method. For all $\theta \in \Theta$ and $n \in \mathbb{N}$, our two-step Lasso SN method will have greater or equal power than the two-step SN method if and only if $ c_{n}^{SN,L}(\theta,\alpha)~\leq~ c_{n}^{SN,2S}(\alpha)$. By inspecting the formulas in \citetalias{chernozhukov/chetverikov/kato:2014c}, this occurs if and only if
\begin{equation}
	|\hat{J}_L(\theta)|~\leq~ \frac{\alpha}{\alpha-2\beta_n}|\hat{J}_{SN}(\theta)|,
	 \label{eq:PowerAdvSN2}
\end{equation}
where, by definition, $\{ \beta_{n} \}_{ n\geq 1}$ satisfies $\beta_n\in ( 0,\alpha/2)$ (see \citetalias{chernozhukov/chetverikov/kato:2014c} (page 15)). We provide sufficient conditions for Eq.\ \eqref{eq:PowerAdvSN2} in the following result.

\begin{theorem}\label{thm:powercompSN} 
	\underline{Part 1:} For all $\theta \in \Theta$ and $n \in \mathbb{N}$,
	\begin{equation}
		\hat{J}_L(\theta)~~\subseteq~~ \hat{J}_{SN}(\theta) 
		\label{eq:JComparisonSN}
	\end{equation}
	implies
	\begin{equation}
		P(T_n(\theta) > c_{n}^{SN,L}(\theta,\alpha)) ~~\geq~~ P (T_n(\theta) > c_{n}^{SN,2S}(\alpha)).
		\label{eq:PowerComparisonSN}
	\end{equation}
	\underline{Part 2:} In turn, Eq.\ \eqref{eq:JComparisonSN} is implied by
		\begin{align}
\Phi ^{-1}( 1-\beta _{n}/p)\sqrt{ {M}_{n,2+\delta }^{2}n^{-\delta /(2+\delta )}-n^{-1} }~\geq~ (1+3\varepsilon/4 )\sqrt{1-\left( \Phi ^{-1}( 1-\beta _{n}/p) \right) ^{2}/n},\label{eq:highlevel}
	\end{align}
	where $\varepsilon>0$ is as defined in Eq.\ \eqref{eq:LambdaConcrete}.
\end{theorem}

Theorem \ref{thm:powercompSN} provides two sufficient conditions (i.e., \eqref{eq:JComparisonSN} or \eqref{eq:highlevel}) under which our two-step Lasso SN method has greater or equal power than the two-step SN method in \citetalias{chernozhukov/chetverikov/kato:2014c} (i.e., \eqref{eq:PowerComparisonSN}). By simple algebraic manipulations, we can equivalently reexpress Eq.\ \eqref{eq:highlevel} as follows:
\begin{equation}
\beta _{n}/p~\leq~ 1-\Phi \left( (1+3\varepsilon /4)\left( {{M} _{n,2+\delta }^{2}}{n^{-\delta /(2+\delta )}}+{( 3\varepsilon /2+9\varepsilon ^{2}/16) }/{n}\right) ^{-1/2}\right).
\label{eq:highlevel2}
\end{equation}
In this sense, Eq.\ \eqref{eq:highlevel} can be interpreted as imposing an upper bound on the sequence $\{ \beta _{n}/p\} _{n\geq 1}$.

It is worth pointing out that the power comparison in Theorem \ref{thm:powercompSN} is a finite sample result. In other words, under either one of the sufficient conditions of Theorem \ref{thm:powercompSN}, the rejection of $H_0$ by an inference method with SN-based first step also implies the rejection of $H_0$ by the corresponding inference method with Lasso-based first step. This result can be expressed in terms of confidence sets. Under the sufficient conditions of Theorem \ref{thm:powercompSN}, the confidence set based on an inference method with our Lasso-based first step will be a subset of the confidence set based on the corresponding inference method with an SN-based first step.\footnote{While Theorem \ref{thm:powercompSN} is a finite sample result, it might be relevant to understand its implications as $n\to \infty$. Note that the assumptions in Section \ref{sec:Assumptions} do not impose restrictions on the sequence $ \{ \beta _{n}\} _{n\geq 1}$ used to implement the SN-based first step. Therefore, given any sequence of parameters of the inference problem that satisfy our assumptions, we can always find a sequence $\{ \beta _{n}\} _{n\geq 1}$ with $\beta _{n}\downarrow 0$ that satisfies Eq.\ \eqref{eq:highlevel} (or, equivalently, Eq.\ \eqref{eq:highlevel2}). On the other hand, if we were to replace Assumption \ref{ass:Rates} with Eq.\ (27) in \citetalias{chernozhukov/chetverikov/kato:2014c}, it is possible to show that Eq.\ \eqref{eq:highlevel} (or, equivalently, Eq.\ \eqref{eq:highlevel2}) will fail to hold for all sufficiently large sample sizes.}

In principle, Theorem \ref{thm:powercompSN} allows for the possibility of the inequality in Eq.\ \eqref{eq:PowerComparisonSN} being an equality. However, in cases in which the Lasso-based first step selects a strict subset of the moment inequalities chosen by the SN method (i.e., the inclusion in Eq.\ \eqref{eq:JComparisonSN} is strict), the inequality in Eq.\ \eqref{eq:PowerComparisonSN} can be strict. In fact, the inequality in Eq.\ \eqref{eq:PowerComparisonSN} can be strict even in cases in which the Lasso-based and SN-based first step agree on the set of binding moment inequalities. The intuition for this is that our Lasso-based method considers the usual $\alpha$-quantile while the other two-step methods consider the $(\alpha - 2\beta_{n})$-quantile for the sequence of positive constants $\{ \beta_{n} \}_{ n\geq 1}$. This slight difference always plays in favor of the Lasso-based first step having more power.\footnote{This is clearly shown in Designs 5-6 of our Monte Carlos. In these cases, both first-step methods to agree on the correct set of binding moment inequalities (i.e.\ $\hat{J}_L(\theta)=\hat{J}_{SN}(\theta)$). Nevertheless, the slight difference in quantiles produced a small but positive power advantage in favor of the Lasso-based first step.}

The relevance of Theorem \ref{thm:powercompSN} depends on the generality of the sufficient conditions in that result. To illustrate this result, Figure \ref{fig:SN} depicts combinations of $M_{n,\delta}$ and $p$ under which Eq.\ \eqref{eq:highlevel} fails to hold when $n=400$, $\beta_n = 0.1\%$, $\varepsilon =2/3$, and $\delta = 1$. This graph shows that Eq.\ \eqref{eq:highlevel} is satisfied for a large section of the points of the parameter space. 

\begin{figure}
\begin{center}
\includegraphics[scale=0.7]{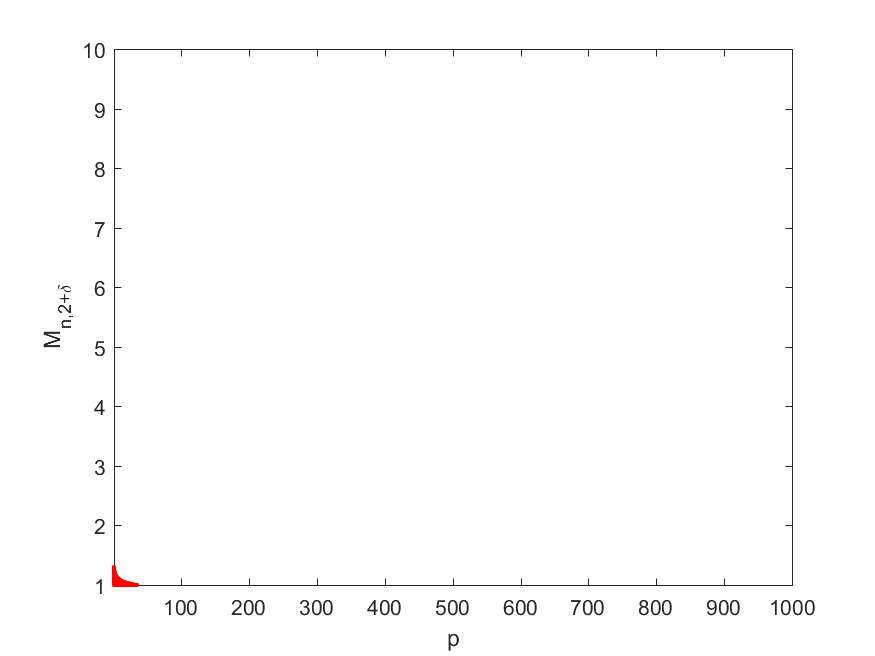}
\caption{\small Marked in red color, combinations $M_{n,2+\delta} \in [1,10]$ and $p\in \{1,\dots,n\}$ such that Eq.\ \eqref{eq:highlevel} does not hold in the case with $n=400$, $\beta_n = 0.1\%$, $\varepsilon =2/3$, and $\delta = 1$. These parameters are used in the simulations in Section \ref{sec:MonteCarlos}.}
\label{fig:SN}
\end{center}
\end{figure}


To conclude the section, we now compare the power of the two-step bootstrap procedures with the two-step Lasso SN method.
\begin{theorem}\label{thm:power_compB}
	Assume Assumption \ref{ass:Rates3} and let $B \in \{MB,EB\}$. 
	
	\underline{Part 1:} For all $\theta \in \Theta$ and $n \in \mathbb{N}$, 
	\begin{align}
\hat{J}_L(\theta) ~\subseteq~ \hat{J}_{B}(\theta)
		\label{eq:JComparisonB}
	\end{align}
	implies
	\begin{align}
		P (T_n(\theta) > c_{n}^{B,2S}(\alpha)) ~\leq~ P (T_n(\theta) > c_{n}^{B,L}(\theta,\alpha)).
		\label{eq:PowerComparisonBpre}
	\end{align}
	
	\underline{Part 2:} Eq.\ \eqref{eq:JComparisonB} occurs with probability approaching one, i.e., for ${C},{c}>0$ as in Assumption \ref{ass:Rates3},
	\begin{equation}
		P(\hat{J}_L(\theta)\subseteq \hat{J}_{B}(\theta)) ~\geq~ 1-{C}n^{-{c}} \label{eq:JComparisonBstock}
	\end{equation}
under the following sufficient conditions: (i) $\beta_n\geq {{C}}n^{-{c}}$ and (ii) either 
\begin{align}
&1-\Phi ( ( 1+3\varepsilon /4) ( M_{n,2+\delta
}^{2}n^{-\delta /( 2+\delta ) }-n^{-1}) ^{-1/2}) \geq 3\beta _{n}\label{eq:2Ssuff}
~~\text{ or,}\\
&\sqrt{( 1-\rho ( \theta ) ) \ln (p)/2}-\sqrt{2\ln
( 1/( 1-3\beta _{n}) ) }\geq ( 1+3\varepsilon
/4) ( M_{n,2+\delta }^{2}n^{-\delta /( 2+\delta )
}-n^{-1}) ^{-1/2},
\label{eq:2Ssuff2}
\end{align}
where $\rho(\theta) \equiv \max_{j_1\neq j_2}corr[X_{j_1}(\theta),X_{j_2}(\theta)]$.
	
	\underline{Part 3:} Under the sufficient conditions in part 2,
	\begin{equation}
		P(T_n(\theta) > c_{n}^{B,2S}(\alpha)) ~\leq~ P(T_n(\theta) > c_{n}^{B,L}(\theta,\alpha))+ {C}n^{-{c}}. \label{eq:PowerComparisonB}
	\end{equation}
\end{theorem}

Theorem \ref{thm:power_compB} provides sufficient conditions under which any power advantage of the two-step bootstrap method in \citetalias{chernozhukov/chetverikov/kato:2014c} relative to our two-step bootstrap Lasso vanishes as the sample size diverges to infinity. Specifically, Eq.\ \eqref{eq:PowerComparisonB} indicates that, under any of the sufficient conditions, this power advantage does not exceed ${C}n^{-{c}}$. As in the SN approximation, this relative power difference is a direct consequence of Eq.\ \eqref{eq:JComparisonB}, i.e., our Lasso-based first step inequality selection procedure chooses a subset of the inequalities selected by the bootstrap-based first step. 

The relevance of the result in Theorem \ref{thm:power_compB} depends on the generality of the sufficient conditions in part 2. This condition has two requirements. The first one, i.e., $\beta_n\geq{C}n^{-{c}}$ is considered mild as $\{\beta_n\}_{n\ge 1}$ is a sequence of positive constants and $ {C}n^{-{c}}$ converges to zero. The second requirement is either Eq.\ \eqref{eq:2Ssuff} or \eqref{eq:2Ssuff2}. The latter one can be understood as imposing an upper bound on the maximal pairwise correlation within the moment inequalities of the model.

\section{Monte Carlo simulations}\label{sec:MonteCarlos}

We now use Monte Carlo simulations to investigate the finite sample properties of our tests and to compare them to those proposed by \citetalias{chernozhukov/chetverikov/kato:2014c}. Our Monte Carlo simulation setup follows closely the one used in \cite{chernozhukov/chetverikov/kato:2014c_v4}, which we describe next. For a hypothetical fixed parameter value $\theta \in \Theta$, the data satisfy
\begin{align*}
X_{i}(\theta)~=~\mu(\theta) + A'\epsilon_{i}~~~~~i=1,\dots,n = 400,
\end{align*}
where $\Sigma(\theta)=A'A$, $\epsilon_i=(\epsilon_{i,1},\dots, \epsilon_{i,p})$, and $p \in \{ 200,~500,~1,000\}$. We simulate $\{\epsilon_i\}_{i=1}^{n}$ i.i.d.\ with $E[\epsilon_i]= {\bf 0}_{p}$ and $Var[\epsilon_i]=\mathbf{I}_{p\times p}$. Thus, $\{X_{i}(\theta)\}_{i=1}^{n}$ are i.i.d.\ with $E[X_{i}(\theta)]=\mu(\theta)$ and $Var[X_{i}(\theta)]=\Sigma(\theta)$. This model satisfies the moment (in)equality model in Eq.\ \eqref{eq:MI} if and only if $\mu(\theta)\leq {\bf 0}_{p}$. In this context, we are interested in implementing the hypothesis test in Eqs.\ \eqref{eq:HypTest} (or, equivalently, Eq.\ \eqref{eq:HypTest2a} or Eq.\ \eqref{eq:HypTest2}) with a significance level of $\alpha = 5\%$.

We simulate $\epsilon_i =(\epsilon_{i,1},\dots, \epsilon_{i,p})$ to be i.i.d.\ according to two distributions: (i) $\epsilon_{i,j}$ follows a $t$-distribution with four degrees of freedom divided by $\sqrt{2}$, i.e., $ \epsilon_{i,j} \sim t_4/\sqrt{2}$ and (ii) $\epsilon_{i,j} \sim U(-\sqrt{3},\sqrt{3})$. Note that both of these choices satisfy $E[\epsilon_i]= {\bf 0}_{p}$ and $Var[\epsilon_i]=\mathbf{I}_{p\times p}$. Since $(\epsilon_{i,1},\dots, \epsilon_{i,p})$ are i.i.d., the correlation structure across moment inequalities depends entirely on $\Sigma(\theta)$, for which we consider two possibilities: (i) $\Sigma(\theta)_{[j,k]}=1[j=k]+\rho \times 1[j\neq k]$ and (ii) a Toeplitz structure, i.e., $\Sigma(\theta)_{[j,k]}=\rho^{|j-k|}$ with $\rho \in \{ 0,0.5,0.9\}$. We repeat all experiments $2,000$ times.

The description of the model is completed by specifying $\mu(\theta)$, given in Table \ref{tab:ParameterChoices}. We consider ten different specifications of $\mu(\theta)$ which, in combination with the rest of the parameters, results in fourteen simulation designs. Our first eight simulation designs correspond exactly to those in \cite{chernozhukov/chetverikov/kato:2014c_v4}, half of which satisfy $H_0$ and half of which do not. We complement these simulations with six designs that do not satisfy $H_0$. The additional designs are constructed so that the moment inequalities that satisfy $H_0$ are only slightly or moderately negative.\footnote{For reasons of brevity, these additional designs only consider $\Sigma(\theta)$ with a Toeplitz structure. We carried out the same designs with equicorrelated $\Sigma(\theta)$ and obtained qualitatively similar results. These are available from the authors, upon request.} As the slackness of these inequalities decreases, it becomes harder for two-step inference methods to correctly classify the non-binding moment conditions as such. As a consequence, these new designs will help us understand which two-step inference procedures have better abilities in detecting slack moment inequalities.

\begin{table}[h]
	\begin{center}
	\scalebox{0.85}{\begin{tabular}{ccccc}
		\hline\hline
		Design no. & $\{\mu_{j}(\theta):j \in \{1,\dots,p\}\}$ & $\Sigma(\theta)$ & Hypothesis & Design in \cite{chernozhukov/chetverikov/kato:2014c_v4}\\
		\hline
		1 &$-0.8\times 1[j>0.1p]$ & Equicorrelated & $H_0$ & 2\\
		2 &$-0.8\times 1[j>0.1p]$ & Toeplitz & $H_0$ & 4\\
		3 &$0$ & Equicorrelated & $H_0$ & 1\\
		4 &$0$ & Toeplitz & $H_0$ & 3\\
		5 &$0.05$ & Equicorrelated & $H_1$ & 5\\
		6 &$0.05$ & Toeplitz & $H_1$ & 7\\
		7 &$-0.75\times 1[j>0.1p]+0.05\times 1[j\leq 0.1p]$ & Equicorrelated & $H_1$ & 6\\
		8 &$-0.75\times 1[j>0.1p]+0.05\times 1[j\leq 0.1p]$ & Toeplitz & $H_1$ & 8\\
		9 &$-0.6\times 1[j>0.1p]+0.05\times 1[j\leq 0.1p]$ & Toeplitz & $H_1$ & New\\
		10 &$-0.5\times 1[j>0.1p]+0.05\times 1[j\leq 0.1p]$ & Toeplitz & $H_1$ & New\\
		11 &$-0.4\times 1[j>0.1p]+0.05\times 1[j\leq 0.1p]$ & Toeplitz & $H_1$ & New\\
		12 &$-0.3\times 1[j>0.1p]+0.05\times 1[j\leq 0.1p]$ & Toeplitz & $H_1$ & New\\
		13 &$-0.2\times 1[j>0.1p]+0.05\times 1[j\leq 0.1p]$ & Toeplitz & $H_1$ & New\\
		14 &$-0.1\times 1[j>0.1p]+0.05\times 1[j\leq 0.1p]$ & Toeplitz & $H_1$ & New\\
	\hline\hline
	\end{tabular}}
	\end{center}
	\caption{Parameter choices in our simulations.}
	\label{tab:ParameterChoices}
\end{table}

We implement all the inference methods described in Table \ref{tab:InfMethods}. These include all of the procedures described in previous sections and some additional ``hybrid'' methods (i.e.\ MB-H and EB-H). The bootstrap based methods are implemented with $B=1,000$ bootstrap replications. Finally, for our Lasso-based first step, we use
\begin{align}
{\lambda}_n~=~(4/3+\varepsilon) n^{-1/2}\left( \hat{M}_{n,3}^2n^{-1/3}-n^{-1} \right)^{-1/2},
\label{eq:LambdaMCs}
\end{align}
with $\varepsilon\in\{2/3, 8/3\}$ and $\hat{M}_{n,3}$ as defined in Eq.\ \eqref{eq:Mconcrete} but with singleton $\Theta$. This corresponds to the sample analogue of Eq.\ \eqref{eq:LambdaConcrete} when $\delta=1$.

\begin{table}
	\begin{center}
	\scalebox{0.9}{
	\begin{tabular}{lcccc}
		\hline\hline 
		Method & No.\ of steps & First step & Second step & Parameters\\
		\hline
		SN Lasso & Two & Lasso & SN  &   $\varepsilon\in\{2/3, 8/3\}$ in  Eq.\ \eqref{eq:LambdaMCs}  \\
		MB Lasso & Two & Lasso & MB  & $\varepsilon\in\{2/3, 8/3\}$ in Eq.\ \eqref{eq:LambdaMCs}  \\
		EB Lasso & Two & Lasso & EB  & $\varepsilon\in\{2/3, 8/3\}$ in Eq.\ \eqref{eq:LambdaMCs} \\
		SN-1S & One & None & SN & None \\
		SN-2S & Two & SN & SN & None \\
		MB-1S & One & None & MB & None \\
		MB-H  & Two & SN & MB  & $\beta_{n} \in \{0.01\%,0.1\%,1\%\}$ \\
		MB-2S & Two & MB & MB & $\beta_{n} \in \{0.01\%,0.1\%,1\%\}$\\
		EB-1S & One & None & EB & None \\
		EB-H  & Two & SN & EB  & $\beta_{n} \in \{0.01\%,0.1\%,1\%\}$ \\
		EB-2S & Two & EB & EB & $\beta_{n} \in \{0.01\%,0.1\%,1\%\}$\\
		\hline\hline 
	\end{tabular}}
	\end{center}
	\caption{Inference methods implemented in our simulations.}
	\label{tab:InfMethods}
\end{table}

We shall begin by considering the simulation designs in \citetalias{chernozhukov/chetverikov/kato:2014c} as reported in Tables \ref{tab1}-\ref{tab8}. The first four tables are concerned with the finite sample size control. The general finding is that all tests under consideration are very rarely over-sized. The maximal size observed for our procedures is 7.15 (e.g.\ EB Lasso in Designs 3-4, $p=1,000$, $\rho=0$, and uniform errors) while the corresponding number for \citetalias{chernozhukov/chetverikov/kato:2014c} is 7.25 (e.g.\ EB-1S, EB-H with $\beta=0.01$, EB-2S with $\beta=0.01$ in Designs 3-4, $p=1,000$, $\rho=0$, and uniform errors). Some procedures, such as SN-1S, can be heavily under-sized. Our simulations reveal that in order to achieve empirical rejection rates close to $\alpha = 5\%$ under $H_0$, one requires using a two-step inference procedure with a bootstrap-based second step (either MB or EB). 

Before turning to the individual setups for power comparison, let us remark that a first step based on our Lasso procedure compares favorably with a first step based on SN. For example, SN-Lasso with $\varepsilon = 2/3$ has more or equal power than SN-2S with $\beta_n=0.1\%$. While the differences may often be small, this finding is in line with the power comparison in Section \ref{sec:Power}.

Tables \ref{tab5}-\ref{tab8} contain the designs used by \citetalias{chernozhukov/chetverikov/kato:2014c} to gauge the power of their tests. Tables \ref{tab5} and \ref{tab6} consider the case where all moment inequalities are violated. Since none of the moment conditions are slack, there is no room for power gains based on a first-step inequality selection procedure. In this sense, it is not surprising that the first-step choice makes no difference in these designs. For example, the power of SN-Lasso is identical to the one of SN-1S while the power of SN-2S is also close to the one of SN-1S. However, the SN-2S has lower power than SN-1S for some values of $\beta_{n}$ while the power of SN Lasso appears to be invariant to the choice of $\varepsilon$. The latter is in accordance with our previous findings. The bootstrap still improves power for high values of $\rho$.

Next, we consider Tables \ref{tab7} and \ref{tab8}. In this setting, $90\%$ of the moment conditions have $\mu_j(\theta)=-0.75$ and our results seem to suggest that this value is relative far away from being binding. We deduce this from the fact that all first-step selection methods agree on the set of binding moment conditions, producing very similar power results. Table \ref{tab15} shows the percentage of moment inequalities retained by each of the first-step procedures in Design 8. When the error terms are $t$-distributed, all first-step procedures retain around $10\%$ of the inequalities which is also the fraction that are truly binding (and, in this case, violated). Thus, all two-step inference procedures are reasonably powerful. When the error terms are uniformly distributed, all first-step procedures have an equal tendency to aggressively remove slack inequalities. However, we have seen from the size comparisons that this does not seem to result in oversized tests. Finally, we notice that the power of our procedures hardly varies with the choice of $\varepsilon$.

The overall message of the simulation results in Designs 1-8 is that our Lasso-based procedures are comparable in terms of size and power to the ones proposed by \citetalias{chernozhukov/chetverikov/kato:2014c}.

Tables \ref{tab9}-\ref{tab14} present simulations results for Designs 9-14. These correspond to modifications of the setup in Design 8 in which progressively decrease the degree of slackness of the non-binding moment inequalities from $-0.75$ to values between $-0.6$ and $-0.1$.

Tables \ref{tab9}-\ref{tab10} shows results for Designs 9 and 10. As in the case of Design 8, the degree of slackness of the non-binding moment inequalities is still large enough so that it can be correctly detected by all first first-step selection methods.  

As Table \ref{tab11} shows, this pattern changes in Design 11. In this case, the MB Lasso with $\varepsilon= 2/3$ has a rejection rate that is at least 20 percentage points higher than the most powerful procedure in \citetalias{chernozhukov/chetverikov/kato:2014c}. For example, with $t$-distributed errors, $p=1,000$, and $\rho=0$, our MB Lasso with $\varepsilon= 2/3$ has a rejection rate of 71.19\% whereas the MB-2S with  $\beta_n = 0.01\%$ has a rejection rate of 20.77\%. Table \ref{tab16} holds the key to these power differences. Ideally, a powerful procedure should retain only the $10\%$ of the moment inequalities that are binding (in this case, violated). The Lasso-based selection indeed often retains close to $10\%$ of the inequalities for $\varepsilon \in \{2/3,8/3\}$. On the other hand, SN-based selection can sometimes retain more than $90\%$ of the inequalities (e.g.\ see $t$-distributed errors, $p=1,000$, and $\rho=0$).

The power advantage in favor of the Lasso-based first step is also present in Design 12 as shown in Table \ref{tab12}. In this case, the MB Lasso with $\varepsilon= 2/3$ has a rejection rate which is at least $10$ percentage points higher than the most powerful procedure in \citetalias{chernozhukov/chetverikov/kato:2014c}. The MB Lasso has rejection rates up $50$ percentage points higher than its competitors (e.g.\ $p=1,000$ and $\rho=0$). As in the previous design, this power gain mainly comes from the Lasso being better at removing the slack moment conditions.

Table \ref{tab13} shows the results for Design 13. Here the degree of slackness of the non-binding moments is getting so small that power advantages of the Lasso-based procedures are small, yet still present.

Design 14 is our last experiment and it is shown in Table \ref{tab14}. In this case, the degree of slackness of the non-binding moment inequalities is so small that it cannot be detected by any of the first-step selection methods. As a consequence, there are very little differences among the various inference procedures and all of them exhibit relatively low power.

The overall message from Tables \ref{tab9}-\ref{tab14} is that our Lasso-based inference procedures can have higher power than those in \citetalias{chernozhukov/chetverikov/kato:2014c} when the slack moment inequalities are difficult to distinguish from zero.

\section{Conclusions}\label{sec:Conclusions}

This paper considers the problem of inference in a partially identified moment (in)equality model with possibly many moment inequalities. We contribute to this literature by proposing new critical values that are the result of combining the approximation methods in \citetalias{chernozhukov/chetverikov/kato:2014c} (i.e.\ self-normalization, multiplier bootstrap, or empirical bootstrap), with a novel first-step moment inequality selection procedure based on the Lasso. Besides the proposing a different first-step moment selection procedure, our inference method uses a second step that can ignore the presence of the first-step moment inequality selection, thus increasing statistical power. We refer to the resulting hypothesis test as a two-step Lasso-based inference methods. Our two-step inference methods can be used to conduct hypothesis tests and to construct confidence sets for the true parameter value.

Our inference method has very desirable properties. First, under reasonable conditions, it is asymptotically uniformly valid, both in the underlying parameter $\theta $ and in the distribution of the data. Second, by virtue of results in \citetalias{chernozhukov/chetverikov/kato:2014c}, our test is asymptotically optimal in a minimax sense. Third, the power of our method compares favorably with that of the corresponding two-step method in \citetalias{chernozhukov/chetverikov/kato:2014c}, both in theory and in simulations. On the theory front, we provide sufficient conditions under which the power of our method dominates. These can sometimes represent a significant part of the parameter space. Our simulations indicate that our inference methods are usually as powerful as the corresponding ones in \citetalias{chernozhukov/chetverikov/kato:2014c}, and can sometimes be more powerful. Fourth, our Lasso-based first step is straightforward to implement.

{\tiny

\begin{landscape}
	\begin{table}
\hspace{-0.5cm}\scalebox{0.64}{
\begin{tabular}{ccccccccccccccccccccccccccc}
\toprule Density & $p$ & $\rho$ & \multicolumn{6}{c}{Our methods} & \multicolumn{18}{c}{\citetalias{chernozhukov/chetverikov/kato:2014c}'s methods} \\
\cmidrule(lr){4-9}\cmidrule(lr){10-27}
 & & & \multicolumn{2}{c}{SN Lasso} & \multicolumn{2}{c}{MB Lasso} & \multicolumn{2}{c}{EB Lasso} & SN-1S & \multicolumn{3}{c}{SN-2S} & MB-1S & \multicolumn{3}{c}{MB-H} & \multicolumn{3}{c}{MB-2S} & EB-1S & \multicolumn{3}{c}{EB-H} & \multicolumn{3}{c}{EB-2S} \\
\cmidrule(lr){4-5}\cmidrule(lr){6-7}\cmidrule(lr){8-9}\cmidrule(lr){11-13}\cmidrule(lr){15-17}\cmidrule(lr){18-20}\cmidrule(lr){22-24}\cmidrule(lr){25-27}
 & & & $\varepsilon=2/3$ & $\varepsilon=8/3$ & $\varepsilon=2/3$ & $\varepsilon=8/3$ & $\varepsilon=2/3$ & $\varepsilon=8/3$ & & 0.01\% & 0.10\% & 1.00\% & & 0.01\% & 0.10\% & 1.00\% & 0.01\% & 0.10\% & 1.00\% & & 0.01\% & 0.10\% & 1.00\% & 0.01\% & 0.10\% & 1.00\% \\
\midrule 
$t_{4}/\sqrt{2}$ & 200 & 0.00 & 4.00 & 4.00 & 5.00 & 5.00 & 4.50 & 4.45 & 0.40 & 3.75 & 3.95 & 2.30 & 0.55 & 4.70 & 4.80 & 2.50 & 4.95 & 4.85 & 2.50 & 0.50 & 4.15 & 4.35 & 2.65 & 4.75 & 4.75 & 2.70 \\
 & & 0.50 & 3.05 & 3.05 & 5.10 & 5.10 & 4.85 & 4.85 & 0.25 & 2.95 & 3.05 & 1.85 & 0.85 & 5.05 & 5.05 & 2.80 & 5.05 & 5.05 & 2.85 & 0.85 & 4.60 & 4.85 & 3.10 & 4.75 & 4.75 & 2.85 \\
 & & 0.90 & 0.90 & 0.90 & 5.00 & 5.00 & 5.15 & 5.15 & 0.05 & 0.90 & 0.90 & 0.45 & 2.80 & 4.95 & 4.95 & 3.10 & 4.95 & 4.95 & 3.10 & 2.80 & 5.20 & 5.15 & 2.95 & 5.25 & 5.10 & 3.10 \\
 & 500 & 0.00 & 4.35 & 4.35 & 5.45 & 5.45 & 4.85 & 4.80 & 0.50 & 4.30 & 4.30 & 2.85 & 0.65 & 5.10 & 5.25 & 3.40 & 5.35 & 5.30 & 3.40 & 0.45 & 4.40 & 4.55 & 3.10 & 4.85 & 4.85 & 3.15 \\
 & & 0.50 & 2.75 & 2.75 & 5.70 & 5.70 & 5.60 & 5.60 & 0.35 & 2.70 & 2.70 & 1.60 & 1.30 & 5.60 & 5.65 & 3.50 & 5.65 & 5.65 & 3.50 & 1.20 & 5.25 & 5.45 & 3.25 & 5.65 & 5.60 & 3.35 \\
\bottomrule
\end{tabular}
}
\caption{Simulation results in Design 1: $\mu_{j}(\theta)=-0.8\cdot 1[j>0.1 p]$, $\Sigma(\theta)$ equicorrelated}
\label{tab1}
\end{table}

\begin{table}
\hspace{-0.5cm}\scalebox{0.64}{
\begin{tabular}{ccccccccccccccccccccccccccc}
\toprule Density & $p$ & $\rho$ & \multicolumn{6}{c}{Our methods} & \multicolumn{18}{c}{\citetalias{chernozhukov/chetverikov/kato:2014c}'s methods} \\
\cmidrule(lr){4-9}\cmidrule(lr){10-27}
 & & & \multicolumn{2}{c}{SN Lasso} & \multicolumn{2}{c}{MB Lasso} & \multicolumn{2}{c}{EB Lasso} & SN-1S & \multicolumn{3}{c}{SN-2S} & MB-1S & \multicolumn{3}{c}{MB-H} & \multicolumn{3}{c}{MB-2S} & EB-1S & \multicolumn{3}{c}{EB-H} & \multicolumn{3}{c}{EB-2S} \\
\cmidrule(lr){4-5}\cmidrule(lr){6-7}\cmidrule(lr){8-9}\cmidrule(lr){11-13}\cmidrule(lr){15-17}\cmidrule(lr){18-20}\cmidrule(lr){22-24}\cmidrule(lr){25-27}
 & & & $\varepsilon=2/3$ & $\varepsilon=8/3$ & $\varepsilon=2/3$ & $\varepsilon=8/3$ & $\varepsilon=2/3$ & $\varepsilon=8/3$ & & 0.01\% & 0.10\% & 1.00\% & & 0.01\% & 0.10\% & 1.00\% & 0.01\% & 0.10\% & 1.00\% & & 0.01\% & 0.10\% & 1.00\% & 0.01\% & 0.10\% & 1.00\% \\
\midrule
 $t_{4}/\sqrt{2}$ & 200 & 0.00 & 4.00 & 4.00 & 5.00 & 5.00 & 4.50 & 4.45 & 0.40 & 3.75 & 3.95 & 2.30 & 0.55 & 4.70 & 4.80 & 2.50 & 4.95 & 4.85 & 2.50 & 0.50 & 4.15 & 4.35 & 2.65 & 4.75 & 4.75 & 2.70 \\
  & & 0.50 & 4.40 & 4.40 & 5.20 & 5.20 & 5.00 & 5.00 & 0.35 & 4.30 & 4.40 & 2.70 & 0.50 & 5.00 & 5.05 & 3.30 & 5.15 & 5.10 & 3.30 & 0.45 & 4.70 & 4.90 & 3.05 & 5.15 & 5.15 & 3.20 \\
  & & 0.90 & 2.30 & 2.30 & 5.20 & 5.20 & 4.95 & 4.95 & 0.35 & 2.30 & 2.30 & 1.55 & 0.70 & 5.15 & 5.10 & 3.35 & 5.20 & 5.15 & 3.35 & 0.65 & 4.85 & 4.85 & 3.35 & 5.30 & 5.30 & 3.10 \\
  & 500 & 0.00 & 4.35 & 4.35 & 5.45 & 5.45 & 4.85 & 4.80 & 0.50 & 4.30 & 4.30 & 2.85 & 0.65 & 5.10 & 5.25 & 3.40 & 5.35 & 5.30 & 3.40 & 0.45 & 4.40 & 4.55 & 3.10 & 4.85 & 4.85 & 3.15 \\
  & & 0.50 & 4.25 & 4.25 & 5.45 & 5.45 & 5.10 & 5.05 & 0.40 & 4.20 & 4.20 & 2.70 & 0.45 & 5.20 & 5.35 & 3.25 & 5.40 & 5.40 & 3.25 & 0.45 & 4.65 & 4.85 & 3.00 & 4.80 & 4.75 & 3.05 \\
\bottomrule
\end{tabular}
}
\caption{Simulation results in Design 2: $\mu_{j}(\theta)=-0.8\cdot 1[j>0.1 p]$, $\Sigma(\theta)$ Toeplitz}
\label{tab2}
\end{table}

\begin{table}
\hspace{-0.5cm}\scalebox{0.64}{
\begin{tabular}{ccccccccccccccccccccccccccc}
\toprule Density & $p$ & $\rho$ & \multicolumn{6}{c}{Our methods} & \multicolumn{18}{c}{\citetalias{chernozhukov/chetverikov/kato:2014c}'s methods} \\
\cmidrule(lr){4-9}\cmidrule(lr){10-27}
 & & & \multicolumn{2}{c}{SN Lasso} & \multicolumn{2}{c}{MB Lasso} & \multicolumn{2}{c}{EB Lasso} & SN-1S & \multicolumn{3}{c}{SN-2S} & MB-1S & \multicolumn{3}{c}{MB-H} & \multicolumn{3}{c}{MB-2S} & EB-1S & \multicolumn{3}{c}{EB-H} & \multicolumn{3}{c}{EB-2S} \\
\cmidrule(lr){4-5}\cmidrule(lr){6-7}\cmidrule(lr){8-9}\cmidrule(lr){11-13}\cmidrule(lr){15-17}\cmidrule(lr){18-20}\cmidrule(lr){22-24}\cmidrule(lr){25-27}
 & & & $\varepsilon=2/3$ & $\varepsilon=8/3$ & $\varepsilon=2/3$ & $\varepsilon=8/3$ & $\varepsilon=2/3$ & $\varepsilon=8/3$ & & 0.01\% & 0.10\% & 1.00\% & & 0.01\% & 0.10\% & 1.00\% & 0.01\% & 0.10\% & 1.00\% & & 0.01\% & 0.10\% & 1.00\% & 0.01\% & 0.10\% & 1.00\% \\
 \midrule
$t_{4}/\sqrt{2}$ & 200 & 0.00 & 3.05 & 3.05 & 4.30 & 4.30 & 3.55 & 3.55 & 3.05 & 3.05 & 3.05 & 1.55 & 4.30 & 4.30 & 4.25 & 2.15 & 4.30 & 4.25 & 2.15 & 3.60 & 3.60 & 3.60 & 2.00 & 3.45 & 3.45 & 1.95 \\
 & & 0.50 & 1.45 & 1.45 & 4.60 & 4.60 & 4.15 & 4.15 & 1.45 & 1.45 & 1.45 & 1.00 & 4.60 & 4.60 & 4.60 & 2.60 & 4.60 & 4.60 & 2.60 & 4.20 & 4.20 & 4.20 & 2.55 & 4.45 & 4.35 & 2.45 \\
 & & 0.90 & 0.25 & 0.25 & 5.20 & 5.20 & 5.30 & 5.30 & 0.25 & 0.25 & 0.25 & 0.15 & 5.20 & 5.20 & 5.15 & 3.10 & 5.20 & 5.15 & 3.10 & 5.30 & 5.30 & 5.30 & 3.00 & 5.15 & 5.15 & 2.95 \\
 & 500 & 0.00 & 2.95 & 2.95 & 4.25 & 4.25 & 3.50 & 3.50 & 2.95 & 2.95 & 2.95 & 1.55 & 4.25 & 4.25 & 4.20 & 2.20 & 4.25 & 4.20 & 2.20 & 3.55 & 3.55 & 3.55 & 1.90 & 3.70 & 3.70 & 2.20 \\
 & & 0.50 & 1.45 & 1.45 & 5.65 & 5.65 & 5.05 & 5.05 & 1.45 & 1.45 & 1.45 & 0.85 & 5.65 & 5.65 & 5.65 & 3.45 & 5.65 & 5.65 & 3.45 & 5.10 & 5.10 & 5.10 & 3.15 & 5.10 & 5.10 & 3.15 \\
\bottomrule
\end{tabular}
}
\caption{Simulation results in Design 3: $\mu_{j}(\theta)=0$ for all $j=1,\dots,p$, $\Sigma(\theta)$ equicorrelated.}
\label{tab3}
\end{table}

\begin{table}
\hspace{-0.5cm}\scalebox{0.64}{
\begin{tabular}{ccccccccccccccccccccccccccc}
\toprule Density & $p$ & $\rho$ & \multicolumn{6}{c}{Our methods} & \multicolumn{18}{c}{\citetalias{chernozhukov/chetverikov/kato:2014c}'s methods} \\
\cmidrule(lr){4-9}\cmidrule(lr){10-27}
 & & & \multicolumn{2}{c}{SN Lasso} & \multicolumn{2}{c}{MB Lasso} & \multicolumn{2}{c}{EB Lasso} & SN-1S & \multicolumn{3}{c}{SN-2S} & MB-1S & \multicolumn{3}{c}{MB-H} & \multicolumn{3}{c}{MB-2S} & EB-1S & \multicolumn{3}{c}{EB-H} & \multicolumn{3}{c}{EB-2S} \\
\cmidrule(lr){4-5}\cmidrule(lr){6-7}\cmidrule(lr){8-9}\cmidrule(lr){11-13}\cmidrule(lr){15-17}\cmidrule(lr){18-20}\cmidrule(lr){22-24}\cmidrule(lr){25-27}
 & & & $\varepsilon=2/3$ & $\varepsilon=8/3$ & $\varepsilon=2/3$ & $\varepsilon=8/3$ & $\varepsilon=2/3$ & $\varepsilon=8/3$ & & 0.01\% & 0.10\% & 1.00\% & & 0.01\% & 0.10\% & 1.00\% & 0.01\% & 0.10\% & 1.00\% & & 0.01\% & 0.10\% & 1.00\% & 0.01\% & 0.10\% & 1.00\% \\
\midrule
$t_{4}/\sqrt{2}$ & 200 & 0.00 & 3.05 & 3.05 & 4.30 & 4.30 & 3.55 & 3.55 & 3.05 & 3.05 & 3.05 & 1.55 & 4.30 & 4.30 & 4.25 & 2.15 & 4.30 & 4.25 & 2.15 & 3.60 & 3.60 & 3.60 & 2.00 & 3.45 & 3.45 & 1.95 \\
 & & 0.50 & 3.80 & 3.80 & 4.85 & 4.85 & 4.45 & 4.45 & 3.80 & 3.80 & 3.80 & 2.45 & 4.85 & 4.85 & 4.85 & 3.05 & 4.85 & 4.85 & 3.05 & 4.55 & 4.55 & 4.55 & 2.85 & 4.55 & 4.55 & 2.95 \\
 & & 0.90 & 2.40 & 2.40 & 5.90 & 5.90 & 5.55 & 5.55 & 2.40 & 2.40 & 2.40 & 1.75 & 5.90 & 5.90 & 5.90 & 3.60 & 5.90 & 5.90 & 3.60 & 5.55 & 5.55 & 5.55 & 3.50 & 5.75 & 5.65 & 3.20 \\
 & 500 & 0.00 & 2.95 & 2.95 & 4.25 & 4.25 & 3.50 & 3.50 & 2.95 & 2.95 & 2.95 & 1.55 & 4.25 & 4.25 & 4.20 & 2.20 & 4.25 & 4.20 & 2.20 & 3.55 & 3.55 & 3.55 & 1.90 & 3.70 & 3.70 & 2.20 \\
 & & 0.50 & 3.75 & 3.75 & 5.45 & 5.45 & 4.85 & 4.85 & 3.75 & 3.75 & 3.75 & 1.95 & 5.45 & 5.45 & 5.45 & 3.05 & 5.45 & 5.45 & 3.05 & 4.95 & 4.95 & 4.95 & 2.45 & 4.55 & 4.50 & 2.55 \\
\bottomrule
\end{tabular}
}
\caption{Simulation results in Design 4: $\mu_{j}(\theta)=0$ for all $j=1,\dots,p$, $\Sigma(\theta)$ Toeplitz.}
\label{tab4}
\end{table}

\begin{table}
\hspace{-0.5cm}\scalebox{0.64}{
\begin{tabular}{ccccccccccccccccccccccccccc}
\toprule Density & $p$ & $\rho$ & \multicolumn{6}{c}{Our methods} & \multicolumn{18}{c}{\citetalias{chernozhukov/chetverikov/kato:2014c}'s methods} \\
\cmidrule(lr){4-9}\cmidrule(lr){10-27}
 & & & \multicolumn{2}{c}{SN Lasso} & \multicolumn{2}{c}{MB Lasso} & \multicolumn{2}{c}{EB Lasso} & SN-1S & \multicolumn{3}{c}{SN-2S} & MB-1S & \multicolumn{3}{c}{MB-H} & \multicolumn{3}{c}{MB-2S} & EB-1S & \multicolumn{3}{c}{EB-H} & \multicolumn{3}{c}{EB-2S} \\
\cmidrule(lr){4-5}\cmidrule(lr){6-7}\cmidrule(lr){8-9}\cmidrule(lr){11-13}\cmidrule(lr){15-17}\cmidrule(lr){18-20}\cmidrule(lr){22-24}\cmidrule(lr){25-27}
 & & & $\varepsilon=2/3$ & $\varepsilon=8/3$ & $\varepsilon=2/3$ & $\varepsilon=8/3$ & $\varepsilon=2/3$ & $\varepsilon=8/3$ & & 0.01\% & 0.10\% & 1.00\% & & 0.01\% & 0.10\% & 1.00\% & 0.01\% & 0.10\% & 1.00\% & & 0.01\% & 0.10\% & 1.00\% & 0.01\% & 0.10\% & 1.00\% \\
\midrule
$t_{4}/\sqrt{2}$ & 200 & 0.00 & 70.25 & 70.25 & 75.70 & 75.70 & 73.15 & 73.15 & 70.25 & 70.25 & 70.10 & 53.85 & 75.70 & 75.70 & 75.60 & 61.55 & 75.70 & 75.60 & 61.55 & 73.55 & 73.55 & 73.50 & 58.05 & 73.25 & 73.10 & 58.60 \\
 & & 0.50 & 22.15 & 22.15 & 37.75 & 37.75 & 36.75 & 36.75 & 22.15 & 22.10 & 22.10 & 16.55 & 37.75 & 37.75 & 37.70 & 28.50 & 37.75 & 37.70 & 28.50 & 37.00 & 37.00 & 36.80 & 28.15 & 37.05 & 37.00 & 27.95 \\
 & & 0.90 & 4.40 & 4.40 & 27.80 & 27.80 & 27.20 & 27.20 & 4.40 & 4.40 & 4.40 & 2.80 & 27.80 & 27.80 & 27.75 & 20.20 & 27.80 & 27.75 & 20.20 & 27.30 & 27.30 & 27.25 & 20.05 & 27.00 & 27.00 & 20.20 \\
 & 500 & 0.00 & 77.45 & 77.45 & 84.05 & 84.05 & 80.55 & 80.55 & 77.45 & 77.45 & 77.40 & 62.00 & 84.05 & 84.05 & 83.90 & 70.90 & 84.05 & 83.90 & 70.90 & 80.75 & 80.75 & 80.65 & 66.05 & 80.65 & 80.60 & 65.90 \\
 & & 0.50 & 20.95 & 20.95 & 38.35 & 38.35 & 37.05 & 37.05 & 20.95 & 20.95 & 20.90 & 15.70 & 38.35 & 38.35 & 38.25 & 30.05 & 38.35 & 38.25 & 30.05 & 37.30 & 37.30 & 37.25 & 28.45 & 37.05 & 36.95 & 28.90 \\
\bottomrule
\end{tabular}
}
\caption{Simulation results in Design 5: $\mu_{j}(\theta)=0.05$ for all $j=1,\dots p$, $\Sigma(\theta)$ equicorrelated.}
\label{tab5}
\end{table}

\begin{table}
\hspace{-0.5cm}\scalebox{0.64}{
\begin{tabular}{ccccccccccccccccccccccccccc}
\toprule Density & $p$ & $\rho$ & \multicolumn{6}{c}{Our methods} & \multicolumn{18}{c}{\citetalias{chernozhukov/chetverikov/kato:2014c}'s methods} \\
\cmidrule(lr){4-9}\cmidrule(lr){10-27}
 & & & \multicolumn{2}{c}{SN Lasso} & \multicolumn{2}{c}{MB Lasso} & \multicolumn{2}{c}{EB Lasso} & SN-1S & \multicolumn{3}{c}{SN-2S} & MB-1S & \multicolumn{3}{c}{MB-H} & \multicolumn{3}{c}{MB-2S} & EB-1S & \multicolumn{3}{c}{EB-H} & \multicolumn{3}{c}{EB-2S} \\
\cmidrule(lr){4-5}\cmidrule(lr){6-7}\cmidrule(lr){8-9}\cmidrule(lr){11-13}\cmidrule(lr){15-17}\cmidrule(lr){18-20}\cmidrule(lr){22-24}\cmidrule(lr){25-27}
 & & & $\varepsilon=2/3$ & $\varepsilon=8/3$ & $\varepsilon=2/3$ & $\varepsilon=8/3$ & $\varepsilon=2/3$ & $\varepsilon=8/3$ & & 0.01\% & 0.10\% & 1.00\% & & 0.01\% & 0.10\% & 1.00\% & 0.01\% & 0.10\% & 1.00\% & & 0.01\% & 0.10\% & 1.00\% & 0.01\% & 0.10\% & 1.00\% \\
\midrule
$t_{4}/\sqrt{2}$ & 200 & 0.00 & 70.29 & 70.29 & 75.74 & 75.74 & 73.19 & 73.19 & 70.29 & 70.29 & 70.14 & 53.88 & 75.74 & 75.74 & 75.64 & 61.58 & 75.74 & 75.64 & 61.58 & 73.59 & 73.59 & 73.54 & 58.08 & 73.29 & 73.14 & 58.63 \\
 & & 0.50 & 64.06 & 64.06 & 71.17 & 71.17 & 69.02 & 69.02 & 64.06 & 64.06 & 64.01 & 49.41 & 71.17 & 71.17 & 71.02 & 57.86 & 71.17 & 71.02 & 57.86 & 69.32 & 69.32 & 69.22 & 54.01 & 69.52 & 69.22 & 54.86 \\
 & & 0.90 & 33.70 & 33.70 & 53.06 & 53.06 & 52.36 & 52.36 & 33.70 & 33.70 & 33.65 & 25.15 & 53.06 & 53.06 & 53.06 & 41.36 & 53.06 & 53.06 & 41.36 & 52.61 & 52.61 & 52.41 & 40.71 & 52.46 & 52.36 & 40.01 \\
 & 500 & 0.00 & 77.49 & 77.49 & 84.09 & 84.09 & 80.59 & 80.59 & 77.49 & 77.49 & 77.44 & 62.03 & 84.09 & 84.09 & 83.94 & 70.94 & 84.09 & 83.94 & 70.94 & 80.79 & 80.79 & 80.69 & 66.08 & 80.69 & 80.64 & 65.93 \\
 & & 0.50 & 72.76 & 72.76 & 81.57 & 81.57 & 78.32 & 78.32 & 72.76 & 72.76 & 72.76 & 58.21 & 81.57 & 81.57 & 81.47 & 67.42 & 81.57 & 81.47 & 67.42 & 78.62 & 78.62 & 78.52 & 64.81 & 78.52 & 78.37 & 64.26 \\
\bottomrule
\end{tabular}
}
\caption{Simulation results in Design 6: $\mu_{j}(\theta)=0.05$ for all $j=1,\dots p$, $\Sigma(\theta)$ Toeplitz.}
\label{tab6}
\end{table}

\begin{table}
\hspace{-0.5cm}\scalebox{0.64}{
\begin{tabular}{ccccccccccccccccccccccccccc}
\toprule Density & $p$ & $\rho$ & \multicolumn{6}{c}{Our methods} & \multicolumn{18}{c}{\citetalias{chernozhukov/chetverikov/kato:2014c}'s methods} \\
\cmidrule(lr){4-9}\cmidrule(lr){10-27}
 & & & \multicolumn{2}{c}{SN Lasso} & \multicolumn{2}{c}{MB Lasso} & \multicolumn{2}{c}{EB Lasso} & SN-1S & \multicolumn{3}{c}{SN-2S} & MB-1S & \multicolumn{3}{c}{MB-H} & \multicolumn{3}{c}{MB-2S} & EB-1S & \multicolumn{3}{c}{EB-H} & \multicolumn{3}{c}{EB-2S} \\
\cmidrule(lr){4-5}\cmidrule(lr){6-7}\cmidrule(lr){8-9}\cmidrule(lr){11-13}\cmidrule(lr){15-17}\cmidrule(lr){18-20}\cmidrule(lr){22-24}\cmidrule(lr){25-27}
 & & & $\varepsilon=2/3$ & $\varepsilon=8/3$ & $\varepsilon=2/3$ & $\varepsilon=8/3$ & $\varepsilon=2/3$ & $\varepsilon=8/3$ & & 0.01\% & 0.10\% & 1.00\% & & 0.01\% & 0.10\% & 1.00\% & 0.01\% & 0.10\% & 1.00\% & & 0.01\% & 0.10\% & 1.00\% & 0.01\% & 0.10\% & 1.00\% \\
\midrule
$t_{4}/\sqrt{2}$ & 200 & 0.00 & 52.19 & 52.14 & 55.44 & 55.39 & 54.24 & 54.14 & 11.84 & 49.94 & 51.19 & 39.53 & 13.79 & 53.64 & 54.69 & 42.88 & 55.34 & 55.24 & 42.88 & 12.84 & 50.94 & 52.94 & 41.58 & 53.44 & 53.04 & 42.08 \\
 & & 0.50 & 27.18 & 27.18 & 35.39 & 35.39 & 34.59 & 34.59 & 7.13 & 26.68 & 26.93 & 20.93 & 15.19 & 34.74 & 35.19 & 27.23 & 35.34 & 35.29 & 27.23 & 14.84 & 34.09 & 34.54 & 26.68 & 34.94 & 34.89 & 26.48 \\
 & & 0.90 & 10.82 & 10.82 & 28.28 & 28.28 & 28.28 & 28.28 & 2.22 & 10.77 & 10.77 & 7.86 & 19.18 & 28.28 & 28.28 & 20.87 & 28.28 & 28.28 & 20.87 & 19.23 & 28.28 & 28.38 & 20.72 & 28.23 & 28.13 & 20.57 \\
 & 500 & 0.00 & 59.24 & 59.19 & 64.69 & 64.64 & 62.64 & 62.34 & 14.19 & 56.89 & 58.09 & 45.23 & 17.99 & 61.54 & 63.34 & 50.54 & 63.84 & 63.74 & 50.54 & 15.74 & 57.74 & 59.84 & 48.18 & 60.89 & 60.89 & 47.43 \\
 & & 0.50 & 25.39 & 25.39 & 36.69 & 36.64 & 36.09 & 35.99 & 7.74 & 24.89 & 25.09 & 19.78 & 16.04 & 35.84 & 36.19 & 28.38 & 36.34 & 36.34 & 28.38 & 15.09 & 34.84 & 35.49 & 27.53 & 35.94 & 35.84 & 27.58 \\
\bottomrule
\end{tabular}
}
\caption{Simulation results in Design 7: $\mu_{j}(\theta)=-0.75\cdot 1[j>0.1 p]+0.05\cdot 1[j\leq 0.1 p]$, $\Sigma(\theta)$ equicorrelated.}
\label{tab7}
\end{table}

\begin{table}
\hspace{-0.5cm}\scalebox{0.64}{
\begin{tabular}{ccccccccccccccccccccccccccc}
\toprule Density & $p$ & $\rho$ & \multicolumn{6}{c}{Our methods} & \multicolumn{18}{c}{\citetalias{chernozhukov/chetverikov/kato:2014c}'s methods} \\
\cmidrule(lr){4-9}\cmidrule(lr){10-27}
 & & & \multicolumn{2}{c}{SN Lasso} & \multicolumn{2}{c}{MB Lasso} & \multicolumn{2}{c}{EB Lasso} & SN-1S & \multicolumn{3}{c}{SN-2S} & MB-1S & \multicolumn{3}{c}{MB-H} & \multicolumn{3}{c}{MB-2S} & EB-1S & \multicolumn{3}{c}{EB-H} & \multicolumn{3}{c}{EB-2S} \\
\cmidrule(lr){4-5}\cmidrule(lr){6-7}\cmidrule(lr){8-9}\cmidrule(lr){11-13}\cmidrule(lr){15-17}\cmidrule(lr){18-20}\cmidrule(lr){22-24}\cmidrule(lr){25-27}
 & & & $\varepsilon=2/3$ & $\varepsilon=8/3$ & $\varepsilon=2/3$ & $\varepsilon=8/3$ & $\varepsilon=2/3$ & $\varepsilon=8/3$ & & 0.01\% & 0.10\% & 1.00\% & & 0.01\% & 0.10\% & 1.00\% & 0.01\% & 0.10\% & 1.00\% & & 0.01\% & 0.10\% & 1.00\% & 0.01\% & 0.10\% & 1.00\% \\
\midrule
$t_{4}/\sqrt{2}$ & 200 & 0.00 & 52.18 & 52.13 & 55.43 & 55.38 & 54.23 & 54.13 & 11.81 & 49.93 & 51.18 & 39.52 & 13.76 & 53.63 & 54.68 & 42.87 & 55.33 & 55.23 & 42.87 & 12.81 & 50.93 & 52.93 & 41.57 & 53.43 & 53.03 & 42.07 \\
 & & 0.50 & 43.56 & 43.51 & 48.42 & 48.37 & 47.22 & 47.07 & 10.15 & 41.96 & 42.96 & 32.61 & 12.51 & 47.12 & 48.12 & 35.81 & 48.12 & 48.12 & 35.86 & 11.51 & 45.02 & 46.27 & 35.91 & 47.42 & 47.37 & 35.51 \\
 & & 0.90 & 18.86 & 18.81 & 33.96 & 33.96 & 34.16 & 34.16 & 4.85 & 18.61 & 18.71 & 14.25 & 8.46 & 33.36 & 33.51 & 24.46 & 33.91 & 33.66 & 24.46 & 8.61 & 33.36 & 33.76 & 24.41 & 33.06 & 32.96 & 24.31 \\
 & 500 & 0.00 & 59.23 & 59.18 & 64.68 & 64.63 & 62.63 & 62.33 & 14.16 & 56.88 & 58.08 & 45.22 & 17.96 & 61.53 & 63.33 & 50.53 & 63.83 & 63.73 & 50.53 & 15.71 & 57.73 & 59.83 & 48.17 & 60.88 & 60.88 & 47.42 \\
 & & 0.50 & 53.46 & 53.46 & 59.32 & 59.22 & 57.42 & 56.97 & 12.45 & 51.36 & 52.66 & 41.56 & 15.56 & 56.92 & 58.07 & 46.61 & 58.77 & 58.67 & 46.66 & 14.61 & 53.47 & 55.67 & 44.46 & 56.17 & 56.12 & 44.56 \\
\bottomrule
\end{tabular}
}
\caption{Simulation results in Design 8: $\mu_{j}(\theta)=-0.75\cdot 1[j>0.1 p]+0.05\cdot 1[j\leq 0.1 p]$, $\Sigma(\theta)$ Toeplitz.}
\label{tab8}
\end{table}

\begin{table}
\hspace{-0.5cm}\scalebox{0.64}{
\begin{tabular}{ccccccccccccccccccccccccccc}
\toprule Density & $p$ & $\rho$ & \multicolumn{6}{c}{Our methods} & \multicolumn{18}{c}{\citetalias{chernozhukov/chetverikov/kato:2014c}'s methods} \\
\cmidrule(lr){4-9}\cmidrule(lr){10-27}
 & & & \multicolumn{2}{c}{SN Lasso} & \multicolumn{2}{c}{MB Lasso} & \multicolumn{2}{c}{EB Lasso} & SN-1S & \multicolumn{3}{c}{SN-2S} & MB-1S & \multicolumn{3}{c}{MB-H} & \multicolumn{3}{c}{MB-2S} & EB-1S & \multicolumn{3}{c}{EB-H} & \multicolumn{3}{c}{EB-2S} \\
\cmidrule(lr){4-5}\cmidrule(lr){6-7}\cmidrule(lr){8-9}\cmidrule(lr){11-13}\cmidrule(lr){15-17}\cmidrule(lr){18-20}\cmidrule(lr){22-24}\cmidrule(lr){25-27}
 & & & $\varepsilon=2/3$ & $\varepsilon=8/3$ & $\varepsilon=2/3$ & $\varepsilon=8/3$ & $\varepsilon=2/3$ & $\varepsilon=8/3$ & & 0.01\% & 0.10\% & 1.00\% & & 0.01\% & 0.10\% & 1.00\% & 0.01\% & 0.10\% & 1.00\% & & 0.01\% & 0.10\% & 1.00\% & 0.01\% & 0.10\% & 1.00\% \\
\midrule
$t_{4}/\sqrt{2}$ & 200 & 0.00 & 52.18 & 51.38 & 55.43 & 54.98 & 54.23 & 52.88 & 11.81 & 27.78 & 40.53 & 38.87 & 13.76 & 31.43 & 43.23 & 42.47 & 51.33 & 51.13 & 42.62 & 12.81 & 28.58 & 40.53 & 39.92 & 47.23 & 47.08 & 40.47 \\
 & & 0.50 & 43.57 & 43.22 & 48.42 & 48.02 & 47.22 & 46.32 & 10.16 & 23.92 & 33.57 & 32.32 & 12.51 & 27.27 & 37.52 & 35.62 & 44.57 & 44.57 & 35.57 & 11.51 & 25.37 & 35.97 & 35.22 & 42.72 & 42.57 & 34.67 \\
 & & 0.90 & 18.86 & 18.61 & 33.97 & 33.27 & 34.17 & 33.07 & 4.85 & 10.76 & 16.16 & 14.21 & 8.45 & 17.27 & 25.72 & 24.31 & 32.22 & 32.12 & 24.46 & 8.60 & 16.67 & 24.97 & 24.16 & 30.62 & 30.47 & 24.16 \\
 & 500 & 0.00 & 59.23 & 58.73 & 64.68 & 64.13 & 62.58 & 61.28 & 14.16 & 28.33 & 40.68 & 44.37 & 17.96 & 32.68 & 46.68 & 49.23 & 57.03 & 56.93 & 49.63 & 15.71 & 28.48 & 40.33 & 44.57 & 48.88 & 48.83 & 44.67 \\
 & & 0.50 & 53.48 & 53.23 & 59.28 & 58.88 & 57.38 & 56.28 & 12.46 & 25.83 & 38.18 & 40.37 & 15.56 & 30.58 & 43.13 & 45.62 & 52.78 & 52.78 & 46.07 & 14.61 & 27.33 & 38.88 & 42.07 & 47.23 & 47.23 & 42.52 \\
\bottomrule
\end{tabular}
}
\caption{Simulation results in Design 9: $\mu_{j}(\theta)=-0.6\cdot 1[j>0.1 p]+0.05\cdot 1[j\leq 0.1 p]$, $\Sigma(\theta)$ Toeplitz.}
\label{tab9}
\end{table}

\begin{table}
\hspace{-0.5cm}\scalebox{0.64}{
\begin{tabular}{ccccccccccccccccccccccccccc}
\toprule Density & $p$ & $\rho$ & \multicolumn{6}{c}{Our methods} & \multicolumn{18}{c}{\citetalias{chernozhukov/chetverikov/kato:2014c}'s methods} \\
\cmidrule(lr){4-9}\cmidrule(lr){10-27}
 & & & \multicolumn{2}{c}{SN Lasso} & \multicolumn{2}{c}{MB Lasso} & \multicolumn{2}{c}{EB Lasso} & SN-1S & \multicolumn{3}{c}{SN-2S} & MB-1S & \multicolumn{3}{c}{MB-H} & \multicolumn{3}{c}{MB-2S} & EB-1S & \multicolumn{3}{c}{EB-H} & \multicolumn{3}{c}{EB-2S} \\
\cmidrule(lr){4-5}\cmidrule(lr){6-7}\cmidrule(lr){8-9}\cmidrule(lr){11-13}\cmidrule(lr){15-17}\cmidrule(lr){18-20}\cmidrule(lr){22-24}\cmidrule(lr){25-27}
 & & & $\varepsilon=2/3$ & $\varepsilon=8/3$ & $\varepsilon=2/3$ & $\varepsilon=8/3$ & $\varepsilon=2/3$ & $\varepsilon=8/3$ & & 0.01\% & 0.10\% & 1.00\% & & 0.01\% & 0.10\% & 1.00\% & 0.01\% & 0.10\% & 1.00\% & & 0.01\% & 0.10\% & 1.00\% & 0.01\% & 0.10\% & 1.00\% \\
\midrule
$t_{4}/\sqrt{2}$ & 200 & 0.00 & 52.13 & 46.83 & 55.43 & 49.78 & 54.13 & 47.23 & 11.81 & 13.66 & 17.77 & 31.42 & 13.76 & 15.77 & 20.82 & 35.27 & 32.58 & 32.53 & 36.52 & 12.81 & 14.71 & 18.42 & 31.87 & 27.87 & 27.87 & 32.72 \\
 & & 0.50 & 43.57 & 37.52 & 48.42 & 42.47 & 47.17 & 39.62 & 10.16 & 11.81 & 15.17 & 27.47 & 12.51 & 14.21 & 18.42 & 30.47 & 28.37 & 28.32 & 31.37 & 11.51 & 13.16 & 17.17 & 29.02 & 24.62 & 24.62 & 28.97 \\
 & & 0.90 & 18.86 & 13.56 & 33.97 & 23.17 & 34.17 & 21.87 & 4.85 & 5.21 & 6.71 & 12.56 & 8.45 & 9.11 & 11.56 & 19.71 & 19.47 & 19.42 & 22.66 & 8.60 & 9.21 & 11.46 & 19.66 & 18.87 & 18.87 & 21.86 \\
 & 500 & 0.00 & 59.23 & 56.63 & 64.68 & 61.53 & 62.48 & 56.78 & 14.16 & 15.66 & 19.52 & 31.57 & 17.96 & 19.52 & 22.97 & 36.08 & 33.53 & 33.43 & 39.08 & 15.71 & 17.46 & 20.37 & 30.42 & 26.37 & 26.32 & 30.67 \\
 & & 0.50 & 53.48 & 50.18 & 59.28 & 55.58 & 57.33 & 52.23 & 12.46 & 13.66 & 16.67 & 29.62 & 15.56 & 16.87 & 20.82 & 34.87 & 31.68 & 31.63 & 36.77 & 14.61 & 15.51 & 18.42 & 29.72 & 25.42 & 25.32 & 30.77 \\
\bottomrule
\end{tabular}
}
\caption{Simulation results in Design 10: $\mu_{j}(\theta)=-0.5\cdot 1[j>0.1 p]+0.05\cdot 1[j\leq 0.1 p]$, $\Sigma(\theta)$ Toeplitz.}
\label{tab10}
\end{table}

\begin{table}
\hspace{-0.5cm}\scalebox{0.64}{
\begin{tabular}{ccccccccccccccccccccccccccc}
\toprule Density & $p$ & $\rho$ & \multicolumn{6}{c}{Our methods} & \multicolumn{18}{c}{\citetalias{chernozhukov/chetverikov/kato:2014c}'s methods} \\
\cmidrule(lr){4-9}\cmidrule(lr){10-27}
 & & & \multicolumn{2}{c}{SN Lasso} & \multicolumn{2}{c}{MB Lasso} & \multicolumn{2}{c}{EB Lasso} & SN-1S & \multicolumn{3}{c}{SN-2S} & MB-1S & \multicolumn{3}{c}{MB-H} & \multicolumn{3}{c}{MB-2S} & EB-1S & \multicolumn{3}{c}{EB-H} & \multicolumn{3}{c}{EB-2S} \\
\cmidrule(lr){4-5}\cmidrule(lr){6-7}\cmidrule(lr){8-9}\cmidrule(lr){11-13}\cmidrule(lr){15-17}\cmidrule(lr){18-20}\cmidrule(lr){22-24}\cmidrule(lr){25-27}
 & & & $\varepsilon=2/3$ & $\varepsilon=8/3$ & $\varepsilon=2/3$ & $\varepsilon=8/3$ & $\varepsilon=2/3$ & $\varepsilon=8/3$ & & 0.01\% & 0.10\% & 1.00\% & & 0.01\% & 0.10\% & 1.00\% & 0.01\% & 0.10\% & 1.00\% & & 0.01\% & 0.10\% & 1.00\% & 0.01\% & 0.10\% & 1.00\% \\
\midrule
$t_{4}/\sqrt{2}$ & 200 & 0.00 & 51.93 & 24.82 & 55.33 & 27.88 & 53.93 & 25.17 & 11.81 & 11.86 & 12.11 & 13.17 & 13.76 & 13.76 & 14.06 & 15.22 & 15.97 & 15.92 & 16.52 & 12.81 & 12.86 & 13.11 & 13.72 & 14.56 & 14.51 & 14.12 \\
 & & 0.50 & 43.52 & 19.37 & 48.37 & 22.62 & 46.97 & 20.47 & 10.16 & 10.16 & 10.41 & 11.46 & 12.51 & 12.51 & 12.66 & 13.67 & 14.41 & 14.31 & 14.72 & 11.51 & 11.56 & 11.91 & 12.32 & 13.21 & 13.21 & 12.91 \\
 & & 0.90 & 18.81 & 6.41 & 33.72 & 10.71 & 33.97 & 10.76 & 4.85 & 4.85 & 4.85 & 5.01 & 8.45 & 8.45 & 8.46 & 8.66 & 9.31 & 9.31 & 10.66 & 8.60 & 8.60 & 8.61 & 8.16 & 9.31 & 9.31 & 9.81 \\
 & 500 & 0.00 & 59.03 & 33.23 & 64.48 & 38.73 & 62.13 & 32.98 & 14.16 & 14.21 & 14.36 & 13.87 & 17.96 & 17.96 & 17.91 & 16.52 & 19.77 & 19.72 & 17.82 & 15.71 & 15.81 & 16.06 & 14.42 & 16.81 & 16.76 & 14.72 \\
 & & 0.50 & 53.43 & 29.48 & 59.23 & 34.38 & 56.78 & 31.08 & 12.46 & 12.46 & 12.51 & 11.77 & 15.56 & 15.56 & 15.76 & 14.97 & 16.77 & 16.72 & 15.77 & 14.61 & 14.61 & 14.61 & 12.97 & 15.06 & 15.06 & 13.27 \\
\bottomrule
\end{tabular}
}
\caption{Simulation results in Design 11: $\mu_{j}(\theta)=-0.4\cdot 1[j>0.1 p]+0.05\cdot 1[j\leq 0.1 p]$, $\Sigma(\theta)$ Toeplitz.}
\label{tab11}
\end{table}

\begin{table}
\hspace{-0.5cm}\scalebox{0.64}{
\begin{tabular}{ccccccccccccccccccccccccccc}
\toprule Density & $p$ & $\rho$ & \multicolumn{6}{c}{Our methods} & \multicolumn{18}{c}{\citetalias{chernozhukov/chetverikov/kato:2014c}'s methods} \\
\cmidrule(lr){4-9}\cmidrule(lr){10-27}
 & & & \multicolumn{2}{c}{SN Lasso} & \multicolumn{2}{c}{MB Lasso} & \multicolumn{2}{c}{EB Lasso} & SN-1S & \multicolumn{3}{c}{SN-2S} & MB-1S & \multicolumn{3}{c}{MB-H} & \multicolumn{3}{c}{MB-2S} & EB-1S & \multicolumn{3}{c}{EB-H} & \multicolumn{3}{c}{EB-2S} \\
\cmidrule(lr){4-5}\cmidrule(lr){6-7}\cmidrule(lr){8-9}\cmidrule(lr){11-13}\cmidrule(lr){15-17}\cmidrule(lr){18-20}\cmidrule(lr){22-24}\cmidrule(lr){25-27}
 & & & $\varepsilon=2/3$ & $\varepsilon=8/3$ & $\varepsilon=2/3$ & $\varepsilon=8/3$ & $\varepsilon=2/3$ & $\varepsilon=8/3$ & & 0.01\% & 0.10\% & 1.00\% & & 0.01\% & 0.10\% & 1.00\% & 0.01\% & 0.10\% & 1.00\% & & 0.01\% & 0.10\% & 1.00\% & 0.01\% & 0.10\% & 1.00\% \\
\midrule
$t_{4}/\sqrt{2}$ & 200 & 0.00 & 48.33 & 12.56 & 51.53 & 14.76 & 48.98 & 13.61 & 11.81 & 11.81 & 11.76 & 8.16 & 13.76 & 13.76 & 13.66 & 10.01 & 13.81 & 13.71 & 10.31 & 12.81 & 12.81 & 12.76 & 9.06 & 12.91 & 12.91 & 9.16 \\
 & & 0.50 & 38.82 & 10.66 & 43.77 & 13.01 & 42.07 & 12.16 & 10.16 & 10.16 & 10.16 & 6.76 & 12.51 & 12.51 & 12.51 & 8.86 & 12.56 & 12.56 & 8.96 & 11.51 & 11.51 & 11.56 & 8.16 & 11.71 & 11.56 & 7.96 \\
 & & 0.90 & 15.66 & 4.90 & 27.07 & 8.51 & 26.07 & 8.51 & 4.85 & 4.85 & 4.85 & 3.60 & 8.45 & 8.45 & 8.45 & 6.35 & 8.45 & 8.45 & 6.51 & 8.60 & 8.60 & 8.55 & 6.20 & 8.30 & 8.25 & 6.15 \\
 & 500 & 0.00 & 56.93 & 16.32 & 61.68 & 20.27 & 57.58 & 17.32 & 14.16 & 14.11 & 14.11 & 9.86 & 17.96 & 17.96 & 17.76 & 12.41 & 17.96 & 17.81 & 12.51 & 15.71 & 15.71 & 15.66 & 11.21 & 15.46 & 15.46 & 10.96 \\
 & & 0.50 & 50.53 & 13.77 & 55.93 & 17.52 & 52.93 & 15.37 & 12.46 & 12.46 & 12.46 & 8.41 & 15.56 & 15.56 & 15.56 & 11.01 & 15.56 & 15.56 & 11.11 & 14.61 & 14.61 & 14.61 & 10.26 & 14.51 & 14.51 & 10.01 \\
\bottomrule
\end{tabular}
}
\caption{Simulation results in Design 12: $\mu_{j}(\theta)=-0.3\cdot 1[j>0.1 p]+0.05\cdot 1[j\leq 0.1 p]$, $\Sigma(\theta)$ Toeplitz.}
\label{tab12}
\end{table}

\begin{table}
\hspace{-0.5cm}\scalebox{0.64}{
\begin{tabular}{ccccccccccccccccccccccccccc}
\toprule Density & $p$ & $\rho$ & \multicolumn{6}{c}{Our methods} & \multicolumn{18}{c}{\citetalias{chernozhukov/chetverikov/kato:2014c}'s methods} \\
\cmidrule(lr){4-9}\cmidrule(lr){10-27}
 & & & \multicolumn{2}{c}{SN Lasso} & \multicolumn{2}{c}{MB Lasso} & \multicolumn{2}{c}{EB Lasso} & SN-1S & \multicolumn{3}{c}{SN-2S} & MB-1S & \multicolumn{3}{c}{MB-H} & \multicolumn{3}{c}{MB-2S} & EB-1S & \multicolumn{3}{c}{EB-H} & \multicolumn{3}{c}{EB-2S} \\
\cmidrule(lr){4-5}\cmidrule(lr){6-7}\cmidrule(lr){8-9}\cmidrule(lr){11-13}\cmidrule(lr){15-17}\cmidrule(lr){18-20}\cmidrule(lr){22-24}\cmidrule(lr){25-27}
 & & & $\varepsilon=2/3$ & $\varepsilon=8/3$ & $\varepsilon=2/3$ & $\varepsilon=8/3$ & $\varepsilon=2/3$ & $\varepsilon=8/3$ & & 0.01\% & 0.10\% & 1.00\% & & 0.01\% & 0.10\% & 1.00\% & 0.01\% & 0.10\% & 1.00\% & & 0.01\% & 0.10\% & 1.00\% & 0.01\% & 0.10\% & 1.00\% \\
\midrule
$t_{4}/\sqrt{2}$ & 200 & 0.00 & 20.42 & 11.81 & 23.68 & 13.76 & 20.62 & 12.66 & 11.81 & 11.81 & 11.76 & 7.80 & 13.76 & 13.76 & 13.66 & 9.76 & 13.76 & 13.66 & 9.76 & 12.81 & 12.81 & 12.76 & 8.80 & 12.91 & 12.91 & 8.80 \\
 & & 0.50 & 16.97 & 10.16 & 20.12 & 12.51 & 18.07 & 11.56 & 10.16 & 10.16 & 10.16 & 6.45 & 12.51 & 12.51 & 12.51 & 8.80 & 12.51 & 12.51 & 8.80 & 11.51 & 11.51 & 11.51 & 7.85 & 11.66 & 11.51 & 7.60 \\
 & & 0.90 & 6.41 & 4.85 & 11.21 & 8.45 & 11.01 & 8.50 & 4.85 & 4.85 & 4.85 & 3.45 & 8.45 & 8.45 & 8.45 & 6.15 & 8.45 & 8.45 & 6.15 & 8.60 & 8.60 & 8.55 & 6.10 & 8.30 & 8.25 & 5.80 \\
 & 500 & 0.00 & 27.18 & 14.21 & 31.03 & 17.96 & 26.93 & 15.71 & 14.16 & 14.11 & 14.11 & 9.75 & 17.96 & 17.96 & 17.76 & 12.26 & 17.96 & 17.76 & 12.26 & 15.71 & 15.71 & 15.66 & 11.21 & 15.46 & 15.46 & 10.91 \\
 & & 0.50 & 24.08 & 12.46 & 29.08 & 15.56 & 25.18 & 14.61 & 12.46 & 12.46 & 12.46 & 8.20 & 15.56 & 15.56 & 15.56 & 10.91 & 15.56 & 15.56 & 10.91 & 14.61 & 14.61 & 14.61 & 10.16 & 14.51 & 14.51 & 9.91 \\
\bottomrule
\end{tabular}
}
\caption{Simulation results in Design 13: $\mu_{j}(\theta)=-0.2\cdot 1[j>0.1 p]+0.05\cdot 1[j\leq 0.1 p]$, $\Sigma(\theta)$ Toeplitz.}
\label{tab13}
\end{table}

\begin{table}
\hspace{-0.5cm}\scalebox{0.64}{
\begin{tabular}{ccccccccccccccccccccccccccc}
\toprule Density & $p$ & $\rho$ & \multicolumn{6}{c}{Our methods} & \multicolumn{18}{c}{\citetalias{chernozhukov/chetverikov/kato:2014c}'s methods} \\
\cmidrule(lr){4-9}\cmidrule(lr){10-27}
 & & & \multicolumn{2}{c}{SN Lasso} & \multicolumn{2}{c}{MB Lasso} & \multicolumn{2}{c}{EB Lasso} & SN-1S & \multicolumn{3}{c}{SN-2S} & MB-1S & \multicolumn{3}{c}{MB-H} & \multicolumn{3}{c}{MB-2S} & EB-1S & \multicolumn{3}{c}{EB-H} & \multicolumn{3}{c}{EB-2S} \\
\cmidrule(lr){4-5}\cmidrule(lr){6-7}\cmidrule(lr){8-9}\cmidrule(lr){11-13}\cmidrule(lr){15-17}\cmidrule(lr){18-20}\cmidrule(lr){22-24}\cmidrule(lr){25-27}
 & & & $\varepsilon=2/3$ & $\varepsilon=8/3$ & $\varepsilon=2/3$ & $\varepsilon=8/3$ & $\varepsilon=2/3$ & $\varepsilon=8/3$ & & 0.01\% & 0.10\% & 1.00\% & & 0.01\% & 0.10\% & 1.00\% & 0.01\% & 0.10\% & 1.00\% & & 0.01\% & 0.10\% & 1.00\% & 0.01\% & 0.10\% & 1.00\% \\
\midrule
$t_{4}/\sqrt{2}$ & 200 & 0.00 & 12.16 & 11.81 & 14.11 & 13.76 & 12.96 & 12.66 & 11.81 & 11.81 & 11.76 & 7.80 & 13.76 & 13.76 & 13.66 & 9.75 & 13.76 & 13.66 & 9.75 & 12.81 & 12.81 & 12.76 & 8.80 & 12.91 & 12.91 & 8.80 \\
 & & 0.50 & 10.41 & 10.16 & 12.61 & 12.51 & 11.86 & 11.51 & 10.16 & 10.16 & 10.16 & 6.45 & 12.51 & 12.51 & 12.51 & 8.80 & 12.51 & 12.51 & 8.80 & 11.51 & 11.51 & 11.51 & 7.85 & 11.66 & 11.51 & 7.60 \\
 & & 0.90 & 4.85 & 4.85 & 8.51 & 8.45 & 8.51 & 8.50 & 4.85 & 4.85 & 4.85 & 3.45 & 8.45 & 8.45 & 8.45 & 6.15 & 8.45 & 8.45 & 6.15 & 8.60 & 8.60 & 8.55 & 6.10 & 8.30 & 8.25 & 5.80 \\
 & 500 & 0.00 & 14.66 & 14.16 & 18.37 & 17.96 & 16.46 & 15.61 & 14.16 & 14.11 & 14.11 & 9.75 & 17.96 & 17.96 & 17.76 & 12.26 & 17.96 & 17.76 & 12.26 & 15.71 & 15.71 & 15.66 & 11.21 & 15.46 & 15.46 & 10.91 \\
 & & 0.50 & 12.86 & 12.46 & 16.12 & 15.56 & 14.86 & 14.61 & 12.46 & 12.46 & 12.46 & 8.20 & 15.56 & 15.56 & 15.56 & 10.91 & 15.56 & 15.56 & 10.91 & 14.61 & 14.61 & 14.61 & 10.16 & 14.51 & 14.51 & 9.91 \\
\bottomrule
\end{tabular}
}
\caption{Simulation results in Design 14: $\mu_{j}(\theta)=-0.1\cdot 1[j>0.1 p]+0.05\cdot 1[j\leq 0.1 p]$, $\Sigma(\theta)$ Toeplitz.}
\label{tab14}
\end{table}
\end{landscape}
}
	
\begin{table}
\centering
\scalebox{0.75}{
\begin{tabular}{cccccccccccccc}
\toprule
Density & $p$ & $\rho$ & \multicolumn{2}{c}{Our method} & \multicolumn{9}{c}{\citetalias{chernozhukov/chetverikov/kato:2014c}'s methods} \\
\cmidrule(lr){6-14}
 & & & \multicolumn{2}{c}{Lasso selection} & \multicolumn{3}{c}{SN selection} & \multicolumn{3}{c}{MB selection} & \multicolumn{3}{c}{EB Selection} \\
\cmidrule(lr){4-5}\cmidrule(lr){6-8}\cmidrule(lr){9-11}\cmidrule(lr){12-14}
 & & & $\varepsilon=2/3$ & $\varepsilon=8/3$ & 0.01\% & 0.10\% & 1.00\% & 0.01\% & 0.10\% & 1.00\% & 0.01\% & 0.10\% & 1.00\% \\
\midrule
{$t_{4}/\sqrt{2}$} & 200 & 0 & 10.00 & 10.00 & 10.35 & 10.15 & 10.06 & 10.16 & 10.13 & 10.05 & 10.25 & 10.17 & 10.06 \\
 & & 0.5 & 10.00 & 10.00 & 10.26 & 10.11 & 10.04 & 10.12 & 10.10 & 10.04 & 10.17 & 10.11 & 10.04 \\
 & & 0.9 & 10.00 & 10.00 & 10.06 & 10.02 & 10.01 & 10.02 & 10.02 & 10.01 & 10.03 & 10.02 & 10.01 \\
 & 500 & 0 & 10.00 & 10.01 & 10.49 & 10.21 & 10.10 & 10.23 & 10.19 & 10.09 & 10.37 & 10.25 & 10.10 \\
 & & 0.5 & 10.00 & 10.01 & 10.37 & 10.17 & 10.08 & 10.18 & 10.14 & 10.07 & 10.27 & 10.18 & 10.07 \\
\bottomrule
\end{tabular}
}
\caption{Percentage of moment inequalities retained by first-step selection procedures in Design 8: $\mu_{j}(\theta)=-0.75\cdot 1[j>0.1 p]+0.05\cdot 1[j\leq 0.1 p]$, $\Sigma(\theta)$ Toeplitz.}
\label{tab15}
\end{table}

\begin{table}
\centering
\scalebox{0.75}{
\begin{tabular}{cccccccccccccc}
\toprule
Density & $p$ & $\rho$ & \multicolumn{2}{c}{Our method} & \multicolumn{9}{c}{\citetalias{chernozhukov/chetverikov/kato:2014c}'s methods} \\
\cmidrule(lr){6-14}
 & & & \multicolumn{2}{c}{Lasso selection} & \multicolumn{3}{c}{SN selection} & \multicolumn{3}{c}{MB selection} & \multicolumn{3}{c}{EB Selection} \\
\cmidrule(lr){4-5}\cmidrule(lr){6-8}\cmidrule(lr){9-11}\cmidrule(lr){12-14}
 & & & $\varepsilon=2/3$ & $\varepsilon=8/3$ & 0.01\% & 0.10\% & 1.00\% & 0.01\% & 0.10\% & 1.00\% & 0.01\% & 0.10\% & 1.00\% \\
\midrule
{$t_{4}/\sqrt{2}$} & 200 & 0 & 10.06 & 12.24 & 95.88 & 81.21 & 49.76 & 78.87 & 74.71 & 44.79 & 82.96 & 79.09 & 47.60 \\
 & & 0.5 & 10.05 & 12.53 & 96.32 & 82.13 & 50.27 & 79.80 & 75.76 & 45.03 & 82.90 & 78.99 & 47.15 \\
 & & 0.9 & 10.03 & 16.31 & 97.22 & 84.36 & 51.58 & 78.14 & 73.37 & 37.67 & 78.91 & 74.25 & 38.46 \\
 & 500 & 0 & 10.05 & 11.24 & 98.06 & 89.18 & 63.79 & 86.63 & 83.75 & 57.92 & 90.44 & 87.96 & 61.93 \\
 & & 0.5 & 10.04 & 11.34 & 98.33 & 89.97 & 64.65 & 87.31 & 84.48 & 58.51 & 90.38 & 87.87 & 61.67 \\
\bottomrule
\end{tabular}
}
\caption{Percentage of moment inequalities retained by first-step selection procedures in Design 11: $\mu_{j}(\theta)=-0.4\cdot 1[j>0.1 p]+0.05\cdot 1[j\leq 0.1 p]$, $\Sigma(\theta)$ Toeplitz.}
\label{tab16}
\end{table}
\newpage

\appendix


\section{Appendix}
\begin{small}

Throughout this section, we omit the dependence of all expressions on $\theta$ as this only complicates the notation without changing the technical arguments. Furthermore, ``s.t.'', ``LHS'', and ``RHS'' abbreviate ``such that'', ``left hand side'' and ``right hand side'', respectively. 

\subsection{Auxiliary results}


\begin{lemma} \label{lem:SampleBound}
Assume Assumptions \ref{ass:Basic}-\ref{ass:Rates}. Then, for any $\gamma $ s.t.\ $\sqrt{n}\gamma /\sqrt{1+\gamma ^{2}}\in [ 0,n^{\delta /(2(2+\delta ))}M_{n,2+\delta }^{-1}]$,
\begin{equation}
P\del[2]{\max_{j=1,\dots ,p}\vert \hat{\mu}_{j}-\mu _{j}\vert / \hat{\sigma} _{j}>\gamma} \leq 2p(1-\Phi (\sqrt{n}\gamma /\sqrt{ 1+\gamma ^{2}}))[ 1+Kn^{-\delta /2}M_{n,2+\delta }^{2+\delta }( 1+ \sqrt{n}\gamma /\sqrt{ 1+\gamma ^{2}}) ^{2+\delta }] , \label{eq:Lemma1Eq1}
\end{equation}
where $K$ is a universal constant.
\end{lemma}
\begin{proof}
	
For any $i=1,\dots ,n$ and $j=1,\dots ,p$, let $Z_{ij}\equiv ( X_{ij}-\mu _{j})/\sigma _{j}$ and $U_{j}\equiv \sqrt{n} \sum_{i=1}^{n} (Z_{ij}/n)/\sqrt{ \sum_{i=1}^{n} (Z_{ij}^{2}/n)}$. We divide the rest of the proof into three steps.

\underline{Step 1.} By definition, $\sqrt{n}(\hat{\mu}_{j}-\mu _{j})/\hat{\sigma}_{j}=U_{j}/\sqrt{ 1-U_{j}^{2}/n}$ and so
\begin{equation}
	\sqrt{n}|\hat{\mu}_{j}-\mu _{j}|/\hat{\sigma}_{j}=|U_{j}|/\sqrt{ 1-|U_{j}|^{2}/n}.\label{eq:reexpressionU}
\end{equation}
Since the RHS of Eq.\ \eqref{eq:reexpressionU} is increasing in $|U_{j}|$, it follows that
\begin{eqnarray}
\cbr[2]{\max_{j=1,\dots ,p}|\hat{\mu}_{j}-\mu _{j}|/\hat{\sigma}_{j}>\gamma } =\cbr[2]{\max_{1\leq j\leq p}|U_{j}|/\sqrt{1-|U_{j}|^{2}/n}>\sqrt{n}\gamma } 
\subseteq \cbr[2]{\max_{1\leq j\leq p}|U_{j}|\geq \sqrt{n}\gamma /\sqrt{1+\gamma ^{2}}}.
\label{eq:SNbound0}
\end{eqnarray}

\underline{Step 2.} For every $j=1,\dots ,p,$ $\{Z_{ij}\}_{i=1}^{n}$ is a sequence of independent random variables with $E[Z_{ij}]=0$ and $1=E[Z_{ij}^{2}]^{1/2}\leq E[|Z_{ij}|^{2+\delta }]^{1/(2+\delta)}\leq M_{n,2+\delta }<\infty $. If we let $ S_{nj}=\sum_{i=1}^{n}Z_{ij}$, $V_{nj}^{2}=\sum_{i=1}^{n}Z_{ij}^{2}$, and $ 0<D_{nj}=[n^{-1}\sum_{i=1}^{n}E[|Z_{ij}|^{2+\delta }]]^{1/(2+\delta )}\leq M_{n,2+\delta }<\infty $, then  \citet[Lemma D.1, page 48]{chernozhukov/chetverikov/kato:2014c_supp} (applied with $\nu=\delta$) implies that for all $t\in [ 0,n^{\delta /(2(2+\delta ))}D_{nj}^{-1}]$,
\begin{equation}
	\envert[2]{\frac{P(S_{nj}/V_{nj}\geq t)}{1-\Phi (t)}-1}\leq Kn^{-\delta /2}D_{nj}^{2+\delta }(1+t)^{2+\delta },
	\label{eq:SNbound1}
\end{equation}
where $K$ is a universal constant.

By using that $U_{j} = S_{nj}/V_{nj}$, $D_{nj}\leq M_{n,2+\delta }$, and applying Eq.\ \eqref{eq:SNbound1} to $t=\sqrt{n}\gamma /\sqrt{1+\gamma ^{2}}$, it follows that for any $\gamma $ s.t.\ $\sqrt{n}\gamma /\sqrt{1+\gamma ^{2}}\in [ 0,n^{\delta /(2(2+\delta ))}M_{n,2+\delta }^{-1}]$,
\begin{eqnarray*}
|P(U_{j}\geq \sqrt{n}\gamma /\sqrt{1+\gamma ^{2}})-(1-\Phi (\sqrt{n}\gamma / \sqrt{1+\gamma ^{2}}))| \leq 
Kn^{-\delta /2}D_{nj}^{2+\delta }(1-\Phi (\sqrt{ n }\gamma /\sqrt{1+\gamma ^{2}}))(1+\sqrt{n}\gamma /\sqrt{1+\gamma ^{2}})^{2+\delta }.
\end{eqnarray*}
Thus, for any $\gamma $ s.t.\ $\sqrt{n}\gamma /\sqrt{1+\gamma ^{2}}\in [ 0,n^{\delta /(2(2+\delta ))}M_{n,2+\delta }^{-1}]$,
\begin{eqnarray}
\sum_{j=1}^{p}P\del[1]{U_{j}\geq \sqrt{n}\gamma /\sqrt{1+\gamma ^{2}}}\leq
p\del[1]{1-\Phi (\sqrt{n}\gamma /\sqrt{1+\gamma ^{2}})}\sbr[1]{1+Kn^{-\delta /2}M_{n,2+\delta }^{2+\delta }(1+\sqrt{n}\gamma /\sqrt{1+\gamma ^{2}})^{2+\delta }}. \label{eq:SNbound2}
\end{eqnarray}
By applying the same argument for $-Z_{ij}$ instead of $Z_{ij}$, it follows that for any $\gamma $ s.t.\ $\sqrt{n}\gamma /\sqrt{1+\gamma ^{2}}\in [ 0,n^{\delta /(2(2+\delta ))}M_{n,2+\delta }^{-1}]$,
\begin{eqnarray}
\sum_{j=1}^{p}P\del[1]{-U_{j}\geq \sqrt{n}\gamma /\sqrt{1+\gamma ^{2}}}\leq 
p\del[1]{1-\Phi (\sqrt{n}\gamma /\sqrt{1+\gamma ^{2}})}\sbr[1]{1+Kn^{-\delta /2}M_{n,2+\delta }^{2+\delta }(1+\sqrt{n}\gamma /\sqrt{1+\gamma ^{2}})^{2+\delta }}. \label{eq:SNbound3}
\end{eqnarray}

\underline{Step 3.} Consider the following argument.
\begin{eqnarray*}
P\del[2]{\max_{j=1,\dots ,p}|\hat{\mu}_{j}-\mu _{j}|/\hat{\sigma}_{j}>\gamma }
&\leq &P\del[2]{\max_{j=1,\dots ,p}|U_{j}|\geq \sqrt{n}\gamma /\sqrt{1+\gamma ^{2}} } \\
&\leq &\sum_{j=1}^{p}P\del[2]{|U_{j}|\geq \sqrt{n}\gamma /\sqrt{1+\gamma ^{2}}} \\
&\leq &\sum_{j=1}^{p}P\del[2]{U_{j}\geq \sqrt{n}\gamma /\sqrt{1+\gamma ^{2}} }+\sum_{j=1}^{p}P\del[2]{-U_{j}\geq \sqrt{n}\gamma /\sqrt{1+\gamma ^{2}}} \\
&\leq &2p\del[2]{1-\Phi (\sqrt{n}\gamma /\sqrt{1+\gamma ^{2}})}\sbr[2]{1+Kn^{-\delta /2}M_{n,2+\delta }^{2+\delta }\del[1]{1+\sqrt{n}\gamma /\sqrt{1+\gamma ^{2}} }^{2+\delta }},
\end{eqnarray*}
where the first inequality follows from Eq.\ \eqref{eq:SNbound0} and the fourth inequality follows from Eqs.\ \eqref{eq:SNbound2} and \eqref{eq:SNbound3}.
\end{proof}


\begin{lemma}\label{lem:SampleBound2} 
Assume Assumptions \ref{ass:Basic}-\ref{ass:Rates} and let $\{\gamma _{n}\}_{n\geq 1}$ denote a sequence in $\mathbb{R}$ that satisfies $\gamma _{n}\geq \gamma _{n}^{\ast }$ with
\begin{equation}
\gamma _{n}^{\ast }\equiv n^{-1/2}(M_{n,2+\delta }^{2}n^{-\delta /(2+\delta )}-n^{-1})^{-1/2}=(nM_{n,2+\delta }^{2+\delta })^{-1/(2+\delta )}(1-(nM_{n,2+\delta }^{2+\delta })^{-2/(2+\delta )})^{-1/2} \downarrow 0. 
\label{Eq:LambdaStar}
\end{equation}
Then, 
\begin{align}
&P\left(\max_{j=1,\dots ,p}|\hat{\mu}_{j}-\mu _{j}|/\hat{\sigma}_{j}>\gamma _{n}]\right)\notag\\
& \leq 2\exp [ \ln (2k-p)( 1-n^{(2c+\delta -1)/(2+\delta)}/( 2C^{2/(2+\delta)}) ) ] [1+K( ( Cn^{-c+( 1-\delta ) /2}) ^{1/( 2+\delta ) }+1) ^{2+\delta }]\to 0,
\label{eq:SampleBoundEq}
\end{align}
where $K,C>0$ and $c>(1-\delta)/2$ are the universal constants in Lemma \ref{lem:SampleBound} and Assumption \ref{ass:Rates}.
\end{lemma}
\begin{proof}
The equality in Eq.\ \eqref{Eq:LambdaStar} follows from algebra. The convergence in Eq.\ \eqref{Eq:LambdaStar} follows from $M_{n,2+\delta }\geq 1$ and so $n M_{n,2+\delta }^{2+\delta }\to \infty$. In turn, $M_{n,2+\delta }\geq 1$ follows from H\"{o}lder's inequality and the definition of $Z_{1,j}$, as it implies that $E( \vert Z_{1j}\vert ^{2+\delta }) ^{1/( 2+\delta ) }\geq E( Z_{1j}^{2}) ^{1/2}=1$.

To show Eq.\ \eqref{eq:SampleBoundEq}, consider the following preliminary derivation.
\begin{align}
P( \max_{j=1,\dots ,p}|\hat{\mu}_{j}-\mu _{j}|/\hat{\sigma}_{j}>\gamma _{n}) &\leq P(\max_{j=1,\dots ,p}|\hat{\mu}_{j}-\mu _{j}|/\hat{\sigma }_{j}>\gamma _{n}^{\ast }) \notag\\
&\leq 2p[1-\Phi (n^{\delta /(2(2+\delta ))}M_{n,2+\delta }^{-1})][1+Kn^{-\delta /2}M_{n,2+\delta }^{2+\delta }(1+n^{\delta /(2(2+\delta ))}M_{n,2+\delta }^{-1})^{2+\delta }] \notag\\
&\leq 2p\exp [-2^{-1}n^{\delta /(2+\delta )}/M_{n,2+\delta }^{2}][1+K(n^{-\delta /(2(2+\delta ))}M_{n,2+\delta }+1)^{2+\delta }],\label{eq:SampleBoundEq_proof}
\end{align}
where the first inequality follows from $\gamma _{n}\geq \gamma _{n}^{\ast }$ , the second inequality follows from Lemma \ref{lem:SampleBound} with $ \gamma =\gamma _{n}^{\ast }$, and the third inequality follows from $1-\Phi (t)\leq \exp({-t^{2}/2})$. Note that the choice $\gamma =\gamma _{n}^{\ast }$ implies that $\sqrt{n}\gamma _{n}^{\ast }/\sqrt{1+(\gamma _{n}^{\ast })^{2}} =n^{\delta /(2(2+\delta ))}M_{n,2+\delta }^{-1}\geq 0$, and so this choice of $\gamma $ lies in the upper bound of the interval $[ 0,n^{\delta /(2(2+\delta ))}M_{n,2+\delta }^{-1}]$ in Lemma \ref{lem:SampleBound}. To complete the proof, it suffices to show that the RHS of Eq.\ \eqref{eq:SampleBoundEq_proof} is bounded by the RHS of Eq.\ \eqref{eq:SampleBoundEq}. To this end, consider first the following derivation.
\begin{align}
2p\exp [-2^{-1}n^{\delta /(2+\delta )}/M_{n,2+\delta }^{2}]& \leq 2(2k-p)\exp [-2^{-1}n^{\delta /(2+\delta )}/M_{n,2+\delta }^{2}] \notag\\
& =2\exp [\ln (2k-p)( 1-2^{-1}( n^{\delta /(2+\delta )}/( M_{n,2+\delta }^{2}\ln (2k-p)) ) ) ] \notag\\
& \leq 2\exp [ \ln (2k-p)( 1-n^{(2c+\delta -1)/(2+\delta)}/( 2C^{2/(2+\delta)}) ) ] \to 0. \label{eq:conv1}
\end{align}
for some $c >(1-\delta)/2$ and $C>0$, where the first inequality follows from $2k-p\geq p$ and the second inequality follows from $2k-p>1$ and Assumption \ref{ass:Rates}.
Next, consider the following derivation.
\begin{equation}
( n^{-\delta /(2(2+\delta ))}M_{n,2+\delta }) ^{2+\delta }=M_{n,2+\delta }^{2+\delta }n^{-\delta /2}\leq M_{n,2+\delta }^{2+\delta }( \ln ( 2k-p) ) ^{( 2+\delta ) /2}n^{-\delta /2}\leq Cn^{-c+( 1-\delta ) /2}\to 0
\label{eq:conv2}
\end{equation}
for some $c >(1-\delta)/2$ and $C>0$, where the first inequality follows from $2k-p\geq 1$, and the second inequality and the convergence follow from Assumption \ref{ass:Rates}. In turn, Eq.\ \eqref{eq:conv2} implies that
\begin{equation}
 1+K( n^{-\delta /(2(2+\delta ))}M_{n,2+\delta }+1) ^{2+\delta } \leq 1+K( ( Cn^{( 1-\delta ) /2-c}) ^{1/( 2+\delta ) }+1) ^{2+\delta }\to 1.
\label{eq:conv3}
\end{equation}
The desired result then follow from Eqs.\ \eqref{eq:conv1} and \eqref{eq:conv3}.
\end{proof}

\begin{lemma}\label{lem:A3implications}
Assumption \ref{ass:Rates2} implies that $B_{n}^{2}( \ln (2k-p)) n^{-( 1-c) /2}\to 0$ and $B_{n}^{2}(\ln (2k-p))^{2}n^{-3/2}\to 0$.
\end{lemma}
\begin{proof}
The first result follows from the next derivation.
\[
B_{n}^{2}( \ln (2k-p)) n^{-( 1-c) /2}~\leq~ Cn^{-c/2}\to 0,
\]
where the inequality holds by Assumption \ref{ass:Rates2} and the convergence occurs by $ c>0$.

To complete the proof, we now show the second result. The result is immediate if we have $2k-p=1$, so we focus the remainder of the proof on $ 2k-p\geq 2$. As a first step, we show that $B_{n}\geq 1$. By H\"{o}lder's inequality and the definition of $Z_{1j}$, $1=E[Z_{1j}^{2}]^{1/2}\leq E[Z_{1j}^{4}]^{1/4}\leq E[\max_{j=1,\ldots ,k}|Z_{1j}|^{4}]^{1/4}$. From here, $B_{n}\geq 1$ follows. As a second step, we show that $(\ln (2k-p))/n\to 0$. To show this, consider the following argument.
\[
(\ln (2k-p))/n~\leq~ (B_{n}^{2}(\ln (2k-p))n^{-( 1-c) /2} ) n^{-( 1+c) /2}\to 0,
\]
where the inequality follows from $B_{n}^{2}\geq 1$ and the convergence follows from $B_{n}^{2}( \ln (2k-p)) n^{-( 1-c) /2}\to 0$. Finally, consider the following derivation.
\[
n^{-3/2}(\ln (2k-p))^{2}B_{n}^{2}~=~( B_{n}^{2}(\ln (2k-p))n^{-( 1-c) /2}) (\ln (2k-p))/n)n^{-c/2}~\to~ 0,
\]
where the convergence follows $B_{n}^{2}( \ln (2k-p)) n^{-( 1-c) /2}\to 0$ and $(\ln (2k-p))/n\to 0$.
\end{proof}


\begin{lemma}\label{lem:A4impliesA2andA3}
Assumption \ref{ass:Rates3} implies Assumption \ref{ass:Rates} with $\delta =1$ and Assumption \ref{ass:Rates2}.
\end{lemma}
\begin{proof} 
	First, we show that Assumption \ref{ass:Rates3} with constants $C$ and $c$ implies Assumption \ref{ass:Rates} for $\delta =1$ and with constants $\tilde{C}=C (\ln 2)^{-2}>0$ and $c$. Next, consider the following derivation.
\begin{eqnarray*}
( M_{n,3}^{3}) ^{2}( \ln ( 2k-p) ) ^{3/2} &=&( M_{n,3}^{3}) ^{2}( \ln ( ( 2k-p) n) ) ^{7/2}( \ln ( 2k-p) ) ^{3/2}( \ln ( ( 2k-p) n) ) ^{-7/2} \\
&\leq &Cn^{1/2-c}( \ln ( 2k-p) ) ^{3/2}( \ln ( 2k-p) +\ln n) ^{-7/2} \\
&\leq &Cn^{1/2-c}( \ln ( 2k-p) ) ^{-2} \\
&\leq &Cn^{1/2-c}(\ln 2)^{-2}= \tilde{C}n^{1/2-c},
\end{eqnarray*}
where the first equality and the third inequality follow from $2k-p>1$, the first inequality follows from Assumption \ref{ass:Rates3}, and the second inequality follows from $n\geq 1$.

Second, we show that Assumption \ref{ass:Rates3} with constants $C>0$ and $c\in ( 0,1/2) $ implies Assumption \ref{ass:Rates2} with constants $\tilde{C}= C (\ln 2)^{-5/2}>0$ and $c$. Consider the following derivation.
\begin{eqnarray*}
( B_{n}) ^{2}\ln ( 2k-p) &=&( B_{n}) ^{2}( \ln ( ( 2k-p) n) ) ^{7/2}\ln ( 2k-p) ( \ln ( ( 2k-p) n) ) ^{-7/2} \\
&\leq &Cn^{1/2-c}\ln ( 2k-p) ( \ln ( 2k-p) +\ln n) ^{-7/2} \\
&\leq &Cn^{1/2-c}( \ln ( 2k-p) ) ^{-5/2} \\
&\leq &Cn^{1/2-c}(\ln 2)^{-5/2}= \tilde{C}n^{1/2-c},
\end{eqnarray*}
where the first equality and the third inequality follow from $2k-p>1$, the first inequality follows from Assumption \ref{ass:Rates3}, and the second inequality follows from $n\geq 1$, and the last equality follows from definition of $\tilde{C}$.
\end{proof}


\begin{proof}[Proof of Lemma \ref{lem:LassoNoOverFit}]
As a preliminary step, we show that for any $r\in (0,1)$,
\begin{equation}
\cbr[2]{\{ J_{I}\not\subseteq \hat{J}_{L}\} \cap \{ \sup_{j=1,\dots ,p}\vert \hat{\sigma}_{j}/\sigma _{j}-1\vert \leq r/(1+r) \} }~\subseteq ~ \cbr[2]{ \sup_{j=1,\dots ,p}\vert \hat{\mu}_{j}-\mu _{j}\vert /\hat{\sigma}_{j}>\lambda _{n}(1-r)3/4}. \label{eq:lasso1}
\end{equation}
To show this, consider the following argument. Suppose that $j\in J_{I}$ and $j\not\in \hat{J}_{L}$, i.e., $\mu _{j}/\sigma _{j}\geq -\lambda _{n}3/{4}$ and $\hat{\mu}_{L,j}/\hat{\sigma}_{j}<-\lambda _{n}$ or, equivalently by Eq.\ \eqref{eq:ClosedForm2}, $\hat{\mu}_{j}/\hat{\sigma}_{j}<-\lambda _{n}{3} /{2}$. Then, $|\mu _{j}-\hat{\mu}_{j}|/\hat{\sigma}_{j}>\lambda _{n}({3 }/{2}-(\sigma _{j}/\hat{\sigma}_{j}){3}/{4})$. In turn, $\sup_{j=1,\dots ,p}|1-\hat{\sigma}_{j}/\sigma _{j}|\leq r/(1+r)$ implies that $|\sigma _{j}/ \hat{\sigma}_{j}-1|\leq r$ and so $\lambda _{n}({3}/{2}-(\sigma _{j}/\hat{\sigma}_{j}){3}/{4})\geq \lambda _{n}(1-r)3/4$. By combining these, we conclude that $\sup_{j=1,\dots ,p}|\hat{\mu}_{j}-\mu _{j}|/\hat{ \sigma}_{j}>\lambda _{n}(1-r)3/4$, as desired.

Then, consider the following derivation for any $r\in (0,1)$, 
\begin{eqnarray}
P(J_{I}\not\subseteq \hat{J}_{L}) &=&\left\{ 
\begin{array}{c}
P( \{J_{I}\not\subseteq \hat{J}_{L}\}\cap \{\sup_{j=1,\dots ,p}|\hat{
\sigma}_{j}/\sigma _{j}-1|\leq r/(1+r)\})  \\ 
+P( \{J_{I}\not\subseteq \hat{J}_{L}\}\cap \{\sup_{j=1,\dots ,p}|\hat{
\sigma}_{j}/\sigma _{j}-1|>r/(1+r)\}) 
\end{array}
\right\}   \notag \\
&\leq &P\del[2]{\sup_{j=1,\dots ,p}\vert \hat{\mu}_{j}-\mu _{j}\vert
/\hat{\sigma}_{j}>\lambda _{n}(1-r)3/4}+P\del[2]{\sup_{j=1,\dots ,p}\vert
\hat{\sigma }_{j}/\sigma _{j}-1\vert >r/(1+r) },\label{eq:lasso2}
\end{eqnarray}
where the inequality follows from Eq.\ \eqref{eq:lasso1}.

In the remainder of this proof, we want to consider Eq.\ \eqref{eq:lasso2} with $r=r_{n}\equiv (((\ln ( 2k-p) n^{-(1-c)/2}+(\ln ( 2k-p) )^{2})n^{-3/2})B_{n}^{2})^{-1}-1)^{-1}\in ( 0,1) $. By Lemma \ref{lem:A3implications}, $B_{n}^{2}\ln ( 2k-p) n^{-(1-c)/2}+B_{n}^{2}(\ln ( 2k-p) )^{2})n^{-3/2}\to 0$, and so $r_{n}\to 0$. Since $r_{n}\to 0$, we deduce that $ \exists \bar{n}\in \mathbb{N} $ s.t.\ $\forall n\geq \bar{n}$, $( 1-r_{n}) ( 4/3+\varepsilon ) \geq 4/3$, where $\varepsilon >0$ is as in Eq.\ \eqref{eq:LambdaConcrete}.

By evaluating Eq.\ \eqref{eq:step1_r_step} with $r=r_{n}$, we deduce that for all $n\geq \bar{n}$,
\begin{align}
&P(J_{I}\not\subseteq \hat{J}_{L})~\leq ~2p\exp (-2^{-1}n^{\delta /(2+\delta )}/M_{n,2+\delta }^{2})[1+K(M_{n,2+\delta }/n^{\delta /(2(2+\delta ))}+1)^{2+\delta }]+\tilde{K}n^{-c}\notag\\
&\leq 2\exp [ \ln (2k-p)( 1-n^{(2c+\delta -1)/(2+\delta)}/( 2C^{2/(2+\delta)}) ) ] [1+K( ( Cn^{-c+( 1-\delta ) /2}) ^{1/( 2+\delta ) }+1) ^{2+\delta }]+\tilde{K}n^{-c}
\to 0 ,\label{eq:lasso3}
\end{align}
where the first inequality holds by \citet[Lemma D.5, page 52]{chernozhukov/chetverikov/kato:2014c_supp} and Lemma \ref{lem:SampleBound2} with $ \gamma _{n}=( 1-r_{n}) \lambda _{n}3/4$  (which applies because $ \gamma _{n}\geq \gamma _{n}^{\ast }=n^{-1/2}(M_{n,2+\delta }^{2}n^{-\delta /(2+\delta )}-n^{-1})^{-1/2}$ when $n\geq \bar{n}$), and the second inequality follows from repeating the arguments in Lemma \ref{lem:SampleBound2}. Finally, note that by appropriately adjusting the constants $C$ and $\tilde{K}$, Eq.\ \eqref{eq:lasso3} can be extended to $n <\bar{n}$. 
\end{proof}

\begin{proof}[Proof of Lemma \ref{lem:LassoClosedForm}]
	Fix $j=1,\dots ,p$ arbitrarily. \citet[Eq.\ (2.5)]{buhlmann/vandegeer:2011} implies that the Lasso estimator in Eq.\ \eqref{eq:Lasso1} satisfies
	\begin{align}
	\hat{\mu}_{L,j} ~=~ \sgn(\hat{\mu}_{j}) \times \max \{|\hat{\mu}_j|-\hat{\sigma}_j{\lambda_n}/{2}, 0\}~~~\forall j=1,\dots,p. \label{eq:ClosedForm1}
	\end{align}
	
	To complete the proof, it suffices to show that
	\begin{equation}
	\{\hat{\mu}_{L,j}\geq -\hat{\sigma}_{j}\lambda _{n}\}~~=~~ \{\hat{\mu }_{j}\geq -3\hat{\sigma}_{j}\lambda _{n}/2\}. \label{eq:ClosedForm2}
	\end{equation}
	We divide the verification into four cases. First, consider that $\hat{\sigma }_{j}=0$. If so, $-\hat{\sigma}_{j}\lambda _{n}=-3\hat{\sigma} _{j}\lambda _{n}/2=0$ and $\hat{\mu}_{L,j}=sign( \hat{\mu} _{j}) \times \max \{ \vert \hat{\mu}_{j}\vert ,0\} =\hat{\mu}_{j}$, and so Eq.\ \eqref{eq:ClosedForm2} holds. Second, consider that $\hat{\sigma}_{j}>0$ and $\hat{\mu}_{j}\geq 0$. If so, $\hat{\mu}_{j}\geq 0\geq -3\hat{\sigma}_{j}\lambda _{n}/2$ and so the RHS condition in Eq.\ \eqref{eq:ClosedForm2} is satisfied. In addition, Eq.\ \eqref{eq:ClosedForm1} implies that $\hat{\mu}_{L,j}\geq 0\geq -\hat{\sigma} _{j}\lambda _{n}$ and so the LHS of condition in Eq.\ \eqref{eq:ClosedForm2} is also satisfied. Thus, Eq.\ \eqref{eq:ClosedForm2} holds. Third, consider that $\hat{\sigma}_{j}>0$ and $\hat{\mu}_{j}\in [ -\hat{\sigma} _{j}\lambda _{n}/2,0)$. If so, $\hat{\mu}_{j}\geq - \hat{\sigma}_{j}\lambda _{n}/2\geq -3\hat{\sigma}_{j}\lambda _{n}/2$ and so the RHS condition in Eq.\ \eqref{eq:ClosedForm2} is satisfied. In addition, Eq.\ \eqref{eq:ClosedForm1} implies that $\hat{\mu}_{L,j}=0\geq -\hat{\sigma}_{j}\lambda _{n}$ and so the LHS of condition in Eq.\ \eqref{eq:ClosedForm2} is also satisfied. Thus, Eq.\ \eqref{eq:ClosedForm2} holds. Fourth and finally, consider that $\hat{\sigma}_{j}>0$ and $\hat{\mu}_{j}<-\hat{\sigma}_{j}\lambda _{n}/2$. Then, Eq.\ \eqref{eq:ClosedForm1} implies that $\hat{\mu}_{L,j}=\hat{\mu}_{j}+\hat{\sigma} _{j}\lambda _{n}/2$ and so Eq.\ \eqref{eq:ClosedForm2} holds.
\end{proof}

\subsection{Results for the self-normalization approximation}

\begin{lemma} \label{lem:CSNincreasing}
For any $\pi \in (0,0.5]$, $n\in \mathbb{N}$, and $d\in \{0,1\dots ,2k-p\}$, define the function
\[
CV(d)\equiv \left\{
\begin{array}{ll}
0 & \text{if }d=0 ,\\
\frac{\Phi ^{-1}( 1-\pi /d) }{\sqrt{1-( \Phi ^{-1}( 1-\pi /d) ) ^{2}/n}} & \text{if }d>0.
\end{array}
\right.
\]
Then, $CV:\{0,1\dots, 2k-p\}\to \mathbb{R}_{+}$ is weakly increasing for all $n$ sufficiently large.
\end{lemma}
\begin{proof}
First, we show that $CV(d)\leq CV(d+1)$ for $d=0$. To see this, use that $\pi \leq 0.5$ such that $\Phi ^{-1}( 1-\pi) \geq 0$, implying that $CV(1)\geq 0=CV(0)$.

Second, we show that $CV(d)\leq CV(d+1)$ for any $d>0$. To see this, notice that $CV(d)$ and $CV(d+1)$ are both the result of the composition $g_{1}(g_{2}(d )):\{ 1\dots ,2k-p\} \to \mathbb{R}$, where
\begin{eqnarray*}
	g_{1}(y) &\equiv&y/\sqrt{1-y^{2}/n}:[0,\sqrt{n})\to \mathbb{R}_{+} \\
	g_{2}(d) &\equiv&\Phi ^{-1}( 1-{\pi }/{d}) :\{ 1\dots ,2k-p\} \to \mathbb{R}.
\end{eqnarray*}
We first show that $g_{1}(g_{2}(d ))$ is properly defined by verifying that the range of $g_{2}$ is included in support of $g_{1}$. Notice that $g_{2}$ is an increasing function and so $ g_{2}(d)\in \lbrack g_{2}(1),g_{2}(2(k-p)+p)]=[\Phi ^{-1}( 1-\pi ) ,\Phi ^{-1}( 1-\pi /(2k-p)) ]$. For the lower bound, $\pi \leq 0.5$ implies that $\Phi ^{-1}( 1-\pi ) \geq 0$. For the upper bound, consider the following argument. On the one hand, $( 1-\Phi ( \sqrt{n} )) \leq \exp ( -n/2)/2$ holds for all $n$ large enough. On the other hand, Assumption \ref{ass:Rates} implies that $\exp(-n/2)/2\leq {\pi}/(2k-p)$ holds for all $n$ large enough. By combining these two, we conclude that $\Phi ^{-1}( 1-\pi /(2k-p)) \leq \sqrt{n}$ for all $n$ large enough, as desired. From here, the monotonicity of $CV(d)$ follows from the fact that $g_{1}$ and $ g_{2}$ are both weakly increasing functions and so $CV(d)=g_{1}(g_{2}(d)) \leq g_{1}(g_{2}(d+1))=CV(d+1)$.
\end{proof}

%
%

\begin{lemma}\label{lem:SNSize}
Assume Assumptions \ref{ass:Basic}-\ref{ass:Rates}, $\alpha \in (0,0.5)$, and that $H_{0}$ holds. For any non-stochastic set $L\subseteq \{1,\dots ,p\}$, define
\begin{eqnarray*}
	T_{n}(L) &\equiv &\max \left\{ \max_{j\in L}{\sqrt{n}\hat{\mu}_{j}}/{\hat{ \sigma}_{j}},\max_{s=p+1,\dots ,k}{\sqrt{n}\vert \hat{\mu} _{s}\vert }/{\hat{\sigma}_{s}} \right\} \\
c_{n}^{SN}(|L|,\alpha) &\equiv &\tfrac{\Phi ^{-1}( 1-\alpha /(2(k-p)+|L|)) }{\sqrt{1-( \Phi ^{-1}( 1-\alpha /(2(k-p)+|L|)) ) ^{2}/n}}.
\end{eqnarray*}
Then,
\begin{equation*}
	P(T_{n}(L)>c_{n}^{SN}(|L|,\alpha)) ~\leq~ \alpha +\alpha 2^{1+\delta }KCn^{-c+( 1-\delta ) /2}[ ( \ln ( 2k-p) ) ^{-( 2+\delta ) /2}+{2} ^{1/2}( 1-( \ln \alpha ) /( \ln ( 2k-p) ) ) ^{(2+\delta )/2}] \to \alpha,
\end{equation*}
where $K,C>0$ and $c>(1-\delta)/2$ are the universal constants in Lemma \ref{lem:SampleBound} and Assumption \ref{ass:Rates}.
\end{lemma}
\begin{proof}
	Under $H_{0}$, $\sqrt{n}\hat{\mu}_{j}/\hat{\sigma}_{j}\leq \sqrt{n}( \hat{\mu}_{j}-\mu _{j}) /\hat{\sigma}_{j}$ for $j\in L$ and $\sqrt{n} \vert \hat{\mu}_{s}\vert /\hat{\sigma}_{s}=\sqrt{n}\vert \hat{\mu}_{s}-\mu _{s}\vert /\hat{\sigma}_{s}$ for $s=p+1,\dots ,k$. From this, we deduce that
\begin{eqnarray*}
T_{n}(L) &=&\max \cbr[2]{ \max_{j\in L}{\sqrt{n}\hat{\mu}_{j}}/{\hat{ \sigma}_{j}},\max_{s=p+1,\dots ,k}{\sqrt{n}\vert \hat{\mu} _{s}\vert }/{\hat{\sigma}_{s}}} \\
&\leq & T_{n}^{\ast }(L) \equiv \max \cbr[2]{ \max_{j\in L}{\sqrt{n}( \hat{\mu}_{j}-\mu _{j}) }/{\hat{\sigma}_{j}},\max_{s=p+1,\dots ,k}{\sqrt{n} \vert \hat{\mu}_{s}-\mu _{s}\vert }/{\hat{\sigma}_{s}}} .
\end{eqnarray*}
For any $i=1,\dots ,n$ and $j=1,\dots ,k$, let $Z_{ij}\equiv (X_{ij}-\mu _{j})/\sigma _{j}$ and $U_{j}\equiv \sqrt{n}
\sum_{i=1}^{n} (Z_{ij}/n)/\sqrt{\sum_{i=1}^{n} (Z_{ij}^2/n)}$. It then follows that $\sqrt{n}[\hat{\mu}_{j}-\mu _{j}]/\hat{\sigma}_{j}=U_{j}/\sqrt{ 1-U_{j}^{2}/n}$, and so
\begin{eqnarray*}
\sqrt{n}(\hat{\mu}_{j}-\mu _{j})/\hat{\sigma}_{j} &=&U_{j}/\sqrt{ 1-|U_{j}|^{2}/n} \\
\sqrt{n}|\hat{\mu}_{j}-\mu _{j}|/\hat{\sigma}_{j} &=&|U_{j}|/\sqrt{ 1-|U_{j}|^{2}/n}.
\end{eqnarray*}
Notice that the expressions on the RHS are increasing in $U_{j}$ and $ |U_{j}|$, respectively. Therefore, for any ${a}\geq 0 $,
\begin{align}
\{ T_{n}^{\ast }(L)>{a}\} &=\cbr[2]{ \max_{j\in L}{\sqrt{n} ( \hat{\mu}_{j}-\mu _{j}) }/{\hat{\sigma}_{j}}>{a}} \cup \cbr[2]{ \max_{s=p+1,\dots ,k}{\sqrt{n}\vert \hat{\mu}_{s}-\mu _{s}\vert }/{\hat{\sigma}_{s}}>{a}}\notag \\
&=\cbr[2]{ \max_{j\in L}U_{j}/\sqrt{1-|U_{j}|^{2}/n}>{a}} \cup \cbr[2]{ \max_{s=p+1,\dots ,k}|U_{s}|/\sqrt{1-|U_{s}|^{2}/n}>{a}} \notag\\
&=\cbr[2]{ \max_{j\in L}U_{j}>{a}/\sqrt{1+{a}^{2}/n}} \cup \cbr[2]{ \max_{s=p+1,\dots ,k}|U_{s}|>{a}/\sqrt{1+{a}^{2}/n}} .
\label{eq:auxSNSize}
\end{align}
From here, we conclude that for all ${a}\geq 0 $ s.t.\
 ${a}/\sqrt{1+{a}^{2}/n}\in [0,n^{\delta /(2(2+\delta ))}M_{n,2+\delta }^{-1}]$,
\begin{eqnarray}
P(T_{n}(L)>{a}) &\leq& P(T_{n}^{\ast }(L)>{a}) \notag \\
&\leq& P\del[3]{ \cbr[2]{ \max_{j\in L}U_{j}>{a}/\sqrt{1+{a}^{2}/n}} \cup \cbr[2]{ \max_{s=p+1,\dots ,k}|U_{s}|>{a}/\sqrt{1+{a}^{2}/n}} }\notag  \\
&\leq& \sum_{j\in L}P\del[2]{ U_{j}>{a}/\sqrt{1+{a}^{2}/n}} +\sum_{s=p+1}^{k}P\del[2]{ |U_{s}|>{a}/\sqrt{1+{a}^{2}/n}} \notag \\
&\leq& \sum_{j\in L}P\del[2]{ U_{j}>{a}/\sqrt{1+{a}^{2}/n}} +\sum_{s=p+1}^{k}P\del[2]{U_{s}>{a}/\sqrt{1+{a}^{2}/n}}+\sum_{g=p+1}^{k}P\del[2]{-U_{g}>{a}/ \sqrt{1+{a}^{2}/n}} \notag \\
&\leq& (2(k-p)+|L|) \del[2]{1-\Phi ({a}/\sqrt{1+{a}^{2}/n})} \sbr[2]{1+Kn^{-\delta /2}M_{n,2+\delta }^{2+\delta }(1+{a}/\sqrt{1+{a}^{2}/n})^{2+\delta }} \label{eq:KeyBound},
\end{eqnarray}
where the first inequality follows from $T_{n}(L)\leq T_{n}^{\ast }(L)$, the second inequality follows from Eq.\ \eqref{eq:auxSNSize}, and the last inequality follows from the same argument as in Step 2 of Lemma \ref{lem:SampleBound} upon choosing $\gamma={a}/\sqrt{n}$ in that result.

We are interested in applying Eq.\ \eqref{eq:KeyBound} with ${a}=c_{n}^{SN}(|L|,\alpha)$ which satisfies
\begin{equation}
	(2(k-p)+|L|) \del[2]{1-\Phi ( c_{n}^{SN}(|L|,\alpha)/\sqrt{ 1+c_{n}^{SN}(|L|,\alpha)^{2}/n})}= \alpha.
	\label{eq:CoverageDefn}
\end{equation}
Before doing this, we need to verify that this is a valid choice, i.e., we need to verify that, for all sufficiently large $n$,
\begin{equation}
	c_{n}^{SN}(|L|,\alpha)/\sqrt{1+c_{n}^{SN}(|L|,\alpha)^{2}/n} 
	~~\in~~ [0,n^{\delta /(2(2+\delta ))}M_{n,2+\delta }^{-1}].
\end{equation}
On the one hand, note that $c_{n}^{SN}(|L|,\alpha )\geq 0$ implies that $ c_{n}^{SN}(|L|,\alpha )/\sqrt{1+c_{n}^{SN}(\alpha ,|L|)^{2}/n}\geq 0$. On the other hand, note that, by definition, $c_{n}^{SN}(|L|,\alpha )/\sqrt{ 1+c_{n}^{SN}(|L|,\alpha )^{2}/n}=\Phi ^{-1}(1-\alpha /(2(k-p)+|L|))$ and so it suffices to show that $\Phi ^{-1}(1-\alpha /(|L|+2(k-p)))M_{n,2+\delta }n^{-\delta /(2(2+\delta ))}\to 0$. To show this, consider the following argument.
\begin{align*}
\Phi ^{-1}(1-\alpha /(2(k-p)+|L|))M_{n,2+\delta }n^{-\delta /(2(2+\delta ))} &\leq \sqrt{2\ln ((|L|+2(k-p))/\alpha )}M_{n,2+\delta }n^{-\delta /(2(2+\delta ))} \\
&\leq \sqrt{2\ln ((2k-p)/\alpha )}M_{n,2+\delta }n^{-\delta /(2(2+\delta ))} \\
&=[ 2( 1- \ln \alpha  /\ln (2k-p)) ( (\ln (2k-p))^{( 2+\delta ) /2}M_{n,2+\delta }^{( 2+\delta ) }n^{-\delta /2}) ^{2/( 2+\delta ) }] ^{1/2} \\
&\leq ( 2( Cn^{-c+( 1-\delta ) /2}) ^{2/( 2+\delta ) }) ^{1/2}\to 0,
\end{align*}
for some $C>0$ and $c>( 1-\delta ) /2$, where the first inequality uses that $1-\Phi (t)\leq \exp (-t^{2}/2)$ for any $t>0$, the second inequality follows from $|L|\leq p$, the equality follows from $2k-p>1 $, and the last inequality follows from $2k-p>1$ and Assumption \ref{ass:Rates}. From here, consider the following derivation.
\begin{align*}
P(T_{n}>c_{n}^{SN}(|L|,\alpha )) 
&\leq \alpha +\alpha Kn^{-\delta /2}M_{n,2+\delta }^{2+\delta }(1+\Phi ^{-1}(1-\alpha /(2(k-p)+|L|)))^{2+\delta }\\
&\leq \alpha + \alpha Kn^{-\delta /2}M_{n,2+\delta }^{2+\delta }(1+\Phi ^{-1}(1-\alpha /(2k-p)))^{2+\delta }\\
&\leq\alpha + \alpha 2^{1+\delta }Kn^{-\delta /2}M_{n,2+\delta }^{2+\delta }(1+|\Phi ^{-1}(1-\alpha /(2k-p))|^{2+\delta }) \\
&\leq \alpha +\alpha 2^{1+\delta }Kn^{-\delta /2}M_{n,2+\delta }^{2+\delta }(1+{2}^{1/2}(\ln ((2k-p)/\alpha ))^{(2+\delta )/2}) \\
&=\alpha +\left\{
\begin{array}{c}
\alpha 2^{1+\delta }K[ ( \ln ( 2k-p) ) ^{( 2+\delta ) /2}M_{n,2+\delta }^{2+\delta }n^{-\delta /2}] \times \\ 
(( \ln ( 2k-p) ) ^{-( 2+\delta ) /2}+{2}^{1/2}( 1-( \ln \alpha ) /( \ln ( 2k-p) ) ) ^{(2+\delta )/2})
\end{array}
\right\}\\
&\leq\alpha + \alpha 2^{1+\delta }KCn^{-c+( 1-\delta ) /2}[ ( \ln ( 2k-p) ) ^{-( 2+\delta ) /2}+{2} ^{1/2}( 1-( \ln \alpha ) /( \ln ( 2k-p) ) ) ^{(2+\delta )/2}], 
\end{align*}
where the first inequality follows from Eq.\ \eqref{eq:KeyBound} with ${a}=c_{n}^{SN}(|L|,\alpha )$ and Eq.\ \eqref{eq:CoverageDefn} and the second inequality follows from $|L|\leq p$ and that $f(x)\equiv \Phi ^{-1}(1-\alpha /(2(k-p)+x))$ is increasing, the third inequality follows from $ (1+a)^{2+\delta }\leq 2^{1+\delta }(1+a^{2+\delta })$ for any $a>0$ (this follows from Jensen's Inequality and that $f(x)\equiv x^{2+\delta }$ is convex), the fourth inequality follows from $1-\Phi (t)\leq \exp (-t^{2}/2)$ for any $t>0$ and so $\Phi ^{-1}(1-\alpha /(2k-p))\leq \sqrt{2\ln ((2k-p)/\alpha )}$, and the fifth inequality follows from Assumption \ref{ass:Rates}.
\end{proof}

\begin{proof}[Proof of Theorem \ref{thm:SN1Scorollary}]
	This result follows from Lemma \ref{lem:SNSize} with $L=\{1,\dots ,p\}$.
\end{proof}

\begin{proof}[Proof of Theorem \ref{thm:Lasso2SSize}]
This proof follows similar steps than \citetalias{chernozhukov/chetverikov/kato:2014c} (Proof of Theorem 4.2). Define the following sequence of sets:
\begin{equation*}
	J_{I} ~\equiv~ \{j=1,\dots ,p:\mu _{j}/\sigma _{j}\geq -3\lambda _{n}/4 \}.
\end{equation*}
We divide the proof into two steps.

\underline{Step 1.} We show that $\hat{\mu}_{j}\leq 0$ for all $j\in J_{I}^{c}$ with high probability, i.e.,
\[
P\del[1]{ \cup _{j\in J_{I}^{c}}\{\hat{\mu}_{j}>0\}}
\leq 2\exp[\ln(2k-p)( 1-n^{c-(1-\delta )/2}/( 2C) ] [1+K( ( Cn^{-c+( 1-\delta ) /2}) ^{1/( 2+\delta ) }+1) ^{2+\delta }]+\tilde{K}n^{-c}
\to 0 ,\label{eq:lasso3b}
\]
where $K,\tilde{K},C>0$ and $c>( 1-\delta ) /2$ are universal constants.

First, we show that for any $r\in (0,1)$,
\[
 \cbr[1]{\cup _{j\in J_{I}^{c}}\{\hat{\mu}_{j}>0\}} \cap \cbr[2]{ \sup_{j=1,\dots ,p}\vert \hat{\sigma}_{j}/\sigma _{j}-1\vert \leq r/(1+r) } \subseteq \cbr[2]{ \sup_{j=1,\dots ,p}\vert \hat{\mu}_{j}-\mu _{j}\vert /\hat{\sigma }_{j}>(1-r)\lambda _{n}3/4 }.
\]
To see this, suppose that there is an index $j=1,\dots ,p$ s.t.\ $\mu _{j}/\sigma _{j}<-\lambda _{n}3/4 $ and $\hat{\mu}_{j}>0$. Then, $\vert \hat{\mu}_{j}-\mu _{j}\vert /\hat{\sigma}_{j}>\lambda_{n}(3/4) ( \sigma _{j}/\hat{\sigma}_{j}) $. In turn, $ \sup_{j=1,\dots ,p}\vert 1- \hat{\sigma}_{j}/\sigma _{j}\vert \leq r/(1+r)$ implies that $\vert 1- \sigma _{j}/\hat{\sigma} _{j}\vert \leq r$ and so $( \sigma _{j}/\hat{\sigma}_{j})\lambda _{n}3/4 \geq  ( 1-r)\lambda _{n}3/4 $. By combining these, we conclude that $ \sup_{j=1,\dots ,p}\vert \hat{\mu}_{j}-\mu _{j}\vert /\hat{\sigma }_{j}>( 1-r)\lambda _{n}(3/4) $.

Then, consider the following derivation for any $r\in (0,1)$,
\begin{align}
P( \cup _{j\in J_{I}^{c}}\{\hat{\mu}_{j}>0\})~&=~\left\{
\begin{array}{c}
P( \cup _{j\in J_{I}^{c}}\{\hat{\mu}_{j}>0\}\cap \sup_{j=1,\dots ,p}\vert \hat{\sigma}_{j}/\sigma _{j}-1\vert \leq r/(1+r)) + \\
P( \cup _{j\in J_{I}^{c}}\{\hat{\mu}_{j}>0\}\cap \sup_{j=1,\dots ,p}\vert \hat{\sigma}_{j}/\sigma _{j}-1\vert >r/(1+r))
\end{array}
\right\} \notag\\
~&\leq~ P\del[2]{\sup_{j=1,\dots ,p}\vert \hat{\mu}_{j}-\mu _{j}\vert /\hat{\sigma}_{j}> ( 1-r)\lambda _{n}3/4 } +
P\del[2]{ \sup_{j=1,\dots ,p}\vert \hat{\sigma }_{j}/\sigma _{j}-1\vert >r/(1+r)}.\label{eq:step1_r_step}
\end{align}
The result then follows from evaluating Eq.\ \eqref{eq:step1_r_step} with $r=r_{n}=(((n^{-( 1-c) /2}\ln p+n^{-3/2}( \ln p) ^{2}) B_{n}^{2}) ^{-1}-1)^{-1}\to 0$ and repeating the argument used in the proof of Lemma \ref{lem:LassoNoOverFit}.

\underline{Step 2.} We now complete the argument. Consider the following derivation.
\begin{align}
\left\{ \{ T_{n}>c_{n}^{SN,L}(\alpha) \}\cap \{J_{I}\subseteq \hat{J}_{L}\} \cap \{\cap _{j\in J_{I}^{c}} \{\hat{\mu}_{j}\leq 0\}\} \right\} &\subseteq \left\{ \{T_{n}>c_{n}^{SN}(|J_{I}|,\alpha )\} \cap \{\cap _{j\in J_{I}^{c}}\cbr[1]{\hat{\mu}_{j}\leq 0}\} \right\}\notag\\
&\subseteq \cbr[3]{\max \cbr[2]{ \max_{j\in J_{I}}\frac{\sqrt{n}\hat{\mu}_{j}}{\hat{ \sigma}_{j}},\max_{s=p+1,\dots ,k}\frac{\sqrt{n}\vert \hat{\mu} _{s}\vert }{\hat{\sigma}_{s}}}>c_{n}^{SN}(\alpha ,|J_{I}|)} ,\label{eq:step2_1}
\end{align}
where we have used $c_{n}^{SN,L}(\alpha) = c_{n}^{SN}(\alpha ,|\hat{J}_L|)$, Lemma \ref{lem:CSNincreasing} (in that $c_{n}^{SN}(\alpha ,d)$ is a non-negative increasing function of $d\in \{0,1\dots ,2k-p\}$), and we take $\max_{j\in J_{I}}\sqrt{n}\hat{\mu}_{j}/\hat{\sigma}_{j}=-\infty $ if $ J_{I}=\emptyset $. Thus,
\begin{align*}
&P( T_{n}>c_{n}^{SN,L}(\alpha)) =\left\{
\begin{array}{c}
P( \{ T_{n}>c_{n}^{SN,L}(\alpha)\} \cap \{ \{ J_{I}\subseteq \hat{J}_{L}\} \cap \{ \cap _{j\in J_{I}^{c}}\{\hat{\mu}_{j}\leq 0\}\} \} ) + \\
P( \{ T_{n}>c_{n}^{SN,L}(\alpha)\} \cap \{ \{ J_{I}\not\subseteq \hat{J}_{L}\} \cup \{ \cup _{j\in J_{I}^{c}}\{\hat{\mu}_{j}>0\}\} \} )
\end{array}
\right\} \\
&\leq 
P\del[3]{ \max \cbr[2]{ \max_{j\in J_{I}}\frac{\sqrt{n}\hat{\mu}_{j}}{\hat{ \sigma}_{j}},\max_{s=p+1,\dots ,k}\frac{\sqrt{n}\vert \hat{\mu} _{s}\vert }{\hat{\sigma}_{s}}} >c_{n}^{SN}(\alpha ,|J_{I}|)} + P( J_{I}\not\subseteq \hat{J}_{L}) +P( \cup _{j\in J_{I}^{c}}\{\hat{\mu}_{j}>0\}) \\
&\leq \alpha+\left\{ 
\begin{array}{c}
\alpha 2^{1+\delta }KCn^{-c+( 1-\delta ) /2}[ ( \ln ( 2k-p) ) ^{-( 2+\delta ) /2}+{2} ^{1/2}( 1-( \ln \alpha ) /( \ln ( 2k-p) ) ) ^{(2+\delta )/2}]+ \\ 
4\exp [ \ln (2k-p)( 1-n^{(2c+\delta -1)/(2+\delta)}/( 2C^{2/(2+\delta)}) ) ][1+K((Cn^{-c+(1-\delta)/2})^{1/(2+\delta )}+1)^{2+\delta }]+2\tilde{K}n^{-c}
\end{array}
\right\} , 
\end{align*}
where the first inequality follows from Eq.\ \eqref{eq:step2_1} and the second inequality follows from step 1 and Lemmas \ref{lem:SNSize} and \ref{lem:LassoNoOverFit}. 
\end{proof}


\subsection{Results for the bootstrap approximation}
\begin{lemma} \label{lem:BootSize}
Assume Assumptions \ref{ass:Basic}, \ref{ass:Rates3}, $\alpha \in (0,0.5)$, and that $H_{0}$ holds. For any non-stochastic set $L\subseteq \{1,\dots ,p\}$, define
\[
T_{n}(L) ~\equiv~ \max \cbr[2]{ \max_{j\in L}{\sqrt{n}\hat{\mu}_{j}}/{\hat{\sigma} _{j}},\max_{s=p+1,\dots ,k}{\sqrt{n}\vert \hat{\mu} _{s}\vert }/{\hat{\sigma}_{s}}},
\]
and let $c_{n}^{B}(L,\alpha)$ with $B\in \{MB,EB\}$ denote the conditional $(1-\alpha )$-quantile based on the bootstrap.
Then,
\[
P(T_{n}(L)>c_{n}^{B}(L,\alpha)) ~\leq~ \alpha +\tilde{C}n^{- \tilde{c}},
\]
where $\tilde{c},\tilde{C}>0$ are positive constants that only depend on the constants $c,C$ in Assumption \ref{ass:Rates3}. Furthermore, if $\mu={\bf 0}_{|L|}$, then
\[
\vert P( T_{n}(L)>c_{n}^{B}(L,\alpha)) -\alpha \vert ~\leq~ \tilde{C}n^{-\tilde{c}}.
\]
\end{lemma}
\begin{proof}
In the absence of moment equalities, this results follow from replacing $\{1,\dots ,p\}$ with $L$ in \citetalias{chernozhukov/chetverikov/kato:2014c} (proof of Theorem 4.3). As we show next, our proof follows from redefining the set of moment inequalities by adding the moment equalities as two sets of inequalities with reversed sign.

Define $A = A(L)\equiv  L\cup \{p+1,\dots ,k\}\cup \{k+1,\dots ,2k-p\}$ with $|A|=|L|+2(k-p)$ and for any $i=1,\dots ,n$, define the following $|A|$-dimensional auxiliary data vector
\[
X_{i}^{E}~\equiv~ \cbr[1]{ \{X_{ij}\}_{j\in L}^{\prime },\{X_{is}\}_{s=p+1,\dots ,k}^{\prime },\{-X_{is}\}_{s=p+1,\dots ,k}^{\prime }} ^{\prime }.
\]
Based on these definitions, we modify all expressions analogously, e.g.,
\begin{eqnarray*}
\mu ^{E} &=&\{ \{\mu _{j}\}_{j\in L}^{\prime },\{\mu _{s}\}_{s=p+1,\dots ,k}^{\prime },\{-\mu _{s}\}_{s=p+1,\dots ,k}^{\prime }\} ^{\prime }, \\
\sigma ^{E} &=&\{ \{\sigma _{j}\}_{j\in L}^{\prime },\{\sigma _{s}\}_{s=p+1,\dots ,k}^{\prime },\{\sigma _{s}\}_{s=p+1,\dots ,k}^{\prime }\} ^{\prime },
\end{eqnarray*}
and notice that $H_{0}$ can be equivalently re-written as $\mu ^{E}\leq \mathbf{0}_{|A|}$.

In the new notation, the test statistic is re-written as $T_{n}(L)=\max_{j\in A} {\sqrt{n}\hat{ \mu}_{j}^{E}}/{\hat{\sigma}_{j}^{E}}$, and the critical values can re-written analogously. In particular, the MB and EB test statistics are respectively defined as follows:
\begin{eqnarray*}
	W_{n}^{MB}(L)&=&\max_{j\in A} {\sqrt{n}\sum_{i=1}^{n}  \epsilon _{i}( X_{ij}^{E}-\hat{\mu}_{j}^{E}) }/{\hat{\sigma }_{j}^{E}},\\
	W_{n}^{EB}(L)&=&\max_{j\in A} {\sqrt{n}\sum_{i=1}^{n}  ( X_{ij}^{*,E}-\hat{\mu}_{j}^{E}) }/{\hat{\sigma }_{j}^{E}}.
\end{eqnarray*}
Given this setup, the result follows immediately from \citetalias{chernozhukov/chetverikov/kato:2014c} (Theorem 4.3).
\end{proof}

\begin{proof}[Proof of Theorem \ref{thm:B1Scorollary}.]
	This result follows from Lemma \ref{lem:BootSize} with $|L|=\{1,\dots ,p\}$.
\end{proof}


\begin{lemma}\label{lem:CvBootIncreasing}
For any $\alpha \in (0,0.5)$, $n\in \mathbb{N}$, $B\in \{MB,EB\}$, and $L_{1}\subseteq L_{2}\subseteq \{1,\dots ,p\}$,
\[
c_{n}^{B}(L_{1},\alpha)~\leq~ c_{n}^{B}(L_{2},\alpha ).
\]
Furthermore, under the above assumptions, $P(c_{n}^{B}(L_{1},\alpha)\geq 0)~\geq~ 1-\check{C}n^{-\check{c}}$, where $\check{c},\check{C}$ are universal constants.
\end{lemma}
\begin{proof}
By definition, $L_{1}\subseteq L_{2}$ implies that $W_{n}^{B}(L_{1})\leq W_{n}^{B}(L_{2})$ which, in turn, implies $c_{n}^{B}(L_{1},\alpha)\leq c_{n}^{B}(L_{2},\alpha)$. 

We now turn to the second result. If the model has at least one moment equality, then $W_n^B(L_{1})\geq 0$ and so $c_{n}^{B}(\alpha ,L_{1})\geq 0$. If the model has no moment equalities, then we consider a different argument depending on the type of bootstrap procedure being implemented.

First, consider MB. Conditionally on the sample, $W_{n}^{MB}(L_{1})=\max_{j\in L}{(1/\sqrt{n})\sum_{i=1}^{n} \epsilon _{i}\left( X_{ij}-\hat{\mu}_{j}\right)  }/{\hat{\sigma}_{j}}$ is the maximum of $L_{1}$ zero mean Gaussian random variables. Thus, $\alpha \in (0,0.5)$ implies that $c_{n}^{MB}(\alpha ,L_{1})\geq 0$. 

Second, consider EM. Let $c_{0}(L_{1},\alpha)$ denote the $(1-\alpha)$-quantile of $ \max_{j \in L_{1}}Y_{j}$ with $\{Y_{j}\}_{j \in L_{1}}\sim N(\mathbf{0},E [ \tilde{Z}\tilde{Z}'] )$ with $\tilde{Z} = \{Z_j\}_{j \in L_{1}}$ and $Z$ as in Assumption \ref{ass:Basic}. We then apply \citet[Eq.\ (86)]{chernozhukov/chetverikov/kato:2014c_supp} to our hypothetical model with the moment inequalities indexed by $L_{1}$, which yields
\begin{equation}
	P\del[1]{c_{n}^{EB}(L_{1},\alpha)\geq c_0(L_{1},\alpha +\gamma_{n})}~\geq~ 1-\check{C}n^{-\check{c}},\label{eq:86inCCK}
\end{equation}
where $\gamma_{n} \equiv \zeta_{n2}+\nu_{n}+8\zeta_{n1}\sqrt{\log p} \in (0, 2\check{C}n^{-\check{c}})$, for sequences $\{(\zeta_{n1},\zeta_{n2},\nu_{n})\}_{n\ge 1}$ and universal positive constants $(\check{c},\check{C})$, all specified in \cite{chernozhukov/chetverikov/kato:2014c_supp}. Since $\alpha<0.5$ and $\gamma_{n}<2\check{C}n^{-\check{c}}$, it follows that for all $n$ sufficiently large, $\alpha + \gamma_{n}<0.5$ and so $c_0(\alpha + \gamma_{n},L_{1})>0$. The desired result follows from combining this with Eq.\ \eqref{eq:86inCCK}.
\end{proof}


\begin{proof}[Proof of Theorem \ref{thm:Lasso2SBootstrap}]
	This proof follows similar steps than \citetalias{chernozhukov/chetverikov/kato:2014c} (Proof of Theorem 4.4). Define the following sequence of sets:
	\begin{equation*}
		J_{I} ~\equiv~ \{j=1,\dots ,p:\mu _{j}/\sigma _{j}\geq -3\lambda _{n}/4 \}
	\end{equation*}
	We divide the proof into two steps. The first step is exactly as in the proof of Theorem \ref{thm:Lasso2SSize} so it is omitted.

\underline{Step 2.} Define $T_{n}(J_{I})$ as in Lemma \ref{lem:BootSize} and consider the following derivation.
\begin{align}
&\{T_{n}>c_{n}^{B}(\hat{J}_L,\alpha)\}~\cap~  \{J_{I}\subseteq \hat{J}_{L}\}~\cap~ \{\cap _{j\in J_{I}^{c}}\{\hat{\mu}_{j}\leq 0\}\}~\cap~\{c_{n}^{B}(J_I,\alpha)\geq 0\}\notag\\
&\subseteq
\{T_{n}>c_{n}^{B}(J_{I},\alpha)\}~\cap~\{\cap _{j\in J_{I}^{c}}\{\hat{\mu}_{j}\leq 0\}\}~\cap~\{c_{n}^{B}(\alpha ,J_I)\geq 0\}\notag\\
&\subseteq 
\{T_n(J_I)>c_{n}^{B}(J_{I},\alpha)\} , \label{eq:step2_2}
\end{align}
where the first inclusion follows from Lemma \ref{lem:CvBootIncreasing}, and the second inclusion follows from the fact that ${\cap _{j\in J_{I}^{c}}\{\hat{\mu}_{j}\leq 0}\}$ and $\{T_{n}>c_{n}^{B}(J_{I},\alpha)\geq 0\}$ implies that $\{T_n(J_I)>c_{n}^{B}(J_{I},\alpha)\}$. Thus,
\begin{align}
&P\del[1]{T_{n}>c_{n}^{B,L}(\alpha)} = P\del[1]{T_{n}>c_{n}^{B}(\hat{J}_{L},\alpha)} \notag\\
 &=\cbr[4]{
\begin{array}{c}
P\del[1]{ \{ T_{n}>c_{n}^{B}(\hat{J}_{L},\alpha)\} \cap \{ \{ J_{I}\subseteq \hat{J}_{L}\} \cap \{ \cap _{j\in J_{I}^{c}}\{\hat{\mu}_{j}\leq 0\}\} \cap\cbr[0]{c_{n}^{B}(\alpha ,J_I)\geq 0}\} } + \\
P\del[1]{ \{ T_{n}>c_{n}^{B}(\hat{J}_{L},\alpha)\} \cap \{ \{ J_{I}\not\subseteq \hat{J}_{L}\} \cup \{ \cup _{j\in J_{I}^{c}}\{\hat{\mu}_{j}>0\}\} \cup\cbr[0]{c_{n}^{B}(\alpha ,J_I)< 0}\} }
\end{array}
}   \notag \\
&\leq P( T_{n}(J_{I})>c_{n}^{B}(J_{I},\alpha)) +P( J_{I}\not\subseteq \hat{J}_{L}) +P( \cup _{j\in J_{I}^{c}}\{\hat{ \mu}_{j}>0\})+P(c_{n}^{B}(\alpha ,J_I)< 0) \nonumber \\
&\leq \alpha + \left\{
\begin{array}{c}
\check{C}n^{-\check{c}}+ \tilde{C} n^{-\tilde{c}} +2\tilde{K}n^{-c}+\\
4\exp [ \ln (2k-p)( 1-n^{(2c+\delta -1)/(2+\delta)}/( 2C^{2/(2+\delta)}) ) ] [1+K( ( Cn^{-c+( 1-\delta ) /2}) ^{1/( 2+\delta ) }+1) ^{2+\delta }]
\end{array}\right\}\to \alpha,  \label{eq:Derivation2SMBLasso}
\end{align}
where the first inequality follows from Eq.\ \eqref{eq:step2_2} and the second inequality follows from the step 1 and Lemmas \ref{lem:LassoNoOverFit}, \ref{lem:BootSize}, and \ref{lem:CvBootIncreasing}. Finally, we note that the convergence in the last line holds uniformly in the manner required by the result.

We next turn to the second part of the result. By the case under consideration, $\mu={\bf 0}_{p}$ and so $J_{I} = \{ 1, \dots, p \}$. Thus, in this case, $\{J_{I} \subseteq \hat{J}_{L}\} = \{ \hat{J}_{L} = J_{I} = \{1,\dots,p\}\}$. By this and Lemma \ref{lem:LassoNoOverFit}, it follows that
\begin{align}
	&P\del[1]{ \hat{J}_{L} = J_{I} = \{1,\dots,p\}} \notag\\
	&\geq 1 -2\exp [ \ln (2k-p)( 1-n^{(2c+\delta -1)/(2+\delta)}/( 2C^{2/(2+\delta)}) ) ] [1+K( ( Cn^{-c+( 1-\delta ) /2}) ^{1/( 2+\delta ) }+1) ^{2+\delta }] -\tilde{K}n^{-c}.
	\label{eq:mehmet1}
\end{align}
In turn, consider the following derivation.
\begin{align}
P ( T_n > c_n^{B,L} (\alpha) ) =
P ( T_n > c_n^{B,1S} (\alpha) )\geq
 P ( T_n > c_n^{B,1S} ( \alpha) ) - P ( \{\hat{J}_{L} = J_{I} = \{1,\dots,p\}\}^{c}) \geq
 \alpha - 2\tilde{C} n^{- \tilde{c}},\label{eq:Derivation2SMBLasso2}
\end{align}
where the first equality follows from the fact that $\{\hat{J}_{L} = J_{I} = \{1,\dots,p\} \}$ implies that $c_n^{B,1S}  (\alpha) = c_n^{B} ({J}_{I},\alpha) = c_n^{B} (\hat{J}_{L},\alpha) = c_n^{B,L} (\alpha)$, and the last inequality follows from the second result in Theorem \ref{thm:B1Scorollary} and Eqs.\ \eqref{eq:RestrictionOnParams} and \eqref{eq:mehmet1}. By combining Eqs.\ \eqref{eq:RestrictionOnParams}, \eqref{eq:Derivation2SMBLasso}, and \eqref{eq:Derivation2SMBLasso2}, the desired result follows.
\end{proof}

\subsection{Results for power comparison}

\begin{proof}[Proof of Theorem  \ref{thm:powercompSN}] This result has several parts. 
	
	\underline{Part 1.} First, note that Eq.\ \eqref{eq:JComparisonSN} implies Eq.\ \eqref{eq:PowerAdvSN2}. Second, note that the arguments in the main text show that Eq.\ \eqref{eq:PowerAdvSN2} occurs if and only if $ c_{n}^{SN,L}(\theta,\alpha)~\leq~ c_{n}^{SN,2S}(\alpha)$ occurs. Finally, note that $ c_{n}^{SN,L}(\theta,\alpha)~\leq~ c_{n}^{SN,2S}(\alpha)$ implies Eq.\ \eqref{eq:PowerComparisonSN}.
	
	\underline{Part 2.} By definition and Lemma \ref{lem:LassoClosedForm},
\begin{align*}
\hat{J}_{SN} &=\{j=1,\dots,p:\hat{\mu}_{j}/\hat{\sigma}_{j}\geq -2c_{n}^{SN,1S}(\beta _{n})/\sqrt{n}\} \\
\hat{J}_{L} &=\{j=1,\dots,p:\hat{\mu}_{j,L}/\hat{\sigma}_{j}\geq -\lambda _{n}\}=\{j=1,\dots ,p:\hat{\mu}_{j}/\hat{\sigma}_{j}\geq -\lambda _{n}3/2\}.
\end{align*}
To complete the proof, it suffices to show that $\hat{J}_{L}\not\subseteq \hat{J}_{SN}$ implies that Eq.\ \eqref{eq:highlevel} does not hold. Suppose $\hat{J}_{L}\not\subseteq \hat{J}_{SN}$ holds. By the previous display, $ \exists j=1,\dots ,p$ s.t.\ $j\in \hat{J}_{L}\cap \hat{J}_{SN}^{c}$, i.e., $-2c_{n}^{SN,1S}(\beta _{n})/\sqrt{n}>{\hat{\mu}_{j}}/{\hat{\sigma}_{j}}\geq - \lambda _{n}3/2$, which implies that $c_{n}^{SN,1S}(\beta _{n})/\sqrt{n}< \lambda _{n}3/4$. By combining this with Eqs.\ \eqref{eq:LambdaConcrete} and \eqref{eq:CVSN_oracle}, we contradict Eq.\ \eqref{eq:highlevel}, as desired.

\end{proof}

\begin{proof}[Proof of Theorem  \ref{thm:power_compB}]
	This result has several parts.

	\underline{Part 1:} This part of the proof is analogous to the the proof of part 1 in Theorem \ref{thm:powercompSN}.
	
	\underline{Part 2:} By definition and Lemma \ref{lem:LassoClosedForm},
		\begin{align*}
		\hat{J}_{B} &~=~\{ j=1,\dots ,p:\hat{\mu}_{j}/\hat{\sigma}_{j}\geq -2c_{n}^{B,1S}( \beta _{n}) /\sqrt{n}\} , \\
		\hat{J}_{L} &~=~\{ j=1,\dots ,p:\hat{\mu}_{j,L}/\hat{\sigma}_{j}\geq -\lambda _{n}\}~=~\{j=1,\dots ,p:\hat{\mu}_{j}/\hat{\sigma}_{j}\geq -\lambda _{n}3/2\}.
		\end{align*}
	It is convenient to divide the remainder of this proof into three steps.
		
	\underline{Part 2.1:} Show that
\begin{equation}
	\{ c_{n}^{B}( \beta _{n}) 4/3\geq \lambda _{n}\sqrt{n} \}  ~\subseteq~ \{ \hat{J}_{L}\subseteq \hat{J}_{B}\}.
	\label{eq:quantiles1}
\end{equation}
	To show Eq.\ \eqref{eq:quantiles1}, suppose that $\{\hat{J}_{L}\subseteq \hat{J}_{B}\}$ does not occur, i.e., $ \exists j\in \hat{J}_{L}\cap \hat{J}_{B}^{c}$ s.t.\ $-2c_{n}^{B}( \beta _{n}) /\sqrt{n}>\hat{\mu}_{j}/\hat{\sigma} _{j}\geq -\lambda _{n}3/2$. From here, we conclude that $\{c_{n}^{B}( \beta _{n}) 4/3< \lambda _{n}\sqrt{n}\}$, as desired.
	
	\underline{Part 2.2:} Under either one of our sufficient conditions, show that
	\begin{equation}
		\{ c_{n}^{B}( \beta _{n}) \geq c_{0}(3\beta _{n})\} ~\subseteq~ \{ c_{n}^{B}( \beta _{n}) 4/3\geq \lambda _{n}\sqrt{n}\},
		\label{eq:quantiles2}
	\end{equation}
	where $c_{0}(3\beta _{n})$ denotes the $(1-3\beta _{n})$-quantile of $ \max_{1\leq j\leq p}Y_{j}$ with $(Y_{1},\ldots ,Y_{p})\sim N(\mathbf{0},E [ ZZ'] )$ with $Z$ as in Assumption \ref{ass:Basic}.
	
	To this end, we consider two strategies. The first strategy relies on Eq.\ \eqref{eq:2Ssuff} and the second strategy relies on Eq.\ \eqref{eq:2Ssuff2}. We begin with the first strategy. First, note that $ \max_{1\leq j\leq p}Y_{j}\geq Y_1$ implies that
	\begin{equation}
	c_{0}(3\beta _{n})~\geq~ \Phi ^{-1}(1-3\beta _{n}).  \label{eq:P2_B}
	\end{equation}
	Second, note that Eqs.\ \eqref{eq:LambdaConcrete} and \eqref{eq:2Ssuff} imply that
	\begin{align}
	\Phi ^{-1}( 1-3\beta _{n}) ~\geq~ ( 1+3\varepsilon /4) ( M_{n,2+\delta }^{2}n^{-\delta
/( 2+\delta ) }-n^{-1}) ^{-1/2} = \sqrt{n} \lambda _{n}3/4
\label{eq:P2_c}
	\end{align}
	By combining Eqs.\ \eqref{eq:P2_B} and \eqref{eq:P2_c}, the implication in Eq.\ \eqref{eq:quantiles2} follows.
	
	We now develop the second strategy. The Borell-Cirelson-Sudakov inequality (see, e.g., \citet[Theorem 5.8]{boucheron/lugosi/massart:2013}) implies that for $x\geq 0$,
	\begin{equation}
	P\del[2]{\max_{1\leq j\leq p}Y_{j}~\leq~ E[ \max_{1\leq j\leq p}Y_{j} ] -x} ~\leq~ \exp({-x^{2}/2}),
	\label{eq:Piece0}
	\end{equation}
	where we used that the diagonal $E[ZZ']$ is a vector of ones. Equating the RHS of Eq.\ \eqref{eq:Piece0} to $(1-3\beta _{n})$ yields $x=\sqrt{ 2\log (1/(1-3\beta _{n}))}$ such that
	\begin{equation}
	c_{0}(3\beta _{n})~\geq~ E[ \max_{1\leq j\leq p}Y_{j}] -\sqrt{2\log (1/[1-3\beta _{n}])}. \label{eq:Piece1}
	\end{equation}
	We now provide a lower bound for the first term on the RHS of Eq.\ \eqref{eq:Piece1}. Consider the following derivation:
	\begin{eqnarray}
	E[ \max_{1\leq j\leq p}Y_{j}] ~\geq~ \min_{i\neq j}\sqrt{ E(Y_{i}-Y_{j})^{2}\log (p)/2} ~\geq ~\sqrt{2(1-\rho )\log (p)/2},\label{eq:Piece2}
	\end{eqnarray}
	where the first inequality follows from Sudakov's minorization inequality (see, e.g., \citet[Theorem 13.4]{boucheron/lugosi/massart:2013}), while the second inequality follows from $E[ZZ']$ having  diagonal elements equal to one and maximal absolute correlation less than $\rho $. Eqs.\ \eqref{eq:Piece1} and \eqref{eq:Piece2} imply that
	\begin{equation}
	c_{0}(3\beta _{n})~\geq~ \sqrt{(1-\rho )\log (p)/2}-\sqrt{2\log (1/(1-3\beta _{n}))}. \label{eq:P2_C}
	\end{equation}
	Combine this with Eqs.\ \eqref{eq:LambdaConcrete} and \eqref{eq:2Ssuff2} to conclude that
	\begin{align*}
	c_0(3\beta_n)\geq\sqrt{( 1-\rho ( \theta ) ) \ln (p)/2}-\sqrt{2\ln
( 1/( 1-3\beta _{n}) ) }\geq ( 1+3\varepsilon
/4) ( M_{n,2+\delta }^{2}n^{-\delta /( 2+\delta )
}-n^{-1}) ^{-1/2}  = \sqrt{n} \lambda _{n}3/4,
	\end{align*}
	and the implication in Eq.\ \eqref{eq:quantiles2} follows.
	
	\underline{Part 2.3:} Use previous results to complete the argument. By combining Eqs.\ \eqref{eq:quantiles1}, \eqref{eq:quantiles2}, and the fact that the function $c_{0}(\cdot )$ is decreasing, we deduce that for any $\mu _{n}\leq 3\beta _{n}$,
	\[
	\{c_{n}^{B}( \beta _{n}) \geq c_{0}(\mu _{n})\} ~\subseteq~\{ c_{n}^{B}( \beta _{n}) \geq c_{0}(3\beta _{n})\} ~\subseteq~ \{ c_{n}^{B}( \beta _{n}) 4/3\geq \lambda _{n}\sqrt{n}\} ~\subseteq~ \{ \hat{J}_{L}\subseteq \hat{J} _{B}\} .
	\]
	To complete the argument, it then suffices to show that for some $\mu _{n}\leq 3\beta _{n}$,
	\begin{equation*}
		P(c_{n}^{B}( \beta _{n}) \geq c_{0}(\mu _{n}) ) ~\geq~ 1- Cn^{-c}.
	\end{equation*}
	To establish this result, we use \citet[Eq.\ (86)]{chernozhukov/chetverikov/kato:2014c_supp} evaluated at $\alpha =\beta _{n}$, $\nu _{n}=Cn^{-c}$, and $(\zeta _{n2},\zeta _{n1})$ s.t.\ $\zeta _{n2}+8\zeta _{n1}\sqrt{\ln p}\leq Cn^{-c}$. Under our assumptions, note that $\mu _{n}\equiv \beta _{n}+\zeta _{n2}+v_{n}+8\zeta _{n1}\sqrt{\log p}\leq \beta _{n}+2Cn^{-c}\leq 3\beta _{n}$, as required.

	\underline{Part 3:} Consider the following argument.
	\begin{align*}
		P(T_{n}\geq c_{n}^{B,2S}( \alpha ) ) &=P(
		T_{n}\geq c_{n}^{B,2S}( \alpha ) \cap \hat{J}_{L}\subseteq \hat{J}_{B}) +P(T_{n}\geq c_{n}^{B,2S}( \alpha )\cap \hat{J}_{L}\not\subseteq \hat{J}_{B})\\
		&\leq  P(T_{n}\geq c_{n}^{B,L}( \alpha )) +P(\hat{J}_{L}\not\subseteq \hat{J}_{B})  \\
		&\leq P(T_{n}\geq c_{n}^{B,L}( \alpha )) + {C}n^{-{c}},
	\end{align*}
	where the first inequality uses Part 1, and the second inequality uses that \eqref{eq:2Ssuff} and \eqref{eq:2Ssuff2} both individually imply Eq.\ \eqref{eq:JComparisonBstock}.
\end{proof}

\end{small}

\bibliography{BIBLIOGRAPHY}

\end{document}